\documentclass[12pt, reqno]{amsart}
\usepackage{amsmath,amssymb,amsthm,amsxtra, setspace}
\usepackage{amssymb}
\usepackage{stmaryrd}
\usepackage{alltt}
\usepackage{graphicx,type1cm,xcolor}
\usepackage{mathtools, accents}
\usepackage{mathrsfs}
\usepackage{hyperref}
\usepackage{alltt}
\usepackage{fouridx}
\usepackage{dsfont}
\usepackage{dsfont}
\usepackage{cancel}
\usepackage[mathcal]{euscript}
\usepackage[margin=1in]{geometry}
\usepackage{shadethm}
\usepackage{float}

\allowdisplaybreaks

\newcommand{\eps}{\varepsilon}

\numberwithin{equation}{section}

\newcommand{\what}{\widehat}
\newcommand{\wtilde}{\widetilde}
\newshadetheorem{Theorem}{Theorem}
\newshadetheorem{Proposition}{Proposition}
\newshadetheorem{Lemma}{Lemma}
\newshadetheorem{Assumption}{Assumption}
\newshadetheorem{Corollary}{Corollary}
\newshadetheorem{Definition}{Definition}
\newshadetheorem{Example}{Example}
\newshadetheorem{Remark}{Remark}
\newshadetheorem{Note}{Note}
\newshadetheorem{Hypothesis}{Hypothesis}
\newshadetheorem{Explanation}{Explanation}
\newtheorem{theorem}{Theorem}[section]
\newtheorem{lemma}[theorem]{Lemma}
\newtheorem{proposition}[theorem]{Proposition}
\newtheorem{assumption}[theorem]{Assumption}
\newtheorem{corollary}[theorem]{Corollary}
\newtheorem{definition}[theorem]{Definition}

\newtheorem{remark}[theorem]{Remark}

\newcommand{\embed}{\hookrightarrow}
\newcommand{\K}{\mathcal{K}}

\def\Ls{\mathcal{L}}
\def\bL{\mathscr{L}}

\def\Fn{\mathscr{F}}
\def\Gn{\mathscr{G}}
\def\E{\mathbb{E}}

\def\Sb{\mathbb{S}}

\def\C{\mathcal{C}}

\def\0{\boldsymbol{0}}

\def\N{\mathbb{N}}

\def\bM{\mathcal{M}}

\def\Pr{\mathbb{P}}

\def\bO{\mathcal{O}}
\def\Vp{{\mathrm{V}_p}}

\def\Fb{\mathbb{F}}

\newcommand{\norm}[1]{\left\Vert#1\right\Vert}
\newcommand{\abs}[1]{\left\vert#1\right\vert}

\newcommand{\R}{\mathbb{R}}

\let\originalleft\left
\let\originalright\right
\renewcommand{\left}{\mathopen{}\mathclose\bgroup\originalleft}
\renewcommand{\right}{\aftergroup\egroup\originalright}
\newcommand{\Nn}{\mathcal{N}}

\newcommand{\ip}[2]{\fourIdx{}{0}{}{\!x}{\mathcal{ X}}}

\newcommand\dela[1]{}

\hypersetup{colorlinks=true,%
	citecolor=red,%
	filecolor=blue,%
	linkcolor=blue,%
}
\usepackage{graphicx}
\usepackage[utf8]{inputenc}
\usepackage[T1]{fontenc}
\mathtoolsset{showonlyrefs}
\usepackage{nomencl}
\makenomenclature
\usepackage{amssymb}
\usepackage{enumitem}
\usepackage[symbol]{footmisc}
\usepackage{centernot}
\usepackage{ragged2e}
\justifying
\usepackage{orcidlink}

\usepackage{todonotes}

\usepackage{xpatch}
\makeatletter   
\xpatchcmd{\@tocline}
{\hfil\hbox to\@pnumwidth{\@tocpagenum{#7}}\par}
{\ifnum#1<0\hfill\else\dotfill\fi\hbox to\@pnumwidth{\@tocpagenum{#7}}\par}
{}{}
\makeatother    

\makeatletter
\def\l@subsection{\@tocline{2}{0pt}{4pc}{6pc}{}}
\def\l@subsubsection{\@tocline{3}{0pt}{8pc}{8pc}{}}
\makeatother

\makeatletter
\def\l@section{\@tocline{1}{12pt}{0pt}{}{\bfseries}}
\makeatother

\title[Existence and uniqueness of a stochastic constrained heat equation]{Existence and uniqueness results of a stochastic nonlinear heat equation with a constraint of codimension one}

\begin{document}
	\maketitle
	\begin{center}
		\author{Ashish Bawalia\footnote[4]{Department of Mathematics, Indian Institute of Technology Roorkee-IIT Roorkee, Haridwar Highway, Roorkee, Uttarakhand 247667, India.}\orcidlink{0009-0002-9141-2766}, Zdzis\l{}aw Brze\'zniak\footnotemark[2]$^\ast$\orcidlink{0000-0001-8731-6523} and Manil T. Mohan\footnotemark[4]\orcidlink{0000-0003-3197-1136}}
		\footnotetext[2]{Department of Mathematics, The University of York, Heslington, York YO10 5DD, United Kingdom.\\
			\textit{e-mail:} Ashish Bawalia: \email{ashish1@ma.iitr.ac.in; ashish1441chd@gmail.com}.\\
			\textit{e-mail:} Zdzis\l{}aw Brze\'zniak: \email{zdzislaw.brzezniak@york.ac.uk}.\\
			\textit{e-mail:} Manil T. Mohan: \email{maniltmohan@ma.iitr.ac.in; maniltmohan@gmail.com}.\\
			$\hspace{2mm} ^\ast$Corresponding author. 
			}
	\end{center}
	\begin{abstract}
		In this work, we investigate the well-posedness of  a stochastic heat equation with an arbitrary (but polynomial) nonlinearity in any dimension $d\geq 1$ perturbed by a multiplicative white noise in the Stratonovich form, subject to an $L^2-$norm constraint on the solution. In bounded smooth domains, we establish the existence of a martingale solution taking values in $H_0^1 \cap L^p$ for arbitrary $2 \le p < \infty$, using a modified Faedo-Galerkin scheme. 
		By utilizing a sequence of self-adjoint operators which are bounded in $L^p$ for any $2 \le p < \infty$, we provide a novel proof of an It\^o formula for the $L^p-$norm of the solution. Together with pathwise uniqueness of the martingale solution, the Yamada-Watanabe result then yields the existence of a strong solution and uniqueness in law.
	\end{abstract}
	\begin{center}
		\tableofcontents
	\end{center}


	\section{Introduction}\label{sec-intro}
	Deterministic constrained partial differential equations have attracted considerable attention over the past two decades and have been studied extensively from classical heat equations, see \cite{PR-06, LC+FL-09}, to nonlinear heat equations, see \cite{LM+LC-09, LM+LC-13, JH-23, ZB+JH-24, PA+PC+BS-24, BS-25, AB+ZB+MTM-25+, AB+ZB+MTM+PR-25+, ZB+JH-26}, and the Navier-Stokes equations, see \cite{ZB+GD+MM-18}, among others.
	
	One of the first article in the stochastic setting of the constrained problems was presented by the second-named author and Hussain \cite{ZB+JH-20}. For initial data in $H_0^1\cap L^p$, they established the existence of a unique global mild solution to a constrained stochastic nonlinear heat equation driven by a multiplicative Gaussian noise of Stratonovich-type on a two-dimensional bounded smooth domain, see \eqref{eqn-main-problem} below, using a fixed-point argument. However, in this manuscript, we generalize their result to arbitrary spacial dimensions $d\geq 1$ and any nonlinearity exponent $2\le p<\infty$. In fact, we prove the existence of a more general class of solutions, known as \textit{martingale solutions}, see Definition \ref{Def-Martingale}, by exploiting a modified Faedo-Galerkin scheme and establishing an \textit{$L^p-$It\^o formula} for the entire regime $2 \le p < \infty$. Futhermore, we also show that these solution are \textit{pathwise unique}. Finally, an application of the Yamada--Watanabe Theorem, see Theorem \ref{Thm-YW}, yields the existence of a \textit{strong solution} that are \textit{unique in law}. This work can also be regarded as a stochastic counterpart of the first part of our recent deterministic study \cite{AB+ZB+MTM-25+}, for more details see Subsection \ref{Subsec-main}.

	Let $\bO \subset \R^d,$ $d\geq 1$, be a bounded domain with $C^2-$boundary, denoted by $\partial \bO$ and let $(\Omega,\Fn,\{\Fn_t\}_{t\in[0,T]},\Pr)$ be a filtered probability space. For a fixed $0<T<\infty$, we are concerned with a stochastic partial differential equation (SPDE), see \cite{ZB+JH-20}, i.e.,  a stochastic constrained nonlinear heat equation, see \cite{AB+ZB+MTM-25+, ZB+JH-24} for the formulation:
	\begin{align}\label{eqn-main-problem}
		\left\{\begin{aligned}
			du(t)& =\big(\Delta u(t) -|u(t)|^{p-2}u(t) + \big( \norm{\nabla u(t)}_{L^2(\bO)}^2 + \norm{u(t)}_{L^p(\bO)}^p \big) u(t)\big)dt\\
			& \quad + \sum_{i=1}^M\Nn_i(u(t))\circ dW_i(t),\\
			u(0) & = u_0, \\
			u(t)|_{\partial\bO} & = 0,
		\end{aligned}\right.
	\end{align}
	where $u:[0,\infty) \times\bO\times\Omega \to \R$ with $u(t)=u(t,x,\omega)$, $2\leq p<\infty$, $M\le\infty$, $W_i$'s, for $1 \le i\le M$, are independent one-dimensional real-valued Brownian motions, $\circ$ means the stochastic integral is understood in the Stratonovich sense and $u_0$ denotes the initial data. Moreover, for fixed elements $f_1,\ldots, f_M$  from 
	\begin{equation}
		\Vp:= H_0^1(\bO)\cap L^p(\bO) \ \text{ with } \ \norm{\cdot}_{\Vp} := \norm{\cdot}_{H_0^1(\bO)} + \norm{\cdot}_{L^p(\bO)}, \label{Def-Vp}
	\end{equation}
	define the mapping $\Nn_i:\Vp\to \Vp$ by
	\begin{equation}\label{eqn-def-Nn}
		\Nn_i(u) := f_i- (f_i,u)u, \ \text{ for }\ 1 \le i \le M.
	\end{equation}
	Since Stratonovich forms are naturally considered due to the constraint provided by the $L^2-$unit sphere $\bM$, cf. \cite{ZB+JH-20, ZB+GD-21}, where 
	\begin{equation}\label{eqn-M}
		\bM:= \{ v \in L^2(\bO) : \|v\|_{L^2(\bO)} = 1\},
	\end{equation}
	we also consider a noise of Stratonovich type in the above SPDE \eqref{eqn-main-problem}. 
	
	\begin{remark}
		When $p > \frac{2d}{d-2}$, we need $\Vp$ with the sum norm, see \eqref{Def-Vp}. However, when $p \le \frac{2d}{d-2}$, the Sobolev-embedding Theorem asserts that $\Vp = H_0^1(\bO)$ with the $H_0^1(\bO)-$norm.
	\end{remark}
	\begin{remark}
		Let us emphasize that the right-hand side of \eqref{eqn-main-problem} lies on the tangent manifold corresponding to the $L^2-$unit sphere. However, our results also hold for other $L^2-$spheres, with the only modification occurring in the projected (or nonlocal) terms; namely, one has to divide these projected terms by the radius of the sphere. For details, see \cite[Remark 1.1]{AB+ZB+MTM-25+}.
	\end{remark}
	Using the It\^o-Stratonovich formula, see \cite[Section 2.1]{DG+DC-24}, we transform the Stratonovich noise \eqref{eqn-main-problem} into the following It\^o form:
	\begin{align*}
		\sum_{i=1}^M\Nn_i(u)\circ dW_i(t) =	\sum_{i=1}^M\Nn_i(u)dW_i(t) + \frac{1}{2}\sum_{i=1}^Md_u\Nn_i(\Nn_i(u))dt,
	\end{align*}
	where $d_u(\cdot)$ denotes the Fr\'echet derivative. Therefore, the equation \eqref{eqn-main-problem} can be represented in Itô form as follows:
	\begin{align}\label{eqn-main-prob-Ito}
		\left\{\begin{aligned}
			du(t) & =\big( \Delta u(t) -|u(t)|^{p-2}u(t) + \big( \norm{\nabla u(t)}_{L^2(\bO)}^2 + \norm{u(t)}_{L^p(\bO)}^p \big) u(t)\big)dt\\ 
			& \quad + \frac{1}{2}\sum_{i=1}^M\kappa_i(u(t))dt + \sum_{i=1}^M\Nn_i(u(t))dW_i(t),\\
			u(0) & = u_0, \\
			u(t)|_{\partial\bO} & = 0,
		\end{aligned}\right.
	\end{align}
	where the It\^o correction term is given by
	\begin{equation*}
		\kappa_i(u)=\sum_{i=1}^Md_u\Nn_i(\Nn_i(u)).
	\end{equation*}

	In \cite{ZB+GD-21}, second-named author with Dhariwal proved the existence and the pathwise uniqueness of the martingale solution for two-dimensional Navier-Stokes equations driven by a multiplicative Gaussian noise with constant coefficient. The existence of a strong solution is also established using the Yamada-Watanabe result.
	Recently, Hussain et al. \cite{JH+FA+AS-24+} studied the equation \eqref{eqn-main-prob-Ito} by the classical Faedo-Galerkin method and prove the existence of martingale solutions.
	Thereafter, the second author and Cerrai \cite{ZB+SC-25} proved the well-posedness of a class of stochastic second-order in time damped evolution equations in Hilbert spaces, subject to the constraint that the solution is invariant in the unit sphere. Further, they have shown that, in the small mass limit, the solution converges to the solution of a stochastic parabolic equation subject to the same constraint.
	Recently, Cerrai and Xie \cite{SC+MX-25} studied a stochastic damped wave equation constrained to the unit sphere in $L^2(0, L)$ via the Smoluchowski-Kramers approximation. More recently, the second-named author and Hussain \cite{ZB+JH-26} studied the large deviation principle for stochastic heat equations with constraints, under the same restrictions as in \cite{ZB+JH-20}.
	
	Motivated by the work of the second-named author and Dhariwal \cite{ZB+GD-21}, the main aim of this work is as follows:
	\textit{Given $u_0\in \Vp\cap\bM$, we aim to prove the existence and pathwise uniqueness of a strong solution $u$ to the problem \eqref{eqn-main-prob-Ito} such that $u(t) \in\bM$, for every $t\in [0, T]$, on a set of full probability.}

	\subsection{What is new?}
	In this manuscript, we investigate a generalized stochastic constrained system, namely a stochastic heat equation with nonlinearity $\abs{u}^{p-2}u$ for $2\leq p < \infty$, projected onto the tangent bundle $T_u\bM$, on arbitrary bounded domains $\bO\subset \R^d$, $d\ge1$, with $C^2-$boundary. We study the existence and the uniqueness of $\Vp-$valued solutions to the problem \eqref{eqn-main-prob-Ito}. More precisely, when initial data belongs to $\Vp \cap \bM$, we establish the following:
	\begin{enumerate}
		\item Using a modified Faedo-Galerkin approximation method \cite{FH-18,ZB+FH+LW-19,ZB+FH+UM-20,ZB+BF+MZ-24}, together with compactness arguments, we prove the existence of a $\Vp-$valued weak solution to the constrained system \eqref{eqn-main-prob-Ito}, see Proposition \ref{Prop-pre-martingale}. This extends the regimes considered in \cite{ZB+JH-20} and \cite{JH+FA+AS-24+}, to any spacial dimensions $d\ge1$ and nonlinearity exponents $p\in [2,\infty)$.
		\item The weak solutions are invariant in the $L^2(\bO)$ unit sphere, i.e., if the initial data is in $\bM,$ then, all its corresponding trajectories stay in $\bM$.
		\item For any $2\le p < \infty$, by employing a family of $L^p-$uniformly bounded self-adjoint operators $\{S_m\}_{m\in\N}$, see Proposition \ref{Prop-S_m}, we establish a novel $L^p-$It\^o formula is established, see Lemma \ref{Lem-Ito}, which is then used to prove the existence of martingale solutions to the SPDE \eqref{eqn-main-prob-Ito}, see Theorem \ref{Thm-main}.
		\item By the Schmalfuss argument \cite{BS-97}, the martingale solution to problem \eqref{eqn-main-prob-Ito} is pathwise unique, see Lemma \ref{Lem-unique}.
		\item The pathwise uniqueness of martingale solutions implies the existence of a strong solution and uniqueness in law, via the Yamada-Watanabe Theorem (Theorem \ref{Thm-YW}) in the form due to Ond\v{r}ej\'at \cite{MO-04}.
		\item In \cite{ZB+JH-20}, the parameter $M< \infty$ was assumed to be finite. In this note, we also treat the case $M=\infty$, which was not addressed previously.
	\end{enumerate}
	
	\begin{remark}\label{Rmk-Krylov}
		\noindent
		\begin{itemize}
			\item[(i)] We emphasize here that this work removes the upper-bound of the nonlinearity parameter, i.e., without the restriction $p \in [2,\frac{2d}{d-2}]$ for $d\ge 3$, arising from the Sobolev embedding $H_0^1(\bO)\embed L^p(\bO)$ that was assumed in previous works \cite{ZB+JH-20} (or \cite{JH+FA+AS-24+}).
			\item[(ii)] Let us mention here that the well-known $L^p-$It\^o formula of Krylov \cite[Lemma 5.1]{NVK-10}
			can only be applied here for the case $p=2$. Therefore, we provide a novel proof which does not rely on the regularization techniques.
		\end{itemize}
	\end{remark}
	
	Next, let us present the main results of this manuscript.
	
	\subsection{Main results}\label{Subsec-main}
	We begin by introducing the definitions of martingale and strong solutions for \eqref{eqn-main-prob-Ito}. We then state the main results, whose proofs are given in the subsequent sections. Let us choose and fix the nonlinearity parameter $p \in [2,\infty)$ and time $0< T < \infty$ throughout this article.
	
	\begin{definition}
		A \emph{stochastic basis} $(\Omega,\Fn,\Fb,\Pr)$ is a probability space equipped with the filtration $\Fb:=\{\Fn_t\}_{t\in[0,T]}$ of its $\sigma-$field $\Fn$.
	\end{definition}

	\begin{definition}\label{Def-Martingale}
		For each fixed $u_0\in \Vp\cap \bM$, we say that there exists a \emph{martingale solution} of \eqref{eqn-main-prob-Ito} if and only if there exist
		\begin{itemize}
			\item[$\boldsymbol{\ast}$] a stochastic basis $(\wtilde{\Omega},\wtilde{\Fn},\wtilde{\Fb},\wtilde{\Pr})$,
			\item[$\boldsymbol{\ast}$] an $\mathbb{R}^M-$valued $\wtilde{\Fb} - $Wiener process $\wtilde{W} = (W_1,\dots, W_M)$,
			\item[$\boldsymbol{\ast}$] and an $\wtilde{\Fb} - $progressively measurable process $u :[0, T] \times\wtilde{\Omega}\to D(A)$ with $\wtilde{\Pr} - $a.e. paths
			$$u(\cdot,\omega)\in C([0,T]; \Vp\cap\bM)\cap L^2(0,T;D(A))\cap L^{2p-2}(0,T;L^{2p-2}(\bO))$$ such that for all $t\in[0, T]$ and all $\psi\in L^2(\bO)$, $\wtilde{\Pr} - $a.s.
			\begin{align}
				\nonumber(u(t),\psi) & =(u_0,\psi) + \int_0^t\big( \Delta u(s) -|u(s)|^{p-2}u(s) + \big( \norm{\nabla u(s)}_{L^2(\bO)}^2 + \norm{u(s)}_{L^p(\bO)}^p \big) (u(s),\psi)ds\\ 
				& \quad + \frac{1}{2}\sum_{i=1}^M\int_0^t(\kappa_i(u(s)),\psi)ds+ \sum_{i=1}^M\int_0^t(\Nn_i(u(s)),u(s))dW_i(s). \label{eqn-strong-form}
			\end{align}
		\end{itemize}
	\end{definition}
	
	\begin{definition}\label{Def-Strong-Soln}
		Let $u_0\in \Vp\cap \bM$ be fixed. Then, we call that problem \eqref{eqn-main-prob-Ito}  has a \emph{strong solution} if and only if for every stochastic basis $(\Omega,\Fn,\Fb,\Pr)$  and every $\mathbb{R}^M-$valued $\Fb-$Wiener process $$W= (W_1, \dots, W_M),$$ there exists a $\Fb-$progressively measurable process $u :[0, T] \times\Omega \to D(A)$ with $\Pr-$a.e. paths
		$$u(\cdot,\omega)\in C([0,T];\Vp)\cap L^2(0,T;D(A))\cap L^{2p-2}(0,T;L^{2p-2}(\bO))$$ such that for all $t\in [0, T]$, $\Pr-$a.s.
		\begin{align}
			u(t)& =u_0+ \int_0^t\big( \Delta u(s) -|u(s)|^{p-2}u(s) + ( \norm{\nabla u(s)}_{L^2(\bO)}^2 + \norm{u(s)}_{L^p(\bO)}^p ) u(s)\big)ds\\
			& \quad + \frac{1}{2}\sum_{i=1}^M\int_0^t\kappa_i(u(s))ds+ \sum_{i=1}^M\int_0^t\Nn_i(u(s))dW_i(s),\label{eqn-strong-form-1}
		\end{align}
		in $L^2(\bO)$, i.e., 
		\eqref{eqn-strong-form} is satisfied.
	\end{definition}

	\begin{assumption}
		Let us now state the assumption needed on the series of the noise coefficients when $M = \infty$,
		\begin{equation}
			\sum_{i=1}^M\|f_i\|_{H_0^1(\bO)}, \sum_{i=1}^M\|f_i\|_{L^p(\bO)} <  \infty.
		\end{equation}
	\end{assumption}
	
	We begin by stating a result on the existence of a probabilistically weak solution, which can be regarded as a first step toward establishing the existence of a martingale solution in the sense of Definition \ref{Def-Martingale}.
	
	\begin{proposition}\label{Prop-pre-martingale}
		Let $u_0\in \Vp\cap \bM$ be fixed. Then, the problem \eqref{eqn-main-prob-Ito} admits a weak solution if and only if there exist
		\begin{itemize}
			\item[(i)] a stochastic basis $(\wtilde{\Omega},\wtilde{\Fn},\wtilde{\Fb},\wtilde{\Pr})$,
			\item[(ii)] an $\mathbb{R}^M-$valued $\wtilde{\Fb} - $Wiener process $\wtilde{W} = (W_1,\dots, W_M)$,
			\item[(iii)] and a $\wtilde{\Fb} - $progressively measurable process $u :[0, T] \times\wtilde{\Omega}\to D(A)$ with $\wtilde{\Pr} - $a.e. paths
			$$u(\cdot,\omega)\in C_w([0,T]; \Vp)\cap L^2(0,T;D(A))\cap L^{2p-2}(0,T;L^{2p-2}(\bO))$$ such that for all $t\in[0, T]$ and all $\psi\in L^2(\bO)$, $u(t)\in \bM$ and \eqref{eqn-strong-form} is satisfied $\wtilde{\Pr} - $a.s.
		\end{itemize}
	\end{proposition}
	Here $C_w([0,T];\Vp)$ means the Banach space $C([0,T];\Vp)$ equipped with the weak topology.
	Next, we state an important results concerning the existence of a martingale solution to the problem \eqref{eqn-main-prob-Ito}, i.e., the It\^o Lemma for $L^p-$norm.
	
	\begin{lemma}[{$L^p-$It\^o-formula}]\label{Lem-Ito}
		Let us choose and fix $u_0\in \Vp\cap \bM$. Let $(\wtilde{\Omega},\wtilde{\Fn},\wtilde{\Fb},\wtilde{\Pr},\wtilde{W},u)$ be the weak solution to the SPDE \eqref{eqn-main-prob-Ito}, guaranteed by Proposition \ref{Prop-pre-martingale}, in the sense of Definition \ref{Def-Martingale}. Then, for any $t\in [0,T]$, the process $u$ satisfies the following It\^o formula:
		\begin{align}
			& \|u(t)\|_{L^p(\bO)}^p+p(p-1)\int_0^t\||u(s)|^{\frac{p-2}{2}}\nabla u(s)\|_{L^2(\bO)}^2ds+p\int_0^t\|u(s)\|_{L^{2p-2}(\bO)}^{2p-2}ds\\
			& = \|u_0\|_{L^p(\bO)}^p + p\int_0^{t} \big( \norm{\nabla u(s)}_{L^2(\bO)}^2 + \norm{u(s)}_{L^p(\bO)}^p\big) \|u(s)\|_{L^p(\bO)}^pds\\
			& \quad + \frac{p}{2}\sum_{i=1}^M\int_0^{t}(|u(s)|^{p-2}u(s),\kappa_i(u(s)))ds+ \frac{p(p-1)}{2}\sum_{i=1}^M \int_0^{t}\||u(s)|^{\frac{p-2}{2}}\Nn_i(u(s))\|_{L^2(\bO)}^2ds\\
			& \quad +p\sum_{i=1}^M \int_0^{t}(|u(s)|^{p-2}u(s),\Nn_i(u(s)))dW_i(s),\ \  \wtilde{\Pr} - a.s.\label{eqn-Ito-L^p}
		\end{align}
	\end{lemma}
	
	By utilizing the above It\^o Lemma along with Proposition \ref{Prop-pre-martingale}, the next result shows that almost all the trajectories of the weak solution are almost everywhere equal to a $\Vp-$valued continuous function defined on $[0, T]$.
	
	\begin{lemma}\label{Lem-Cont}
		Let $(\wtilde{\Omega},\wtilde{\Fn},\wtilde{\Fb},\wtilde{\Pr},\wtilde{W},u)$ be a solution, obtained in Proposition \ref{Prop-pre-martingale}, of \eqref{eqn-main-prob-Ito} on $[0, T]$ such that $u(t) \in\bM$, for $t\in [0, T]$, and
		\begin{equation}\label{eqn-ener-1}
			\wtilde{\E} \Big[\sup_{t\in[0,T]} \big(\|u(t)\|_{H_0^1(\bO)}^{4} + \|u(t)\|_{L^p(\bO)}^{2p}\big)\Big] + \wtilde{\E}\Big[\int_0^T\big(\|\Delta u(t)\|_{L^2(\bO)}^2+ \|u(t)\|_{L^{2p-2}(\bO)}^{2p-2}\big)dt\Big]<\infty.
		\end{equation}
		Then, for $\wtilde{\Pr}$ almost all $\omega\in\wtilde{\Omega}$, 
		\begin{equation}
			u(\cdot, \omega) \in C([0,T]; \Vp\cap \bM).
		\end{equation}
		Moreover, for every $t\in[0, T]$, \eqref{eqn-strong-form-1} is satisfied $\wtilde{\Pr} - $a.s.
	\end{lemma}
	
	Let us now state the existence result for the martingale solution to the problem \eqref{eqn-main-prob-Ito}, which will be proved in the following sections with the help of the results stated above.
	
	\begin{theorem}\label{Thm-main}
		Let $T>0$ be fixed and $u_0\in \Vp\cap\bM$ be given. Then, there exists a martingale solution $(\wtilde{\Omega},\wtilde{\Fn},\wtilde{\Fb},\wtilde{\Pr},\wtilde{W},u)$, in sense of Definition \ref{Def-Martingale}, of the SPDE \eqref{eqn-main-prob-Ito}  such that
		\begin{equation*}
			\wtilde{\E}\bigg[\sup_{t\in[0,T]}\big(\|u(t)\|_{H_0^1(\bO)}^2+ \|u(t)\|_{L^p(\bO)}^p\big)\bigg]+	\wtilde{\E}\bigg[\int_0^T\big(\|\Delta u(t)\|_{L^2(\bO)}^2+ \|u(t)\|_{L^{2p-2}(\bO)}^{2p-2}\big)dt\bigg]<\infty.
		\end{equation*}
	\end{theorem}

	Now, we provide the definition of pathwise uniqueness.
	
	\begin{definition}
		Let $(\Omega,\Fn,\Fb,\Pr,W,u^i),$ for $ i=1, 2$ be two martingale solutions of the problem \eqref{eqn-main-prob-Ito} with $u^i(0) =u_0,$  for $i=1, 2$. Then, the solutions $u^1, u^2$ are \emph{pathwise unique} if for all $t\in[0, T], $ $u^1(t) =u^2(t)$ $\Pr-$a.s..
	\end{definition}
	
	The next result summarizes the main result of our work, showing that the pathwise uniqueness property holds for the SPDE \eqref{eqn-main-prob-Ito}, and hence that strong solutions exist which are unique in law.
	
	\begin{theorem}\label{Thm-uniqueness}
		Let $T>0$ be fixed and $u_0\in \Vp\cap\bM$ be given. Then, there exists a pathwise unique strong solution $u$, in the sense of Definition \ref{Def-Strong-Soln}, of the SPDE \eqref{eqn-main-prob-Ito} such that
		\begin{align}
			{\E}\bigg[\sup_{t\in[0,T]}\big(\|u(t)\|_{H_0^1(\bO)}^2+ \|u(t)\|_{L^p(\bO)}^p \big) + \int_0^T\left(\|\Delta u(t)\|_{L^2(\bO)}^2+ \|u(t)\|_{L^{2p-2}(\bO)}^{2p-2}\right)dt\bigg]<\infty.
		\end{align}
	\end{theorem}
	
	\begin{remark}
		The solution of \eqref{eqn-main-prob-Ito},  discussed in Theorem \ref{Thm-uniqueness}, is strong in both probabilistic and analytic (PDE) sense.
	\end{remark}
	
	The remainder of the manuscript is organized as follows.
	In Section \ref{Sec-Pre}, we present several fundamental tools required for the analysis, including the functional framework, linear and nonlinear operators, properties of the noise, and compactness results essential for establishing the tightness of the family of laws corresponding to the solutions of the SPDE \eqref{eqn-main-prob-Ito}. We also recall the classical Jakubowski-Skorokhod representation Theorem. These results provide the convergence framework that will be employed in Section \ref{Sec-Exist}, where we first formulate an approximate system of stochastic differential equations (SDE) using a modified Faedo-Galerkin scheme, derive the corresponding energy estimates and tightness arguments, and then prove Proposition \ref{Prop-pre-martingale}, supported by several auxiliary lemmas and an It\^o formula for the $L^p-$norm of the weak solution, see Lemma \ref{Lem-Ito}. By utilizing these results along with Lemma \ref{Lem-Cont}, the proof of Theorem \ref{Thm-main} is completed. Finally, in Section \ref{Sec-Path+Strong}, we prove pathwise uniqueness and by applying the Yamada-Watanabe Theorem provided in Section \ref{Sec-Append}, we conclude the manuscript by showing that the martingale solution obtained in Theorem \ref{Thm-main} is, in fact, a strong solution which is unique in law, see Theorem \ref{Thm-uniqueness}.
	
	
	\section{Preliminaries}\label{Sec-Pre}
	The primary aim of this section is to demonstrate the preliminaries used throughout the article, including the relevant function spaces, a linear and nonlinear operators.
	
	\subsection{Functional framework}
	{Let $C_c^\infty(\bO)$ consists of all infinite-time continuously differentiable functions having compact support.} Let  $L^p(\bO)$ denote the space of equivalence classes of Lebesgue measurable functions $f : \mathcal{O} \to\mathbb{R}$ such that $\int_{\mathcal{O}}|f(x)|^pdx<\infty$, where two measurable functions are equivalent if they are equal a.e. If $f \in L^p(\bO)$, then its $L^p-$norm is given by $\|f\|_{L^p(\bO)}:=\left(\int_{\bO}|f(x)|^pdx\right)^{1/p}$. When $p=2$, the Lebesgue  space $L^2(\bO)$ forms a Hilbert space with the scalar product $(\cdot,\cdot)$. Furthermore, we denote the Sobolev space by $H_0^1(\bO)$ (or $W^{1,2}_0(\bO)$). In other words, it is the collection of equivalence classes of Lebesgue measurable functions $f \in L^2(\bO)$ whose weak gradient $\nabla f \in L^2(\bO)$ and which vanish in the trace sense. The norm associated to $H_0^1(\bO)$, through the Poincar\'e inequality, is given by $\norm{f}_{H_0^1(\bO)} := \left(\int_{\bO}|\nabla f(x)|^2dx\right)^{\frac{1}{2}}$.
	If $X$ denotes a Banach space with it's dual as $X^\prime$, then the duality pairing is given by
	\begin{equation}
		\fourIdx{}{X^\prime}{}{X}{\langle \cdot, \cdot\rangle} = \langle \cdot, \cdot\rangle.
	\end{equation}	
	Now, we define dual of the Sobolev space $H_0^1(\bO)$, i.e., $H^{-1}(\bO) := (H_0^1(\bO))'$, with its associated norm 
	\begin{equation*}
		\norm{f}_{H^{-1}(\bO)} := \sup\big\{ \langle f, g\rangle:  g \in H_0^1(\bO),\ {\norm{g}_{H_0^1(\bO)}}\leq 1\big\}.
	\end{equation*}
	Moreover, the second order Hilbertian Sobolev space is denoted by $H^2(\mathcal{O})$.
	
	Next, we discuss the intersection and sum spaces that will be used in this work. From classical analysis, it is known that the Lebesgue space $L^{p^\prime}(\bO)$ (where $p^\prime = \frac{p}{p-1}$) and and the Sobolev space $H_0^{-1}(\bO)$ are Banach space equipped with the norms $\norm{\cdot}_{L^{p^\prime}(\bO)}$ and $\norm{\cdot}_{H_0^{-1}(\bO)}$, respectively. Moreover, the intersection $\Vp$, defined in \eqref{Def-Vp}, is a dense subspace of both $L^{p^\prime}(\bO)$ and $H_0^{-1}(\bO)$ with their associated norms. Thus, their sum space is defined as
	\begin{equation*}
		L^{p^\prime}(\bO) + H^{-1}(\bO) := \{\overline{f} + \wtilde{f} : \overline{f} \in L^{p^\prime}(\bO), \wtilde{f} \in H^{-1}(\bO)\}
	\end{equation*}
	It is a well-defined Banach space equipped with the norm
	\begin{equation*}
		\norm{f}_{L^{p^\prime}(\bO) + H^{-1}(\bO)}  := \inf\{\|\overline{f}\|_{L^{p^\prime}(\bO)} + \|\wtilde{f}\|_{H^{-1}(\bO)} : f = \overline{f} + \wtilde{f},\ \overline{f} \in L^{p^\prime}(\bO),\ \wtilde{f} \in H^{-1}(\bO)\}.
	\end{equation*}
	The space $\Vp$ is also a Banach space, where its norm is defined as
	\[\norm{f}_{\Vp} := \max\{\norm{f}_{L^p(\bO)}, \norm{f}_{H_0^1(\bO)}\},\]
	which is equivalent to the sum norm $\norm{f}_{L^p(\bO)} + \norm{f}_{H_0^1(\bO)}$. Furthermore, the dual space $L^{p^\prime}(\bO) + H^{-1}(\bO)$ is given by
	\begin{equation*}
		(L^{p^\prime}(\bO) + H^{-1}(\bO))^\prime \cong \Vp,
	\end{equation*}
	with the natural pairing
	\begin{equation*}
		\langle f, g \rangle = \langle \overline{f}, g \rangle + \wtilde{f}, g \rangle,
	\end{equation*}
	for all $f = \overline{f} + \wtilde{f} \in L^{p^\prime}(\bO) + H^{-1}(\bO) \cong \Vp^\prime$ and $g\in \Vp$. Therefore, \cite[cf. Section 2]{RF+HK+HS-05}, the $\Vp^\prime-$norm is given by
	\begin{equation*}
		\norm{f}_{\Vp^\prime} = \sup \big\{ {\langle \overline{f} + \wtilde{f}, g\rangle}  :  g \in \Vp,\ {\norm{g}_{\Vp}}  \leq 1\big\}.
	\end{equation*}
	
	\subsubsection{Nonlinear operator}
	Let us now consider a nonlinear operator $\C: L^p( \bO)\to L^{p^\prime}( \bO )$ defined as
	\begin{align}\label{Def-Nn}
		\C(v) := \abs{v}^{p-2} v, \; \; \mbox{with}\;\; p^\prime = \frac{p}{p-1}.
	\end{align}
	
	\begin{proposition}[{\cite[Proposition 2.2]{AB+ZB+MTM-25+}}]\label{Prop-Mono-Non-lin}
		Assume a nonlinear operator $\C$ is defined in \eqref{Def-Nn}. Then, for any $v_1,v_2\in L^p(\bO)$, $\C$ satisfies:
		\begin{equation}\label{eqn-nonlinear-est}
			\langle \C(v_1) - \C(v_2), v_1 - v_2 \rangle \geq \frac{1}{2^{p-2}} \norm{v_1 - v_2}^{p}_{L^p(\bO)}.
		\end{equation}
	\end{proposition}

	\subsection{Properties of the noise}
	Let us now examine certain properties of the noise coefficient featured in \eqref{eqn-main-prob-Ito}. To begin with, we demonstrate that the noise coefficients $\Nn_i$'s are locally Lipschitz continuous. For simplicity, we omit the subscript $i$ throughout this section and prove that $\Nn$ is locally Lipschitz in $\Vp$, \cite[cf.~Proposition 2.1]{ZB+JH-20}.
	
	\begin{lemma}\label{Lem-B-Loc-Lip}
		Let $f\in \Vp$ be fixed. Let us consider the map, see \eqref{eqn-def-Nn},
		\begin{align}\label{eqn-map-B}
			\Vp \ni v \mapsto \Nn(v):= f - (f,v)v \in \Vp.
		\end{align}
		Then, for all $v_1,v_2\in \Vp$, the map $\Nn$ is locally Lipschitz, i.e.,
		\begin{equation*}
			\|\Nn(v_1) -\Nn(v_2)\|_{\Vp} \leq \bigg(\frac{1}{\sqrt{\lambda_1}} \|f\|_{L^2(\bO)}\left(\|v_1\|_{H_0^1(\bO)} + \|v_2\|_{\Vp}\right) + \|f\|_{L^2(\bO)} \|v_1\|_{L^p(\bO)} \bigg) \|v_1-v_2\|_{\Vp},
		\end{equation*}
	where $\lambda_1$ is the first eigenvalue of the Dirichlet Laplacian.
	\end{lemma}
	
	\begin{proof}
		Let us choose and fix $f, v\in \Vp$. Then, it follows from the definition of the map $\Nn$ that
		\begin{align}
			\|\Nn(v)\|_{H_0^1(\bO)} & = \|f- (f,v)v\|_{H_0^1(\bO)}\leq\|f\|_{H_0^1(\bO)} +|(f,v)|\|v\|_{H_0^1(\bO)}\\
			& \leq \|f\|_{H_0^1(\bO)} + \|f\|_{L^2(\bO)}\|v\|_{L^2(\bO)}\|v\|_{H_0^1(\bO)}< \infty,\label{B-H_0^1-bound}
		\end{align}
		which implies that the operator $\Nn$ is well defined on $H_0^1(\bO)$. In a similar manner, using the Cauchy-Schwarz inequality it can be easily seen that 
		\begin{align}
			\|\Nn(v)\|_{L^p(\bO)} & = \|f- (f,v)v\|_{L^p(\bO)}\leq\|f\|_{L^p(\bO)} +|(f,v)|\|v\|_{L^p(\bO)}\\
			& \leq \|f\|_{L^p(\bO)} + \|f\|_{L^{2}(\bO)}\|v\|_{L^2(\bO)}\|v\|_{L^p(\bO)},\label{B-L^p-bound}
		\end{align}
		so the operator $\Nn$ is well defined in $L^p(\bO)$ as well as in $\Vp$. Next, we proceed to establish the local Lipschitz type bounds.		
		Let us fix and choose $v_1,v_2\in \Vp$. Then, by applying Poincar\'e's inequality, we obtain
		\begin{align}
			\|\Nn(v_1) -\Nn(v_2)\|_{H_0^1(\bO)}
			& = \|f- (f,v_1)v_1- (f- (f,v_2)v_2)\|_{H_0^1(\bO)}\\ 
			& \leq |(f,v_1)|\|v_1-v_2\|_{H_0^1(\bO)} +|(f,v_1-v_2)|\|v_2\|_{H_0^1(\bO)}\nonumber\\ 
			& \leq \frac{1}{\sqrt{\lambda_1}}\|f\|_{L^2(\bO)}\left(\|v_1\|_{H_0^1(\bO)} + \|v_2\|_{H_0^1(\bO)}\right)\|v_1-v_2\|_{H_0^1(\bO)}.\label{eqn-b-differene-1}
		\end{align}
		Similarly, by applying the Cauchy-Schwarz and Poincar\'e inequalities, we deduce
\begin{align}
	&\|\Nn(v_1) -\Nn(v_2)\|_{L^p(\bO)}
	= \|f- (f,v_1)v_1- (f- (f,v_2)v_2)\|_{L^p(\bO)}\\ 
	& \leq |(f,v_1)|\|v_1-v_2\|_{L^p(\bO)} +|(f,v_1-v_2)|\|v_2\|_{L^p(\bO)}\nonumber\\ 
	& \leq \|f\|_{L^{2}(\bO)} \big(\|v_1\|_{L^2(\bO)} \|v_1-v_2\|_{L^p(\bO)} + \|v_2\|_{L^p(\bO)}\|v_1-v_2\|_{L^2(\bO)} \big)\label{eqn-b-differene-2}\\ 
	& \leq \|f\|_{L^{2}(\bO)} \bigg(\|v_1\|_{L^2(\bO)} \|v_1-v_2\|_{L^p(\bO)} + \frac{1}{\sqrt{\lambda_1}}\|v_2\|_{L^p(\bO)}\|v_1-v_2\|_{H_0^1(\bO)} \bigg).
\end{align}
	Thus, adding the above two estimates yields the required result.
	\end{proof}
	
	In the following result, we demonstrate the existence of the Fr\'echet derivative, denoted by $d_v(\cdot)$, of the noise coefficient $\Nn$. 
	
	\begin{lemma}[{\cite[Lemma 2.3]{ZB+JH-20}}]\label{lem-Fre}
		Assume that  $f \in L^2(\bO)$. Let $\Nn : L^2(\bO)\to L^2(\bO)$ be the map as defined in \eqref{eqn-map-B}. Then, for $v \in L^2(\bO)$, the Fr\'echet derivative $d_v\Nn$ exists and satisfies
		\begin{align}\label{eqn-Fre}
			d_v\Nn(h)=- (f,v)h- (f,h)v ,\ \text{ for all }\ h\in L^2(\bO).
		\end{align}
	\end{lemma}
	
	\begin{proof}
		Let us first choose and fix $h,v, f\in L^2(\bO).$ We consider
		\begin{align*}
			\Nn(h+v) - \Nn(v)& = - (f,h+v)(h+v) + (f,v)v\\
			& =- (f,h)h - (f,h)v - (f,v)h.
		\end{align*}
		Then, it immediate follows that
		\begin{align}
			& \lim_{\|h\|_{L^2(\bO)}\to 0}\frac{\|\Nn(h+v) -\Nn(v) - ( - (f,h)v- (f,v)h) \|_{L^2(\bO)}}{\|h\|_{L^2(\bO)}}\\
			& = \lim_{\|h\|_{L^2(\bO)}\to 0}\frac{\|(f,h)h\|_{L^2(\bO)}}{\|h\|_{L^2(\bO)}}\leq \|f\|_{L^2(\bO)}\lim_{\|h\|_{L^2(\bO)}\to 0}\|h\|_{L^2(\bO)} = 0.
		\end{align}
		Hence the Fr\'echet derivative $d_v\Nn$ exists and satisfies \eqref{eqn-Fre}.
	\end{proof}
	
	Next, we consider the mapping
	\begin{equation}\label{eqn-kappa-def}
		\kappa: L^p(\bO) \ni v \mapsto d_v\Nn(\Nn(v)) \in L^p(\bO).
	\end{equation}
	Before proceeding further, let us first note that $\kappa$ is well defined, i.e., for fixed $f, v\in L^p(\bO)$, it holds that
	\begin{align}
		\|\kappa(v)\|_{L^p(\bO)}
		& = \|d_{v}\Nn(\Nn(v))\|_{L^p(\bO)}=\| -  (f,v)\Nn(v) - (f,\Nn(v))v\|_{L^p(\bO)}\\ 
		& \leq|(f,v)|\|\Nn(v)\|_{L^p(\bO)} +|(f,\Nn(v))|\|v\|_{L^p(\bO)}\\ 
		& \leq\|f\|_{L^2(\bO)}\|\Nn(v)\|_{L^p(\bO)}\|v\|_{L^2(\bO)} + \|f\|_{L^2(\bO)}\|\Nn(v)\|_{L^2(\bO)}\|v\|_{L^p(\bO)}\\
		& \le \|f\|_{L^2(\bO)} \|v\|_{L^p(\bO)} \big(\|f\|_{L^p(\bO)} + \|f\|_{L^{2}(\bO)}\|v\|_{L^2(\bO)}\|v\|_{L^p(\bO)}\big)\\
		& \quad + \|f\|_{L^2(\bO)} \|v\|_{L^p(\bO)} \big(\|f\|_{L^2(\bO)} + \|f\|_{L^{2}(\bO)}\|v\|_{L^2(\bO)}^2\big)\\
		& \le \|f\|_{L^2(\bO)} \|v\|_{L^p(\bO)} \|f\|_{L^p(\bO)} + \|f\|_{L^{2}(\bO)}^2 \|v\|_{L^2(\bO)}\|v\|_{L^p(\bO)}^2 \\
		& \quad + \|f\|_{L^2(\bO)}^2 \|v\|_{L^p(\bO)}  + \|f\|_{L^{2}(\bO)}^2 \|v\|_{L^2(\bO)}^2 \|v\|_{L^p(\bO)}.\label{eqn-kappa-est}
	\end{align}
	
	\begin{remark}
		Observe that for fixed $v,f\in L^2(\bO)$, we have the following relations:
		\begin{align}
			(\Nn(v),v)& =(f- (f,v)v,v)=(f,v)(1-\|v\|_{L^2(\bO)}^2),\label{eqn-est-a}\\
			(\kappa(v),v)& =(d_v\Nn(\Nn(v)),v)=(- (f,v)\Nn(v) - (f,\Nn(v))v,v)\\
			& =(- (f,v)(f- (f,v)v) - (f,f- (f,v)v)v,v)\\
			& =(- (f,v)f-\|f\|_{L^2(\bO)}^2v+2|(f,v)|^2v,v)\\
			& = - \|f\|_{L^2(\bO)}^2 \|v\|_{L^2(\bO)}^2 + |(f,v)|^2 \big(2\|v\|_{L^2(\bO)}^2-1\big),\label{eqn-est-c}\\
			(\Nn(v),\Nn(v))& =(f- (f,v)v,f- (f,v)v)=\|f\|_{L^2(\bO)}^2+|(f,v)|^2\big(\|v\|_{L^2(\bO)}^2-2\big).\label{eqn-est-b}
		\end{align}
		Similarly, when $v,f\in H_0^1(\bO)$, we infer
		\begin{align}
			(\Nn(v),-\Delta v)
			& =(f- (f,v)v,-\Delta v)=(\nabla f,\nabla v) - (f,v)\|v\|_{H_0^1(\bO)}^2,\label{eqn-est-d}\\
			(\kappa(v),-\Delta v)
			& =(d_v\Nn(\Nn(v)),-\Delta v)=(- (f,v)\Nn(v) - (f,\Nn(v))v,-\Delta v)\\
			& =(- (f,v)(f- (f,v)v) - (f,f- (f,v)v)v,-\Delta v)\\
			& =(- (f,v)f-\|f\|_{L^2(\bO)}^2v+2|(f,v)|^2v, -\Delta v)\\
			& =- (f,v)(\nabla f,\nabla v) -\|f\|_{L^2(\bO)}^2\|v\|_{H_0^1(\bO)}^2+2|(f,v)|^2\|v\|_{H_0^1(\bO)}^2,\label{eqn-est-e}\\
			(\Nn(v),-\Delta \Nn(v))
			& =(\nabla \Nn(v),\nabla \Nn(v))=(\nabla f- (f,v)\nabla v,\nabla f- (f,v)\nabla v)\\
			& = \|f\|_{H_0^1(\bO)}^2-2(f,v)(\nabla f,\nabla v) +|(f,v)|^2\|v\|_{H_0^1(\bO)}^2. \label{eqn-est-f}
		\end{align}
		Also, for all $v,f\in L^p(\bO)$,  we find
		\begin{align}
			(\Nn(v),|v|^{p-2}v)
			& =(f- (f,v)v,|v|^{p-2}v)=(f,|v|^{p-2}v) - (f,v)\|v\|_{L^p(\bO)}^p,\label{eqn-est-g}\\
			(\kappa(v),|v|^{p-2}v)
			& =(d_v\Nn(\Nn(v)),|v|^{p-2}v)=(- (f,v)\Nn(v) - (f,\Nn(v))v,|v|^{p-2}v)\\
			& =(- (f,v)(f- (f,v)v) - (f,f- (f,v)v)v,|v|^{p-2}v)\\
			& =(- (f,v)f-\|f\|_{L^2(\bO)}^2v+2|(f,v)|^2v,|v|^{p-2}v)\\
			& =- (f,v)(f,|v|^{p-2}v) -\|f\|_{L^2(\bO)}^2\|v\|_{L^p(\bO)}^p+2|(f,v)|^2\|v\|_{L^p(\bO)}^p,\label{eqn-est-h}\\
			(\Nn(v),|v|^{p-2}\Nn(v))
			& =(f- (f,v)v,|v|^{p-2}(f- (f,v)v))\\
			& =(f,|v|^{p-2}f) -2(f,v)(f,|v|^{p-2}v) +|(f,v)|^2\|v\|_{L^p(\bO)}^p.\label{eqn-est-i}
		\end{align}
	\end{remark}

	In the next result, we establish that the Fr\'echet derivative map $\kappa(\cdot)$ is locally Lipschitz in $\Vp$, \cite[cf.~Proposition 2.2]{ZB+JH-20}.
	
	\begin{lemma}
		Under the spirit of Lemma \ref{Lem-B-Loc-Lip}, we show the following:
If $f\in \Vp$, then $\kappa$ maps the space $\Vp$ into itself and it is locally Lipschitz, i.e., for all $v_1,v_2\in \Vp$,
			\begin{align}\label{eqn-est-2}
				\|\kappa(v_1) -\kappa(v_2)\|_{\Vp} 
				& \leq \bigg(\frac{2}{\lambda_1^2}\|f\|_{H^1_0(\bO)}^2\left(\lambda_1 + (\|v_1\|_{H^1_0(\bO)} + \|v_2\|_{H^1_0(\bO)})^2\right)\\
				& \qquad+ \|f\|_{L^2(\bO)}^2 \big(\|v_2\|_{L^2(\bO)}\|v_1\|_{L^2(\bO)} + 1 + \|v\|_{L^p(\bO)}^2 \big)\\ 
				&\qquad + \frac{1}{\sqrt{\lambda_1}} \|f\|_{L^2(\bO)} \Big(\|f\|_{L^p(\bO)} + \|f\|_{L^2(\bO)} \|v_1\|_{L^2(\bO)}\|v_1\|_{L^p(\bO)}\\
				& \qquad + \|f\|_{L^2(\bO)} \|v_2\|_{L^p(\bO)} \big(\|v_1\|_{L^2(\bO)} + 2\|v_2\|_{L^2(\bO)}\big)\Big)\bigg) \|v_1-v_2\|_{\Vp},
			\end{align}
		where $\lambda_1$ denotes the first eigenvalue of the Dirichlet-Laplacian.
	\end{lemma}
	
	\begin{proof}
		Let us choose and fix $v_1,v_2, f\in \Vp$. Then, the estimates \eqref{B-H_0^1-bound} and \eqref{eqn-b-differene-1} along with Poincar\'e's inequality yield
		\begin{align*}
			&	\|\kappa(v_1) -\kappa(v_2)\|_{H_0^1(\bO)}
			\\&
			=\|d_{v_1}\Nn(\Nn(v_1)) -d_{v_2}\Nn(\Nn(v_2))\|_{H_0^1(\bO)}
			\\
			& = \| -  (f,v_1)\Nn(v_1) - (f,\Nn(v_1))v_1+ (f,v_2)\Nn(v_2) + (f,\Nn(v_2))v_2\|_{H_0^1(\bO)}
			\\ & \leq|(f,v_2)|\|\Nn(v_2) -\Nn(v_1)\|_{H_0^1(\bO)} +|(f,v_2-v_1)\|\Nn(v_1)\|_{H_0^1(\bO)}
			\\ & \quad +|(f,\Nn(v_2) -\Nn(v_1))|\|v_2\|_{H_0^1(\bO)} +|(f,\Nn(v_1))|\|v_2-v_1\|_{H_0^1(\bO)}\nonumber
			\\ & \leq \frac{2}{\lambda_1}\|f\|_{H_0^1(\bO)} \big( \|v_2\|_{H_0^1(\bO)}\|\Nn(v_1) -\Nn(v_2)\|_{H_0^1(\bO)} + \|\Nn(v_1)\|_{H_0^1(\bO)}\|v_2-v_1\|_{H_0^1(\bO)} \big) \\
			& \leq \frac{2}{\lambda_1^2}\|f\|_{H_0^1(\bO)}\big( \|v_2\|_{H_0^1(\bO)}\|f\|_{H_0^1(\bO)}\left(\|v_1\|_{H_0^1(\bO)} + \|v_2\|_{H_0^1(\bO)}\right)\|v_1-v_2\|_{H_0^1(\bO)}\\
			& \quad + \|f\|_{H_0^1(\bO)}(\lambda_1+ \|v_1\|_{H_0^1(\bO)}^2)\|v_2-v_1\|_{H_0^1(\bO)}\big)\\
			& \leq \frac{2}{\lambda_1^2}\|f\|_{H_0^1(\bO)}^2\big( \lambda_1 + \|v_1\|_{H_0^1(\bO)}^2+ \|v_2\|_{H_0^1(\bO)}^2+ \|v_1\|_{H_0^1(\bO)}\|v_2\|_{H_0^1(\bO)}\big) \|v_2-v_1\|_{H_0^1(\bO)}\\
			& \leq \frac{2}{\lambda_1^2}\|f\|_{H_0^1(\bO)}^2\big( \lambda_1 + (\|v_1\|_{H_0^1(\bO)} + \|v_2\|_{H_0^1(\bO)})^2\big) \|v_2-v_1\|_{H_0^1(\bO)}.
		\end{align*}
		Similarly, the estimates \eqref{B-L^p-bound} and \eqref{eqn-b-differene-2} together with H\"older's inequality 
		implies
		\begin{align}
			&\|\kappa(v_1) -\kappa(v_2)\|_{L^p(\bO)}  =\|d_{v_1}\Nn(\Nn(v_1)) -d_{v_2}\Nn(\Nn(v_2))\|_{L^p(\bO)} \\
			& =\| -  (f,v_1)\Nn(v_1) - (f,\Nn(v_1))v_1+ (f,v_2)\Nn(v_2) + (f,\Nn(v_2))v_2\|_{L^p(\bO)} 	\\
			& \leq |(f,v_2)| \|\Nn(v_2) -\Nn(v_1)\|_{L^p(\bO)} + |(f,v_2-v_1)| \|\Nn(v_1)\|_{L^p(\bO)}\\
			& \quad + |(f,\Nn(v_2) -\Nn(v_1))| \|v_2\|_{L^p(\bO)} +|(f,\Nn(v_1))| \|v_2-v_1\|_{L^p(\bO)}\\
			& \leq \|f\|_{L^2(\bO)} \|v_2\|_{L^2(\bO)} \|\Nn(v_2) -\Nn(v_1)\|_{L^p(\bO)} + \|f\|_{L^2(\bO)} \|v_2-v_1\|_{L^2(\bO)} \|\Nn(v_1)\|_{L^p(\bO)}\\
			& \quad + \|f\|_{L^2(\bO)} \|\Nn(v_2) -\Nn(v_1)\|_{L^2(\bO)} \|v_2\|_{L^p(\bO)} + \|f\|_{L^2(\bO)} \|\Nn(v_1)\|_{L^2(\bO)} \|v_2-v_1\|_{L^p(\bO)} \label{eqn-est-3}\\
			& \leq \|f\|_{L^2(\bO)}^2 \|v_2\|_{L^2(\bO)}\bigg(\|v_1\|_{L^2(\bO)} \|v_1-v_2\|_{L^p(\bO)} + \frac{1}{\sqrt{\lambda_1}}\|v_2\|_{L^p(\bO)}\|v_1-v_2\|_{H_0^1(\bO)} \bigg)\\
			&\quad + \|f\|_{L^2(\bO)} \|v_2-v_1\|_{L^2(\bO)} \big(\|f\|_{L^p(\bO)} + \|f\|_{L^2(\bO)} \|v_1\|_{L^2(\bO)}\|v_1\|_{L^p(\bO)}\big)\\
			& \quad + \frac{1}{\sqrt{\lambda_1}}\|f\|_{L^2(\bO)}^2 \|v_2\|_{L^p(\bO)} \big(\|v_1\|_{L^2(\bO)} + \|v_2\|_{L^2(\bO)}\big) \|v_1-v_2\|_{H_0^1(\bO)} \\
			&\quad + \|f\|_{L^2(\bO)}^2\big(1 + \|v\|_{L^p(\bO)}^2 \big) \|v_2-v_1\|_{L^p(\bO)}\\
			& \leq \|f\|_{L^2(\bO)}^2 \big(\|v_2\|_{L^2(\bO)}\|v_1\|_{L^2(\bO)} + 1 + \|v\|_{L^p(\bO)}^2 \big) \|v_1-v_2\|_{L^p(\bO)}\\ 
			&\quad + \frac{1}{\sqrt{\lambda_1}} \|f\|_{L^2(\bO)} \bigg(\|f\|_{L^p(\bO)} + \|f\|_{L^2(\bO)} \|v_1\|_{L^2(\bO)}\|v_1\|_{L^p(\bO)}\\
			& \quad + \|f\|_{L^2(\bO)} \|v_2\|_{L^p(\bO)} \big(\|v_1\|_{L^2(\bO)} + 2\|v_2\|_{L^2(\bO)}\big)\bigg) \|v_1-v_2\|_{H_0^1(\bO)}.
		\end{align}
		Thus, by combining the above two estimates, the proof follows immediately.
	\end{proof}

	\subsection{Compactness}\label{Sec-Compact}
	The main role of this section is to demonstrate few deterministic compactness results and utilize it together with Aldous condition, see Definition \ref{Aldous}, to obtain a tightness criterion, see Corollary \ref{Cor-Aldous}. The two key components in the proof of main Theorem \ref{Thm-main} are the tightness criterion and the Jakubowski-Skorokhod representation Theorem, see Theorem \ref{thm-law}. For the tightness criterion, we need the following functional spaces:
	
	\begin{itemize}
		\item[$\boldsymbol{\star}$] $C([0,T];L^2(\bO)):=$ the space of continuous functions $v :[0, T] \to L^2(\bO)$ with the norm, $\|v\|_{C([0,T];L^2(\bO))}:=\sup_{t\in [0,T]}\|v(t)\|_{L^2(\bO)}$, induced topology $\mathcal{T}_1$
		\item[$\boldsymbol{\star}$] $L^2_w(0, T; D(A)) :=$ the space $L^2(0, T; D(A))$ with the weak topology $\mathcal{T}_2$,
		\item[$\boldsymbol{\star}$] $L^2(0, T; H_0^1(\bO)) :=$ the space of measurable functions $v :[0, T] \to H_0^1(\bO)$ such that
		$$\|v\|_{L^2(0, T; H_0^1(\bO))}=\left(\int_0^T\|v(t)\|_{H_0^1(\bO)}^2dt\right)^{\frac{1}{2}}<\infty,$$ 
		with the norm, $\|v\|_{L^2(0, T; H_0^1(\bO))}$, induced topology $\mathcal{T}_3$.
		\item[$\boldsymbol{\star}$] $C_w([0, T]; \Vp) :=$ the space of weakly continuous functions $v :[0, T] \to \Vp$, 
		endowed with the weakest topology $\mathcal{T}_4$ such that, for all $h \in \Vp^\prime \cong H^{-1}(\bO) + L^{p^\prime}(\bO)$, the mappings
		\[
		C_w([0,T]; \Vp)\ni v\mapsto {\langle v(\cdot),h\rangle} \in C([0,T];\mathbb{R})
		\]
		are continuous. In particular, $v_m\to v$ in $C_w([0, T]; \Vp)$ if and only if  for all $h \in \Vp^\prime$: 
		\begin{equation*}
			\lim_{m\to\infty}\sup_{t\in[0,T]}\big| {\langle v_m(t) -v(t),h\rangle} \big| = 0.
		\end{equation*}
	\end{itemize}
	Let us fix $R>0$ and consider the closed ball centered at $0$ with radius $R$ in a separable reflexive Banach space $\Vp$
	$$\mathbb{B}:=\left\{x\in \Vp: \|x\|_{\Vp}\leq R\right\}.$$
	Suppose $q$ is the metric compatible with the weak topology on $\mathbb{B}$. Then, we consider the following subspace of the space $	C_w([0,T]; \Vp)$:
	\begin{align}
		C_w([0,T];\mathbb{B}) & := \mbox{the space of weakly continuous functions } v: [0,T]\to \Vp\\
		& \qquad\mbox{ such that } \sup_{t\in[0,T]}\|v(t)\|_{\Vp}\leq R.\label{eqn-cbw}
	\end{align}
	The space $C_w([0,T];\mathbb{B})$ is metrizable  with respect to the metric
	\begin{align}\label{eqn-ruv}
		\varrho(v_1,v_2) := \sup_{t\in[0,T]} q(v_1(t),v_2(t)).
	\end{align} An application of the Banach-Alaoglu Theorem yields $\mathbb{B}$ is compact and therefore, $C_w([0,T];\break\mathbb{B})$  is a complete metric space.
	
	Next result asserts that every sequence $\{u_m\}_{m\in\N}\subset C([0, T]; \mathbb{B})$ converging in $C([0, T]; \break L^2(\bO))$ also converges in the space $C_w([0,T];\mathbb{B})$.
	
	\begin{lemma}[{\cite[Lemma 4.1]{ZB+GD-21}}]\label{lem-strong}
		Let choose and fix a sequence of funciton $u_m:[0, T] \to \Vp$, for $m\in\N$, such that
		\begin{itemize}
			\item[(i)] $\displaystyle{\sup_{m\in\N}\sup_{t\in[0,T]}\|u_m(t)\|_{\Vp}\leq R.}$
			\item[(ii)] $u_m\to u$ in $C([0,T];L^2(\bO))$ as $m\to\infty$.
		\end{itemize}
		Then, $u, u_m\in C_w([0,T];\mathbb{B})$, for every $m\in\N$, and $u_m\to u$ in $C_w([0,T];\mathbb{B})$ as $m\to\infty$.
	\end{lemma}
	
	\noindent Let us define
	\begin{equation}\label{def-zt}
		\begin{aligned}
			\mathcal{Z}_T:& =  C_w([0,T]; \Vp) \cap L^2_w(0,T;D(A))	\cap L^2(0,T;H_0^1(\bO)) \cap C([0,T];L^2(\bO)),
		\end{aligned}
	\end{equation}
	equipped with the coarsest topology $\mathcal{T}$, that is finer than $\mathcal{T}_1$, $\mathcal{T}_2$, $\mathcal{T}_3$ and $\mathcal{T}_4$, in other words, $\mathcal{T}$ is the supremum of these topologies.
	
	\begin{lemma}\label{lem-compact}
		Suppose $(\mathcal{Z}_T$, $\mathcal{T})$ be fixed as in \eqref{def-zt}. Then, a set $\wtilde{\K}\subset\mathcal{Z}_T$ is $\mathcal{T} - $relatively compact if the following three conditions hold:
		\begin{itemize}
			\item[(i)] \begin{equation}
				\sup_{u\in\wtilde{\K}}\sup_{t\in[0,T]}\|u(t)\|_{\Vp}<\infty.
			\end{equation}
			\item[(ii)] $\wtilde{\K}$ is bounded in the space  $L^2(0,T;D(A))$, i.e., 
			\begin{equation}
				\sup_{u\in\wtilde{\K}}\bigg(\int_0^T\|\Delta u(t)\|_{L^2(\bO)}^2dt \bigg) < \infty.
			\end{equation}
			\item[(iii)] \begin{equation}
				\lim_{\delta\to 0}\sup_{u\in\wtilde{\K}}\sup_{s,t\in[0,T],|t-s|\leq\delta}\|u(t) -u(s)\|_{L^2(\bO)} = 0.
			\end{equation}
		\end{itemize}
	\end{lemma}
	\begin{proof}
		Let us choose and fix $\wtilde{\K}\subset \mathcal{Z}_T$. Due to hypothesis (i), one can consider the metric space $C_w([0, T]; \mathbb{B}) \subset C_w([0, T]; \Vp)$ defined by \eqref{eqn-cbw} and \eqref{eqn-ruv} with
		$$r=\sup_{u\in\wtilde{\K}}\sup_{s\in[0,T]}\|u(s)\|_{\Vp}.$$
		From the hypothesis (ii), it follows that the set $\wtilde{\K}$ with the subspace to topology $\mathcal{T}_2$, i.e., the weak topology on $L^2(0, T; D(A))$, is metrizable. Since the restrictions to $\wtilde{\K}$ of the four topologies considered in $\mathcal{Z}_T$ are metrizable, it follows that the compactness and sequential compactness of a subset of $\mathcal{Z}_T$ are equivalent.	
		
		Next, let us select and fix a sequence $\{u_m\}_{m\in\N}$ in $\wtilde{\K}$. Then, by an application of the Banach-Alaoglu Theorem along with the condition (ii) yields that $\overline{\wtilde{\K}}$, the closure of $\wtilde{\K}$, is relatively compact in $L^2_w(0, T; D(A))$. Thus, condition (iii) implies that the sequence of functions $\{u_m\}_{m\in\N}$ is equicontinuous in $C([0, T];L^2(\bO))$. Since the embeddings $D(A)\embed H_0^1(\bO)\embed L^2(\bO)$ and $D(A)\embed H_0^1(\bO)$ are continuous and compact, respectively, it follows by the Dubinsky Theorem, see \cite[Theorem IV.4.1]{MJV+AVF-88}, together with conditions (ii) and (iii) that $\wtilde{\K} \subset L^2(0, T; H_0^1(\bO)) \cap C([0, T]; L^2(\bO))$ is compact. In particular, one can find a subsequence, denoted by the same notation $\{u_m\}_{m\in\N}$, that converges in $L^2(\bO)$. Hence, an application of Lemma \ref{lem-strong} infer that the sequence $\{u_m\}_{m\in\N}$ converges in $C_w([0, T]; \mathbb{B})$, which concludes the proof.
	\end{proof}
	
	\subsubsection{Tightness criterion}
	Let $(\Sb ,\varrho )$ be a separable and complete metric space.
	\begin{definition}[{\cite[Definition 4.3]{ZB+GD-21}}]
		For each fixed $u \in C  ([0,T];\Sb )$.
		The modulus of continuity of $u$ on $[0,T]$ is defined by
		$$
		\mu (u,\delta ):= \sup_{s,t \in [0,T] , |t-s|\le \delta } \varrho (u(t),u(s)),
		\ \delta >0 .
		$$
	\end{definition}
	
	\noindent
	Let $(\Omega , \Fn, \Pr )$ be a probability space with filtration $\Fb:=\{{\Fn}_{t}\}_{t \geq 0}$ satisfying the usual conditions, and let $\{X_m\}_{m \in \N }$ be a sequence of continuous $\Fb-$adapted  $\Sb -$valued processes.
	
	\begin{definition}[{\cite[Definition 4.4]{ZB+GD-21}}]
		The sequence $\{X_m\}_{ \in \N }$ of $\Sb -$valued random variables
		satisfies  condition $[\mathbf{\wtilde{T}}]$ if and only if for all $ \eps >0,$ for all $ \eta >0,$ there exists  $ \delta  >0 $ such that
		\begin{equation} \label{cond_modulus}
			\sup_{m \in \N } \, \Pr \bigl\{  \mu (X_m ,\delta )  > \eta \bigr\}  \le \eps .
		\end{equation}
	\end{definition}

	\begin{lemma}[{\cite[Lemma 4.5]{ZB+GD-21}}]\label{modulus_conv}
		Suppose that $\{X_m\}_{ \in \N }$ satisfies condition $[\mathbf{\wtilde{T}}]$. Let us denote the law of $X_m$ on $C  ([0,T]; \Sb )$ by ${\Pr }_{m}$, for each fixed $m \in \N $. Then, for every $\eps >0$, there exists a subset	$ {A}_{\eps } \subset C  ([0,T], \Sb )$ that satisfy
		$$ \sup_{m\in \N } {\Pr }_{m} ({A}_{\eps }) \ge 1 - \eps, $$
		and
		\begin{equation} \label{E:modulus_conv}
			\lim_{\delta \to 0 }  \sup_{u \in {A}_{\eps } }  \mu (u,\delta ) =0.
		\end{equation}
	\end{lemma}

	Let us recall from \cite{DA-78} that the Aldous condition is connected with condition $[\mathbf{\wtilde{T}}]$. Thus, by means of stopped processes, this avail us to examine the modulus of continuity for the sequence of stochastic processes.
	\begin{definition}[{\cite[Definition 4.6]{ZB+GD-21}}] \label{Aldous}
		A sequence $\{{X}_m\}_{m\in \N }$  satisfies  condition $[\mathbf{A}]$
		if and only if 	for all $ \eps >0$,  for all $\eta >0 $, there exists $ \delta >0$ such that for every sequence $\{{\tau}_m \}_{m \in \N }$ of $\Fb-$stopping times with
		${\tau }_m\le T,$ one has
		$$
		\sup_{m \in \N} \, \sup_{0 \le \theta \le \delta }  \Pr \bigl\{
		\varrho \bigl( {X}_m ({\tau }_m + \theta ),{X}_m ( {\tau }_m ) \bigr) \ge \eta \bigr\}  \le \eps.
		$$
	\end{definition}
	
	\begin{lemma}[{\cite[Theorem 3.2]{MM-88}}] \label{Aldous_equiv}
		Conditions $[\mathbf{A}]$ and $[\mathbf{\wtilde{T}}]$ are  equivalent.
	\end{lemma}
	
	Applying the compactness criterion, from Lemma \ref{lem-compact}, together with the preceding results on Aldous condition, we obtain the following corollary. This will be used in the next section to establish the tightness of the laws for the processes defined by a modified Faedo-Galerkin approximation.
	
	\begin{corollary}[Tightness criterion]\label{Cor-Aldous}
		Let $\{X_m \}_{m \in \N }$ be a sequence of continuous $\Fb-$adapted $L^2(\bO) -$valued processes. Then
		\begin{itemize}
			\item[(a)] there exists a positive constant ${C}_{1}$ such that
			$$
			\sup_{m\in \N} \E \Big[ \sup_{t \in [0,T]} \big( \| X_m (t) \|_{H^1_0(\bO)}^{2} + \|X_m(t)\|_{L^p(\bO)}^p\big) \Big]  \le {C}_{1} ,
			$$
			\item[(b)] there exists a positive constant ${C}_{2}$ such that
			\begin{equation*}
				 \sup_{m \in \N} \E \left[  \int_{0}^{T} \|AX_m (t)\|_{L^2(\bO)}^2  dt \right] \le {C}_{2} ,
			\end{equation*}
			\item[(c)]  $\{X_m \}_{m \in \N }$ satisfies the Aldous condition $[\mathbf{A}]$ in $L^2(\bO)$.
		\end{itemize}
		Let $\Pr_{m}$ be the law of $X_m $ on $\mathcal{Z}_T$, for all $m\in\N$.
		Then, for every $\varepsilon >0 ,$ there exists a compact subset ${K}_{\varepsilon }$ of $\mathcal{Z}_T $ such that 
		$$
		\sup_{m\in \N} \Pr_m ({K}_{\varepsilon })\ge 1-\varepsilon .
		$$
	\end{corollary}
	\begin{proof}
		Let us choose and fix $\eps>0$. Then, by the Chebyshev inequality and assumption (a), for any fixed $m \in\N$ and $r>0$, we deduce that
		\begin{align}
			& \Pr_m\Big(\Big\{X_m\in\mathcal{Z}_T:\sup_{t\in[0,T]} \big( \|X_m(t)\|_{H_0^1(\bO)}^2+ \|X_m(t)\|_{L^p(\bO)}^2\big) > r \Big\}\Big)\\
			& \leq\frac{1}{r}\E\Big[\sup_{t\in[0,T]} \big( \|X_m(t)\|_{H_0^1(\bO)}^2+ \|X_m(t)\|_{L^p(\bO)}^2\big) \Big]\leq\frac{C_1}{r}.\label{eqn-cal}
		\end{align}
		Let $R_1>0$ be such that $\frac{C_1}{R_1}\leq\frac{\eps}{3}$ and let $$B_1:=\Big\{X_m\in\mathcal{Z}_T:\sup_{t\in[0,T]} \big( \|X_m(t)\|_{H_0^1(\bO)}^2+ \|X_m(t)\|_{L^p(\bO)}^2 \big) \leq R_1 \Big\}.$$ 
		Then, from the estimate \eqref{eqn-cal}, we have
		\begin{equation*}
			\Pr_m\big(\big\{X_m\in B_1^c\big\}\big)\leq\frac{\eps}{3}.
		\end{equation*}
		Similarly, one can define
		$$B_2:=\left\{X_m\in\mathcal{Z}_T:\int_0^T \|AX_m(t)\|_{L^2(\bO)}^2 dt\leq R_2 \right\},$$
		  where $R_2>0$ is chosen such that $\frac{C_2}{R_2}\leq\frac{\eps}{3}$ and a calculation similar to \eqref{eqn-cal} yields
		\begin{align*}
			\Pr_m\big(\big\{X_m\in B_2^c\big\}\big)\leq\frac{\eps}{3}.
		\end{align*}
		Now, using Lemmas \ref{modulus_conv} and \ref{Aldous_equiv}, we assert the existence of a subset $A_{\eps}\subset C([0,T];L^2(\bO))$ such that
		\begin{equation*}
			\sup_{m\in\N}\Pr_m\left(A_{\eps}\right)\geq 1-\frac{\eps}{3},
		\end{equation*}
		and
		\begin{equation*}
			\lim_{\delta\to 0}\sup_{u\in A_{\eps}}\sup_{s,t\in[0,T],|t-s|\leq\delta}\|u(t) -u(s)\|_{L^2(\bO)}=0.
		\end{equation*}
		It is sufficient to define $K_{\eps}$ as the closure of the set $B_1\cap B_2\cap A_{\eps}$ in $\mathcal{Z}_T$. By Lemma \ref{lem-compact}, $K_{\eps}$ is compact in $\mathcal{Z}_T$, which completes the proof.
	\end{proof}

	\subsubsection{The Skorokhod Theorem}
	Next, we present a well-known generalization of the Skorokhod Theorem, namely the Jakubowski--Skorokhod Theorem, see Brze\'zniak and Ond\v{r}ej\'at \cite{ZB+MO-07} or cf.\;\cite{AJ-97} also. We apply it in an appropriate topological space to obtain a convergent sequence.
	
	\begin{theorem}[{\cite[Theorem 4.9]{ZB+GD-21}}]\label{thm-law}
		Let $\mathcal{X}$ be a topological space such that there exists a sequence $\{f_m\}_{m\in\N}$ of continuous functions $f_m:\mathcal{X}\to\mathbb{R}$ that separates points of $\mathcal{X}$. Let us denote by $\mathcal{S}$, the $\sigma-$algebra generated by the sequence of maps $\{f_m\}_{m\in\N}$. Then
		\begin{itemize}
			\item[(a)] every compact subset of $\mathcal{X}$  is metrizable,
			\item[(b)] if $\{\nu_m\}_{m\in\N}$ is a tight sequence of probability measures on $(\mathcal{X}, \mathcal{S})$, then there exists a subsequence $\{m_k\}_{k\in\N}$, a probability space $(\Omega,\Fn,\Pr)$ with $\mathcal{X} - $valued Borel measurable random variables $\xi_k,\xi$ such that $\mu_{m_k}$ is the law of $\xi_k$ and $\xi_k$ converges to $\xi$ almost surely on $\Omega$. Moreover, the law of $\xi$ is a Radon measure.
		\end{itemize}
	\end{theorem}
	
	\begin{lemma}\label{lem-zt}
		The topological space $\mathcal{Z}_T$ defined in \eqref{def-zt} satisfies the assumptions of Theorem \ref{thm-law}.
	\end{lemma}
	\begin{proof}
		Our aim is to show that on every space present in the definition of the parent space $\mathcal{Z}_T$, see \eqref{def-zt}, we can find a countable set of separating continuous $\R-$valued functions. First, obsere that both the spaces $C([0, T]; L^2(\bO))$ and $L^2(0, T; H_0^1(\bO))$ are complete, metrizable as well as separable, the required condition is hold immediately, see \cite[expos\'e 8]{AB-06}. In the case of the space $L^2_w(0,T;D(A))$, we only need to define
		\begin{align*}
			f_m(u):=\int_0^T(- \Delta u(t), -\Delta \phi_m(t))dt\in\mathbb{R},\ u\in L^2_w(0,T;D(A)),\ m\in\N,
		\end{align*}
		provided $\{\phi_m\}_{m\in\N}$ is any dense sequence in $L^2(0,T;D(A))$. 
		Finally, suppose $\{\zeta_m\}_{m\in\N}$ is any dense subset of $\Vp$ and let $\mathbb{Q}_T$ denote the collection of all the rational numbers contained to the time interval $[0, T]$. Then, the collection countable $\{h_{m,t}\}_{m \in\N, t\in\mathbb{Q}_T}$ defined as follows
		\begin{align*}
			h_{m,t}(u):=\langle u(t),\zeta_m(t)\rangle\in\mathbb{R},\ u\in C_w([0,T]; \Vp),\ m\in\N,\ t\in\mathbb{Q}_T,
		\end{align*}
		consists the separating continuous functions in the space $C_w([0,T]; \Vp)$. Hence it concludes the proof.
	\end{proof}
	The following result can be obtained by applying Theorem \ref{thm-law} and Lemma \ref{lem-zt}. It will be used later to obtain a martingale solution to the SPDE \eqref{eqn-main-prob-Ito}.
	\begin{corollary}[{\cite[Corollary 4.11]{ZB+GD-21}}]\label{cor-Skoro}
		Let $\{\eta_m\}_{m\in\N}$ be a sequence of $\mathcal{Z}_T-$valued random variables such that their laws $\{\bL(\eta_m)\}_{m\in\N}$ on $(\mathcal{Z}_T, \mathcal{T})$ form a tight sequence of probability measures. Then, there exists a subsequence $\{m_k\}_{k\in\N}$, a probability space $(\wtilde{\Omega},\wtilde{\Fn},\wtilde{\Pr})$ and $\mathcal{Z}_T-$valued random variables $\wtilde{\eta},\wtilde{\eta}_{m_k},$ for all $ k\in\N$, such that the random variables $\eta_{m_k}$ and $\wtilde{\eta}_{m_k}$ have the same laws on $\mathcal{Z}_T$ and $\wtilde{\eta}_{m_k}$ converges to $\wtilde{\eta}$ almost surely on $\wtilde{\Omega}$.
	\end{corollary}
	
	
	\section{Existence of a martingale solution}\label{Sec-Exist}
	In this section, we establish the existence of a martingale solution to the problem \eqref{eqn-main-prob-Ito}. By using a modified Faedo-Galerkin approximation, we first obtain a-priori estimates to prove the tightness of laws induced by the solutions of the approximating equations \eqref{eqn-Faedo-Galerkin}. Then, we prove Proposition \ref{Prop-pre-martingale} with the help of Jakubowski-Skorokhod Representation Theorem. Thereafter, we establish the $L^p$ and $H_0^1-$It\^o formula, and conclude this section by demonstrating the proof of Theorem \ref{Thm-main}.
	
	\subsection{A modified Faedo-Galerkin scheme}
	The following construction is already used in the works \cite{FH-18, ZB+FH+UM-20, ZB+FH+LW-19, ZB+BF+MZ-24, AB+ZB+MTM-25+}. Let $S = -\Delta = A$, where $A$ is the Dirichlet Laplacian. Then, $S$ has compact resolvent. Precisely, there is an orthonormal basis $\{e_n\}_{n\in\N}$ of $L^2(\bO)$ and a nondecreasing sequence $\{\lambda_n\}_{n\in\N}$ with $\lambda_n>0$ and $\lambda_n\to\infty$ as $n\to\infty$  such that
	\begin{align}\label{eqn-s-rep}
		Su=\sum_{n=1}^{\infty}\lambda_n(u,e_n),e_n,\ \ u\in D(S):=\left\{x\in L^2(\bO):\sum_{n=1}^{\infty}\lambda_n^2|(x,e_n)|^2<\infty\right\}.
	\end{align}
	Moreover, $S$ is self-adjoint, strictly positive and commutes with $A$ as well as satisfies $D(S^k) \embed \Vp$ for sufficiently large $k$ ($k>\frac{d}{2}$). By the functional calculus, we define the following sequence of operators $P_m:L^2(\bO)\to L^2(\bO)$ as:
	\begin{align}\label{eqn-P_m-S}
		P_m:=\mathds{1}_{(0,2^{m+1})}(S)\ \text{ for
		}\ m\in\N_0=\N\cup\{0\}.
	\end{align} 
	Through the representation of $S$, see \eqref{eqn-s-rep}, the $P_m$ forms an orthogonal projection from $H$ to $L^2_m(\bO):=\mathrm{span}\left\{e_n:n\in\N, \lambda_n<2^{m+1}\right\},$  and
	\begin{align*}
		P_m u=\sum_{\lambda_n<2^{m+1}}(u,e_n)e_n,\ u\in L^2(\bO).
	\end{align*}
	Note that $e_n\in\bigcap_{k\in\N}D(S^k)$ for $n\in \N$. Since $ D(S) \embed H_0^1(\bO)$, it follows that $L^2_m(\bO)$ is a closed subspace of $H_0^1(\bO)$ for $m\in\N$.
	Using the fact that the operators $S$ and $A$ commute, one can immediately obtain that $P_m$ and $A^{\frac{1}{2}}$ also commute. Thus, we infer
	\begin{align*}
		\|P_mu\|_{L^2(\bO)}& \leq\|u\|_{L^2(\bO)}, \ u\in L^2(\bO),\\
		\|P_mu\|_{H_0^1(\bO)}^2& = \|A^{\frac{1}{2}}P_mu\|_{L^2(\bO)}^2=\|P_mA^{\frac{1}{2}}u\|_{L^2(\bO)}^2\leq\|A^{\frac{1}{2}}u\|_{L^2(\bO)}^2=\|u\|_{H_0^1(\bO)}^2,\ u\in H_0^1(\bO).
	\end{align*}
	Moreover, we have
	\begin{align}
		\lim_{m\to\infty}\|P_mu-u\|_{L^2(\bO)}=0,\  x\in L^2(\bO)\ \text{ and }\ \lim_{m\to\infty}\|P_mu-u\|_{H_0^1(\bO)}=0,  \ x\in H_0^1(\bO).
	\end{align}
	Unfortunately, the sequence of operators $\{P_m\}_{m\in\N}$, is not uniformly bounded from $L^p(\bO)$ to $L^p(\bO)$. This property is crucial in the proof of the a-priori estimates in the $L^p-$norm. To overcome this difficulty, in the following result produce a sequence $\{S_m\}_{m\in\N}$ that owns the required properties.
	
	\begin{proposition}[{\cite[Proposition 4.1]{AB+ZB+MTM-25+}}]\label{Prop-S_m}
		There exists a sequence $\{S_m\}_{m\in\N}$ of self-adjoint operators $S_m:L^2(\bO)\to L^2_m(\bO)$, for $m\in\N$ and $\psi\in \Vp$, such that $$S_m\psi\to \psi\ \text{ in }\ \Vp, \ \text{ as }\ m\to\infty$$ and the following uniform norm estimates hold:
		\begin{align}\label{eqn-L^p-S_m-cgs}
			\sup_{m\in\N}\|S_m\|_{\Ls(L^2(\bO))}\leq 1, \ 	\sup_{m\in\N}\|S_m\|_{\Ls(H_0^1(\bO))}\leq 1,\
			\sup_{m\in\N}\|S_m\|_{\Ls(L^p(\bO))}<\infty.
		\end{align}
	\end{proposition}
	
	\begin{remark}\label{Rmk-S_m-L^{2p-2}}
		We emphasize that the above result remains valid for $L^{2p-2}(\bO)$ also, i.e., the sequence of self-adjoint operators $\{S_m\}_{m\in\N}$ is bounded on $L^{2p-2}(\bO)$. Precisely, for $\psi \in L^{2p-2}(\bO)$, we have
		\begin{equation}\label{eqn-S_m-L^{2p-2}}
			S_m\psi\to \psi\ \text{ in }\ L^{2p-2}(\bO)\ \text{ as }\ m\to \infty \ \text{ and }\ \sup_{m\in\N}\|S_m\|_{\Ls(L^{2p-2}(\bO))}<\infty.
		\end{equation}
	\end{remark}
	
	\begin{proof}[Proof of Proposition \ref{Prop-S_m}]
		Let us select and fix $\gamma\in C_c^{\infty}(0,\infty)$ with the support contained in $[\frac{1}{2},2]$ with $$\sum_{m\in\mathbb{Z}}\gamma(2^{-m}t)=1, \ t>0.$$ 
		For fixed $m\in\N_0$, we define 
		\begin{align*}
			s_m:(0,\infty)\to\mathbb{R},\ \ s_m(x):=\sum_{n=-\infty}^m\gamma(2^{-n}x),
		\end{align*}
		i.e.,
		\begin{align*}
			s_m(x)=\left\{
			\begin{array}{cl}1 & x\in(0,2^m),\\
				\gamma(2^{-m}x) & x\in[2^m,2^{m+1}),\\
				0 & x\geq 2^{m+1}.
			\end{array}
			\right.
		\end{align*}
		With the help of the self-adjoint functional calculus, we define the operator $S_m := s_m(S)$. This, in turn, for $v\in L^2(\bO)$ yields the following representation:
		\begin{align}\label{def-S_m-op}
			S_m v := \sum_{\lambda_n<2^{m}}(v,w_n)w_n + \sum_{\lambda_n\in[2^m,2^{m+1})}\gamma(2^{-m}\lambda_n)(v, w_n)w_n,
		\end{align}
		that immediate implies that the range of $S_m$ is contained in $V_m$. 
		Rest of the proof follows via functional calculus’s convergence properties, for details we refer the reader to \cite[Proposition 4.1]{AB+ZB+MTM-25+}.
	\end{proof}
	
	\begin{remark}
		The definition of $P_m$ given in \eqref{eqn-P_m-S} asserts that $S_m$ represents a regularized version of indicator function $\mathds{1}_{(0,2^{m+1})}$. This observation is helpful in proving the uniform $L^p-$boundedness of $S_m$ via the Spectral Multiplier Theorems. The proof provided above can be considered as a motivation from \cite[Proposition 5.2]{ZB+BF+MZ-24}.
		On the other hand, in \cite{ZB+FH+LW-19}, the authors established this result using abstract Littlewood--Paley theory, which can be traced back to \cite[Proposition 10]{ZB+FH+UM-20}.
	\end{remark}
	
	\begin{proposition}[{\cite[Proposition 3.2]{LH-18}}]\label{Prop-P_m}
		The operators $P_m$ and $S_m$ satisfy the following properties: 
		\begin{enumerate}
			\item[(i)] $P_m$ is projection, i.e., $P_m^2 = P_m$ for all $n\in\N$.
			\item[(ii)] The operators $P_m, S_m$ are self-adjoint with $\norm{P_m}_{\Ls(L^2(\bO))} = \norm{S_m}_{\Ls(L^2(\bO))} = 1$ for every $m\in\N$.
			\item[(iii)] The ranges of $P_m$ and $S_m$ are finite dimensional.
			\item[(iv)] Also, $R(S_{m-1}) \subset R(P_m) \subset R(S_m)$
			and $P_mS_{m-1} = S_{m-1}$ for every $m\in\N$.
			\item[(v)] $\displaystyle\lim_{m\to\infty} P_m v = \lim_{m\to\infty} S_m v =v$ for every $v\in L^2(\bO)$.
		\end{enumerate}
	\end{proposition}
	
	Now, we consider a modified Faedo-Galerkin approximated system in the space $L^2_m(\bO)$:
	\begin{equation}\label{eqn-Faedo-Galerkin}
		\left\{\begin{aligned}
			du_m(t) & =\bigg( \Delta u_m(t) -P_m(|u_m(t)|^{p-2}u_m(t)) + \big( \norm{\nabla u_m(t)}_{L^2(\bO)}^2 + \norm{u_m(t)}_{L^p(\bO)}^p \big) u_m(t)\\ & \quad + \frac{1}{2}\sum_{i=1}^M \kappa_i^m(u_m(t))\bigg) dt+ \sum_{i=1}^M \Nn_i^m(u_m(t))dW_i(t),\\
			u_m(0) & = \frac{S_{m-1}u_0}{\|S_{m-1}u_0\|_{L^2(\bO)}},
		\end{aligned}\right.
	\end{equation}
	where $P_mu_m = u_m \in L^2_m(\bO)$, $\Nn_i^m = S_{m-1}\Nn_i$ and $\kappa_i^m = S_{m-1}\kappa_i$.
	
	\begin{remark}
		The ultimate reason of considering the self-adjoint operator $S_{m-1}$ acting on the stochastic terms and the initial data is that the projection operator $P_m$ and $S_{m-1}$ commute, as described in Proposition~\ref{Prop-P_m}, and that $S_{m-1}$ is bounded in $L^p(\bO)$, which are necessary for the analysis to complete the existence of a martingale solution.
	\end{remark}
	
	Since the nonlinear terms appearing in the above system of SDEs satisfy a locally Lipschitz condition, the Banach Fixed Point Theorem yields a local maximal solution $u_m$ up to some stopping time $\tau\leq T$ for the system \eqref{eqn-Faedo-Galerkin}. In the sequel, we intend to prove that this local solution is invariant under the manifold $\bM$, i.e., $u_m(t) \in \bM$ for all $t\in[0, \tau)$.

		\subsection{Energy estimates}
		In this subsection, we provide some a-priori energy estimates satisfied by the solutions to the problem \eqref{eqn-Faedo-Galerkin}.
		\begin{lemma}\label{Lem-manifold}
			For the sequence of stopping times
			\begin{align}\label{eqn-stop-1}
				\tau^k_m:=\inf\left\{t\in[0,T]:\|u_m(t)\|_{H_0^1(\bO)} + \|u_m(t)\|_{L^p(\bO)}\geq k\right\},
			\end{align}
			where $k\in\N$, the solution of the SDE \eqref{eqn-Faedo-Galerkin} lie on manifold $\bM$, i.e., $u_m(t\wedge\tau^k_m)\in\bM$ for all $t\in[0,T]$.
		\end{lemma}
		
		\begin{proof}
			Let the sequence of stopping times be defined as in \eqref{eqn-stop-1}.
			Let us choose and fix $t\in[0,T]$. Applying the finite dimensional It\^o's formula to the function $v \mapsto \frac{1}{2} \|v\|_{L^2(\bO)}^2$ and to the process $u_m$ to find $\Pr-$a.s.,
			\begin{align}
				\|u_m(t\wedge\tau^k_m)\|_{L^2(\bO)}^2
				& = \|u_m(0)\|_{L^2(\bO)}^2 + 2\int_0^{t\wedge\tau^k_m}  \big( u_m(s), \Delta u_m(s) -P_m(|u_m(s)|^{p-2}u_m(s))\\
				& \qquad\qquad \qquad\qquad\qquad\quad + \big( \norm{\nabla u_m(s)}_{L^2(\bO)}^2 + \norm{u_m(s)}_{L^p(\bO)}^p \big) u_m(s) \big)ds\\
				& \quad + \sum_{i=1}^M\int_0^{t\wedge\tau^k_m}(u_m(s),\kappa_i^m(u_m(s)))ds\\
				& \quad + \sum_{i=1}^M\int_0^{t\wedge\tau^k_m}(\Nn_i^m(u_m(s)),\Nn_i^m(u_m(s)))ds\\
				& \quad + 2\sum_{i=1}^M\int_0^{t\wedge\tau^k_m}(u_m(s),\Nn_i^m(u_m(s)))dW_i(s). \label{eqn-ener-equ}
			\end{align}
			Using the fact that $\|u_m(0)\|_{L^2(\bO)}^2=1$, $P_m u_m = u_m$  and the equations \eqref{eqn-est-c}, \eqref{eqn-est-b} and \eqref{eqn-est-a} in \eqref{eqn-ener-equ}, we obtain $\Pr-$a.s.,
			\begin{align}
				& \big(\|u_m(t\wedge\tau^k_m)\|_{L^2(\bO)}^2 - 1 \big)\\
				& = 2\int_0^{t\wedge\tau^k_m}  \big(\norm{\nabla u_m(s)}_{L^2(\bO)}^2 + \norm{ u_m(s)}_{L^p(\bO)}^p \big)( \norm{ u_m(s)}_{L^2(\bO)}^2 -1)ds\\
				& \quad + \sum_{i=1}^M \int_0^{t\wedge\tau^k_m} \big( -\|f_i\|_{L^2(\bO)}^2\|u_m(s)\|_{L^2(\bO)}^2 + |(f_i,u_m(s))|^2 \big(2\|u_m(s)\|_{L^2(\bO)}^2-1\big)\big) ds\\ 
				& \quad + \sum_{i=1}^M\int_0^{t\wedge\tau^k_m} \big(\|f_i\|_{L^2(\bO)}^2+|(f_i,u_m(s))|^2\big(\|u_m(s)\|_{L^2(\bO)}^2-2\big) \big) ds\\
				& \quad + 2\sum_{i=1}^M\int_0^{t\wedge\tau^k_m} \big( (f_i,u_m(s))\big(1-\|u_m(s)\|_{L^2(\bO)}^2\big) \big) dW_i(s)\\
				& = \int_0^{t\wedge\tau^k_m} \bigg( 2\big(\norm{\nabla u_m(s)}_{L^2(\bO)}^2 + \norm{ u_m(s)}_{L^p(\bO)}^p\big)\\
				& \qquad\qquad\qquad + \sum_{i=1}^M\big(-\|f_i\|_{L^2(\bO)}^2+3|(f_i,u_m(s))|^2 \big)\bigg) \big( \norm{ u_m(s)}_{L^2(\bO)}^2 -1 \big)ds\\ 
				& \quad - 2\sum_{i=1}^M \int_0^{t\wedge\tau^k_m} \big((f_i,u_m(s))\big(\|u_m(s)\|_{L^2(\bO)}^2-1\big) \big) dW_i(s).\label{eqn-ener-equ-1}
			\end{align}
			In order to simplify the above argument, for $t\in[0,T]$, we define
			\begin{align*}
				\varphi_m(t) 
				& := \big(\|u_m(t\wedge\tau^k_m)\|_{L^2(\bO)}^2-1 \big),\ \beta_m(t) := - 2 \sum_{i=1}^M (f_i,u_m(t\wedge\tau^k_m)),\\
				\alpha_m(t)
				& := 2\big(\norm{\nabla u_m(t\wedge\tau^k_m)}_{L^2(\bO)}^2 + \norm{ u_m(t\wedge\tau^k_m)}_{L^p(\bO)}^p \big)\\
				& \quad + \sum_{i=1}^M \big(- \|f_i\|_{L^2(\bO)}^2+3|(f_i,u_m(t\wedge\tau^k_m))|^2 \big),\\		
				F_m(t,\varphi_m(t)) 
				&
				:= \alpha_m(t)\varphi_m(t),\
				G_m(t,\varphi_m(t)) := \beta_m(t)\varphi_m(t),
			\end{align*}
			so that the equation \eqref{eqn-ener-equ-1} reduces to, for all $t\in[0,T]$, $\Pr-$a.s.
			\begin{equation}\label{eqn-new}
				\left\{
				\begin{aligned}
					\varphi_m(t\wedge\tau^k_m)
					& = \int_0^{t\wedge\tau^k_m} F_m(s,\varphi_m(s))ds+ \int_0^{t\wedge\tau^k_m} G_m(s,\varphi_m(s))dW(s),\\
					\varphi_m(0)
					& = \|u_m(0)\|_{L^2(\bO)}^2-1=0.
				\end{aligned}
				\right.
			\end{equation}
			To get the result $\|u_m(t\wedge\tau^k_m)\|_{L^2(\bO)}^2=1$ for all $t\in[0,T]$, we only need to show the existence and the uniqueness to the above-mentioned problem \eqref{eqn-new}.  If both $F_m$ and $G_m$ are Lipschitz in the second argument, then using \cite[Theorem 7.7]{ZB+TZ-99}, it is immediate that the problem \eqref{eqn-new} has a unique solution. 
			
			However, for all $\varphi,\psi\in \mathbb{R}$ and $t\in[0,T]$, it can be easily see that 
			\begin{align*}
				|F_m(t,\varphi) -F_m(t,\psi)| & = |\alpha_m(t,\omega)\varphi-\alpha_m(t,\omega)\psi| = |\alpha_m(t,\omega)||\varphi-\psi|,\\|G_m(t,\varphi) -G_m(t,\psi)| & = |\beta_m(t,\omega)\varphi-\beta_m(t,\omega)\psi| = |\beta_m(t,\omega)||\varphi-\psi|.
			\end{align*}
			Therefore, in order to prove that both the functions $F_m$ and $G_m$ are Lipschitz, it is suffices to establish the boundedness of the maps $\alpha_m$ and $\beta_m$. Considering $\beta_m(t,\omega)$ and using the definition of stopping time given in \eqref{eqn-stop-1}, we get
			\begin{align*}
				|\beta_m(t,\omega)| = 2|(f_i,u_m(t\wedge\tau^k_m,\omega))|\leq 2\|f_i\|_{L^2(\bO)}\|u_m(t\wedge\tau^k_m,\omega)\|_{L^2(\bO)}\leq \frac{2}{\sqrt{\lambda_1}}k\|f_i\|_{L^2(\bO)},
			\end{align*}
			where $k\in\N$, $\lambda_1$ is the first eigenvalue of the Dirichlet-Laplacian, and the boundedness of $\beta_m$ follows. Let us now consider  $\alpha(t,\omega)$ and use the definition of stopping time given in \eqref{eqn-stop-1} to deduce
			\begin{align}
				|\alpha_m(t,\omega)|
				& \leq 2\big(\norm{\nabla u_m(t\wedge\tau^k_m)}_{L^2(\bO)}^2 + \norm{ u_m(t\wedge\tau^k_m)}_{L^p(\bO)}^p \big)\\
				& \quad + \sum_{i=1}^M\|f_i\|_{L^2(\bO)}^2 + \sum_{i=1}^M 3\|f_i\|_{L^2(\bO)}^2\|u_m(t\wedge\tau^k_m)\|_{L^2(\bO)}^2\\ 
				& \leq 2\left(k^2+k^p\right) + \bigg(1+ \frac{3k^2}{\lambda_1}\bigg) \sum_{i=1}^M \|f_i\|_{L^2(\bO)}^2,\label{eqn-L^2-f_i}
			\end{align}
			where $k\in\N$, so that $\alpha_m$ is also bounded. Therefore, there exists a unique solution, say  $v_m$ of the  linear equation \eqref{eqn-new}. Since $\varphi_m(t) \equiv 0$, for all $t\in[0, T]$, is also a solution to \eqref{eqn-new},   by uniqueness, we conclude that  $v_m(t) = \varphi_m (t) = 0$  for evet $t\in [0, T] $, i.e., 
			\begin{align}\label{eqn-man}
				\|u_m(t\wedge\tau^k_m)\|_{L^2(\bO)}^2=1\ \text{ for every }\ t\in[0,T],
			\end{align}
			which completes the proof.
		\end{proof}
		
		\begin{remark}
			In Step V of the proof of Lemma~\ref{Lem-energy} below, we show that $t\wedge\tau_m^k\to t$ as $k\to\infty$. Consequently, passing to the limit $k\to\infty$ in \eqref{eqn-man} and using the fact that $u_m\in C([0,T];L^2_m(\bO))$ (as a consequence of Lemma \ref{Lem-energy}), we conclude that, for every $m\in\N$
			\begin{align}
				\|u_m(t)\|_{L^2(\bO)}^2=1\ \text{ for all }\ t\in[0,T].
			\end{align}  
			In particular, $u_m(t)\in\bM$ for all $t\in[0,T]$.
		\end{remark}
		
		\begin{lemma}\label{Lem-energy}
			Let us fix $T>0$, $p\in [2, \infty)$ and $u_0\in \Vp\cap \bM$. If $u_m$ denotes the unique solution of the SDE \eqref{eqn-Faedo-Galerkin}, then, for each fixed $m\in \N$, the following estimate holds true:
			\begin{align}
				& \sup_{m\in\N}\E \bigg[ \sup_{t\in[0,T]}\big(\|u_m(t)\|_{H_0^1(\bO)}^{4} + \|u_m(t)\|_{L^p(\bO)}^{2p}\big) \bigg] + \sup_{m\in\N}\E\left[\int_0^T\|{P_m(|u_m(t)|^{p-2}u_m(t))}\|_{L^{2}(\bO)}^{2}dt\right]\\
				& \quad + \sup_{m\in\N}\E\left[\int_0^T\|\Delta u_m(t)\|_{L^2(\bO)}^2dt\right] + 2(p-1)\sup_{m\in\N}\E\left[\int_0^T\||u_m(t)|^{\frac{p-2}{2}} \nabla u_m(t)\|_{L^2(\bO)}^2dt\right]\\
				& \leq C\left(\|u_0\|_{\Vp},T,  \sum_{i=1}^M \|f_i\|_{\Vp}\right).\label{energy-estimate}
			\end{align}
		\end{lemma}
		\begin{proof}
			We divide the proof into ten steps: in the first five steps, we derive estimates for the lower-order moments, while the remaining five steps are devoted to estimates for the higher-order moments.
			
			\vspace{2mm}
			\noindent
			\textbf{Step I.} To begin, we fix $m\in\N$ and suppose $u_m$ is the solution to the SDE \eqref{eqn-Faedo-Galerkin}, then by an application of the finite dimensional It\^o formula to the function
			\[
			L^2_m\ni v \mapsto \frac{1}{2}\|\nabla v\|_{L^2(\bO)}^2 \in \mathbb{R},
			\]
			and to the process $u_m$ to obtain for all $t\in[0,T] $, $\Pr-$a.s.,
			\begin{align}
				& \frac{1}{2} \|u_m(t\wedge\tau_m^k)\|_{H_0^1(\bO)}^2\\
				& = \frac{1}{2} 	\|u_m(0)\|_{H_0^1(\bO)}^2 + \int_0^{t\wedge\tau_m^k}\big(-\Delta u_m(s),\Delta u_m(s) - {P_m}(|u_m(s)|^{p-2}u_m(s)) \big)ds\\
				& \quad + \int_0^{t\wedge\tau_m^k}\big(-\Delta u_m(s), \big( \norm{\nabla u_m(s)}_{L^2(\bO)}^2 + \norm{u_m(s)}_{L^p(\bO)}^p \big) u_m(s)\big)ds\\
				& \quad + \frac{1}{2}\sum_{i=1}^M\int_0^{t\wedge\tau_m^k}(-\Delta u_m(s),\kappa_i^m(u_m(s)))ds + \frac{1}{2}\sum_{i=1}^M \int_0^{t\wedge\tau_m^k}(-\Delta \Nn_i^m(u_m(s)),\Nn_i^m(u_m(s)))ds\\
				& \quad + \sum_{i=1}^M \int_0^{t\wedge\tau_m^k}(-\Delta u_m(s),\Nn_i^m(u_m(s)))dW_i(s).\label{eqn-est-u-1}
			\end{align}
			Again, we apply the finite dimensional It\^o's formula to the function
			\[
			L^2_m\ni v \mapsto \frac{1}{p}\|v\|_{L^p(\bO)}^p \in \mathbb{R}.
			\]
			This function is of $C^2-$class because $p \geq 2$ and $L^2_m \subset L^p$. We use the formulae for the derivatives of this function, as in, e.g., \cite{ZB+SP-01}, and apply the formula to the process $u_m$ to derive, for all $t\in[0,T] $, $\Pr-$a.s.,
			\begin{align}
				& \frac{1}{p} \|u_m(t\wedge\tau_m^k)\|_{L^p(\bO)}^p\\
				& = \frac{1}{p}\|u_m(0)\|_{L^p(\bO)}^p + \int_0^{t\wedge\tau_m^k}\big(|u_m(s)|^{p-2}u_m(s),\Delta u_m(s) - {P_m}(|u_m(s)|^{p-2}u_m(s)) \big)ds\nonumber\\
				\nonumber & \quad + \int_0^{t\wedge\tau_m^k}\big(|u_m(s)|^{p-2}u_m(s), \big( \norm{\nabla u_m(s)}_{L^2(\bO)}^2 + \norm{u_m(s)}_{L^p(\bO)}^p \big) u_m(s)\big)ds\\
				& \quad + \frac{1}{2}\sum_{i=1}^M\int_0^{t\wedge\tau_m^k}(|u_m(s)|^{p-2}u_m(s),\kappa_i^m(u_m(s)))ds\\
				&\quad + \frac{p-1}{2}\sum_{i=1}^M \int_0^{t\wedge\tau_m^k}(|u_m(s)|^{p-2} \Nn_i^m(u_m(s)),\Nn_i^m(u_m(s)))ds\\
				& \quad + \sum_{i=1}^M \int_0^{t\wedge\tau_m^k}(|u_m(s)|^{p-2}u_m(s),\Nn_i^m(u_m(s)))dW_i(s).\label{eqn-est-u-2}
			\end{align}
			\textbf{Step II.} Since $\|u_m(t\wedge\tau^k_m)\|_{L^2(\bO)}^2=1\ \text{ for all }\ t\in[0,T]$, see Lemma \ref{Lem-manifold}, it follows that
			\begin{align}
				& - \big\| - \Delta u_m +P_m(|u_m|^{p-2}u_m) -\big(\|\nabla u_m\|_{L^2(\bO)}^2+ \|u_m\|_{L^p(\bO)}^p \big)u_m\big\|_{L^2(\bO)}^2\\
				& =-\big\| - \Delta u_m +P_m(|u_m|^{p-2}u_m) \big\|_{L^2(\bO)}^2 - \big\|\big(\|\nabla u_m\|_{L^2(\bO)}^2+ \|u_m\|_{L^p(\bO)}^p \big)u_m\big\|_{L^2(\bO)}^2\\
				& \quad + 2\big(-\Delta u_m +P_m(|u_m|^{p-2}u_m),\big(\|\nabla u_m\|_{L^2(\bO)}^2+ \|u_m\|_{L^p(\bO)}^p\big)u_m \big)\\
				& = - \big\| - \Delta u_m +P_m(|u_m|^{p-2}u_m) \big\|_{L^2(\bO)}^2 + \big(\|\nabla u_m\|_{L^2(\bO)}^2+ \|u_m\|_{L^p(\bO)}^p\big)^2 \label{eqn-L^2-relation}\\
				\nonumber& {= -\big(-\Delta u_m + P_m(|u_m|^{p-2}u_m), -\Delta u_m + P_m(|u_m|^{p-2}u_m) - \big(\|\nabla u_m\|_{L^2(\bO)}^2+ \|u_m\|_{L^p(\bO)}^p\big)u_m\big)}.
			\end{align}
			Therefore, by utilizing the above identity in the sum of identities \eqref{eqn-est-u-1} and \eqref{eqn-est-u-2}, we obtain
			\begin{align}
				& \frac{1}{2}\|u_m(t\wedge\tau_m^k)\|_{H_0^1(\bO)}^2+ \frac{1}{p}	\|u_m(t\wedge\tau_m^k)\|_{L^p(\bO)}^p + \int_0^{t\wedge\tau_m^k} \big\| - \Delta u_m(s) + P_m(|u_m(s)|^{p-2}u_m(s))\\
				&\quad\hspace*{6cm} - \big(\|\nabla u_m(s)\|_{L^2(\bO)}^2+ \|u_m(s)\|_{L^p(\bO)}^p\big)u_m(s)\big\|_{L^2(\bO)}^2ds\\
				& = \frac{1}{2} \|u_m(0)\|_{H_0^1(\bO)}^2 + \frac{1}{p}\|u_m(0)\|_{L^p(\bO)}^p \\ 
				& \quad + \frac{1}{2}\sum_{i=1}^M \int_0^{t\wedge\tau_m^k}\big(-\Delta u_m(s) + {P_m}(|u_m(s)|^{p-2}u_m(s)), \kappa_i^m(u_m(s))\big)ds\\
				& \quad + \frac{1}{2}\sum_{i=1}^M\int_0^{t\wedge\tau_m^k}\big(-\Delta \Nn_i^m(u_m(s)) + ({p-1})|u_m(s)|^{p-2}\Nn_i^m(u_m(s)), \Nn_i^m(u_m(s))\big)ds\\
				& \quad + \sum_{i=1}^M \int_0^{t\wedge\tau_m^k}\big(-\Delta u_m(s) +|u_m(s)|^{p-2}u_m(s), \Nn_i^m(u_m(s))\big)dW_i(s) \\
				& =: \frac{1}{2} \|u_m(0)\|_{H_0^1(\bO)}^2 + \frac{1}{p}\|u_m(0)\|_{L^p(\bO)}^p + I_1 + I_2\\
				& \quad + \sum_{i=1}^M \int_0^{t\wedge\tau_m^k}\big(-\Delta u_m(s) +|u_m(s)|^{p-2}u_m(s), \Nn_i^m(u_m(s))\big)dW_i(s),\label{eqn-est-u-5}
			\end{align}
			where we have utilized Proposition \ref{Prop-P_m} (iv), i.e., $P_mS_{m-1} = S_{m-1}$, and the self-adjointness of the projection operator $P_m$.
			
			\vspace{2mm}
			\noindent
			\textbf{Step III.} Let us now estimate $I_i$'s for each $i=1,2,3$ present in the equality \eqref{eqn-est-u-5}. First, note from \eqref{eqn-est-e} that
			\begin{align*}
				(-\Delta u_m, \kappa_i^m(u_m)) & = \big( 2|(f_i, u_m)|^2 - \|f_i \|_{L^2(\bO)}^2 \big) (-\Delta u_m, S_{m-1} u_m)  - (f_i, u_m) (-\Delta u_m, S_{m-1} f)\\
				& \le 3\|f_i \|_{L^2(\bO)}^2 |(\nabla u_m, S_{m-1} \nabla u_m)| + |(f_i, u_m)| |( \nabla u_m, S_{m-1}\nabla f)|\\
				& \le 3\|f_i\|_{L^2(\bO)}^2\|u_m\|_{H_0^1(\bO)}^2 + \|f_i\|_{L^2(\bO)} \| f_i\|_{H_0^1(\bO)} \|u_m\|_{H_0^1(\bO)}\\ 
				& \le \frac{7}{2} \|f_i\|_{L^2(\bO)}^2 \| u_m\|_{H_0^1(\bO)}^2 + \frac{1}{2}\|f_i\|_{H_0^1(\bO)}^2,
			\end{align*}
			where we have applied the H\"older's and Young's inequalities. Again, by the identity \eqref{eqn-est-h} and the H\"older inequality (with exponents $\frac{p}{p-1}$ and $p$), we deduce
				\begin{align*}
					&(|u_m|^{p-2}u_m, \kappa_i^m(u_m))\\ 
					& = \big(2 |(f_i,u_m)|^2 - \|f_i\|_{L^2(\bO)}^2\big) (|u_m|^{p-2} u_m, S_{m-1} u_m) - (f_i, u_m)(|u_m|^{p-2}u_m, S_{m-1} f_i) \\
					& \le 3\|f_i\|_{L^2(\bO)}^2 |(|u_m|^{p-2}u_m, S_{m-1} u_m)|  + |(f_i, u_m)| |(|u_m|^{p-2}u_m, S_{m-1} f_i)|\\
					& \le 3\|f_i\|_{L^2(\bO)}^2 \|u_m\|_{L^p(\bO)}^{p-1} \|S_{m-1}\|_{\Ls(L^p(\bO))} \|u_m\|_{L^p(\bO)}\\
					&\quad + \|f_i\|_{L^{2}(\bO)} \|u_m\|_{L^p(\bO)}^{p-1} \|S_{m-1}\|_{\Ls(L^p(\bO))} \|f_i\|_{L^p(\bO)}\\
					& \le 3\|f_i\|_{L^2(\bO)}^2 \|S_{m-1}\|_{\Ls(L^p(\bO))} \|u_m\|_{L^p(\bO)}^p\\
					&\quad +  \frac{1}{p}\|f_i\|_{L^{p}(\bO)}^p + \frac{p-1}{p}\|S_{m-1}\|_{\Ls(L^p(\bO))}^{\frac{p}{p-1}}\|f_i\|_{L^2(\bO)}^{\frac{p}{p-1}}\|u_m\|_{L^p(\bO)}^{p} \\
					& \le  \bigg(3\|f_i\|_{L^2(\bO)}^2\|S_{m-1}\|_{\Ls(L^p(\bO))} + \frac{p-1}{p}\|f_i\|_{L^2(\bO)}^{\frac{p}{p-1}}\|S_{m-1}\|_{\Ls(L^p(\bO))}^{\frac{p}{p-1}}\bigg) \|u_m\|_{L^p(\bO)}^p + \frac{1}{p} \|f_i\|_{L^{p}(\bO)}^p.
			\end{align*}
			Therefore, by clubbing the above two bounds, we estimate 
			\begin{align}
				&I_1 =  \frac{1}{2}\sum_{i=1}^M \int_0^{t\wedge\tau_m^k}\big(-\Delta u_m(s) + P_m(|u_m(s)|^{p-2}u_m(s)), \kappa_i^m(u_m(s))\big) ds\\ 
				& \le \frac{T}{2} \sum_{i=1}^M \bigg(\frac{1}{2}\|f_i\|_{H_0^1(\bO)}^2 + \frac{1}{p} \|f_i\|_{L^p(\bO)}^p\bigg)	\\
				&\quad +  \sum_{i=1}^M C^i_1 \int_0^{t\wedge\tau_m^k}\big( \|u_m(s)\|_{H_0^1(\bO)}^2 + \|u_m(s)\|_{L^p(\bO)}^{p}\big)ds,\label{eqn-est-u-7}
			\end{align}
			where 
			\begin{equation}\label{eqn-C^i_1}
					C^i_1 := \max\bigg\{ \frac{7}{4} \|f_i\|_{L^2(\bO)}^2,\frac{3}{2}\|f_i\|_{L^2(\bO)}^2\|S_{m-1}\|_{\Ls(L^p(\bO))} + \frac{p-1}{2p}\|f_i\|_{L^2(\bO)}^{\frac{p}{p-1}}\|S_{m-1}\|_{\Ls(L^p(\bO))}^{\frac{p}{p-1}} \bigg\}.
			\end{equation}
			On the other hand, by using the equality \eqref{eqn-est-f} and, H\"older's and Young's inequalities, we find
			\begin{align}
				(-\Delta  \Nn_i^m(u_m), \Nn_i^m(u_m)) & = \|\nabla \Nn_i^m(u_m)\|_{L^2(\bO)}^2 =  \|\nabla S_{m-1}(f_i - (f_i, u_m)u_m)\|_{L^2(\bO)}^2\\
				& \le \|S_{m-1}\|_{\Ls(L^2(\bO))}^2 \big(|(f_i, u_m)|^2\|u_m\|_{H_0^1(\bO)}^2 + \|f_i\|_{H_0^1(\bO)}^2 \big)\\
				&\quad + 2 |(f_i, u_m)| |(S_{m-1}\nabla f_i, S_{m-1}\nabla u_m)| \\
				& \le \|f_i\|_{L^2(\bO)}^2 \|u_m\|_{H_0^1(\bO)}^2 + \|f_i\|_{H_0^1(\bO)}^2 \\
				&\quad + 2 \|S_{m-1}\|_{\Ls(L^2(\bO))}^2 \|f_i\|_{L^2(\bO)} \| f_i\|_{H_0^1(\bO)} \| u_m\|_{H_0^1(\bO)} \\
				& \le 2\|f_i\|_{L^2(\bO)}^2 \|u_m\|_{H_0^1(\bO)}^2 + 2\|f_i\|_{H_0^1(\bO)}^2,
			\end{align}
			where in the last line above we have employed the fact that $\|S_{m-1}\|_{\Ls(L^2(\bO))} \le 1$, see Proposition \ref{Prop-S_m}.
			Similar to \eqref{eqn-est-i}, we use the H\"older inequality (with exponents $(\frac{p}{p-2},p,p)$, $(\frac{p}{p-1},p)$, and $(\frac{p}{p-2},\frac{p}{2})$) to produce
			\begin{align}
				& (|u_m|^{p-2}\Nn_i^m(u_m), \Nn_i^m(u_m))\\
				& = (|u_m|^{p-2} \big(S_{m-1}f_i - (f_i, u_m) S_{m-1}u_m \big), S_{m-1}f_i - (f_i, u_m) S_{m-1} u_m)\\ 
				& = (|u_m|^{p-2} S_{m-1}f_i, S_{m-1} f_i ) - (f_i, u_m )(|u_m|^{p-2}S_{m-1} f_i , S_{m-1}u_m)\\
				&\quad - (f_i, u_m )(|u_m|^{p-2}S_{m-1}u_m, S_{m-1}f_i) +|(f_i ,u_m)|^2 ( |u_m|^{p-2}S_{m-1}u_m, S_{m-1} u_m)\\
				& \le \|S_{m-1}\|_{\Ls(L^p(\bO))}^2 \|f_i\|_{L^p(\bO)}^2 \|u_m\|_{L^p(\bO)}^{p-2}  +  2 \|f_i\|_{L^{2}(\bO)} \|u_m\|_{L^p(\bO)}^{p-1} \|S_{m-1}\|_{\Ls(L^p(\bO))}^2 \|f_i\|_{L^p(\bO)}\\
				&\quad + \|f_i\|_{L^2(\bO)}^2 \|S_{m-1}\|_{\Ls(L^p(\bO))}^2 \|u_m\|_{L^p(\bO)}^{p}\\
				& \le  \frac{2}{p}\big( \|f_i\|_{L^p(\bO)}^{\frac{p}{2}} + \|f_i\|_{L^p(\bO)}^{p} \big) + \|f_i\|_{L^2(\bO)}^2 \|S_{m-1}\|_{\Ls(L^p(\bO))}^2 \|u_m\|_{L^p(\bO)}^{p}\\
				&\quad +   \bigg( \frac{p-2}{p} \|S_{m-1}\|_{\Ls(L^p(\bO))}^{\frac{2p}{p-2}} \|f_i\|_{L^p(\bO)}^{\frac{p}{p-2}} + \frac{2p-2}{p} \|S_{m-1}\|_{\Ls(L^p(\bO))}^{\frac{2p}{p-1}}\|f_i\|_{L^2(\bO)}^{\frac{p}{p-1}} \bigg) \|u_m\|_{L^p(\bO)}^{p},\label{eqn-L^p-f_i}
			\end{align}
			where in the last step we have used Young's inequality. Thus, by utilizing the above two estimates, we deduce
			\begin{align}
				I_2 & = \frac{1}{2}\sum_{i=1}^M\int_0^{t\wedge\tau_m^k}\big(-\Delta \Nn_i^m(u_m(s)) + ({p-1})|u_m(s)|^{p-2}\Nn_i^m(u_m(s)), \Nn_i^m(u_m(s))\big)ds\\
				& = \frac{1}{2} \sum_{i=1}^M \int_0^{t\wedge\tau_m^k}\big( \big(-\Delta S_{m-1} \Nn_i(u_m(s)), S_{m-1}\Nn_i(u_m(s))\big)\\ 
				& \qquad\qquad\qquad\quad + (p-1) \big(|u_m(s)|^{p-2}S_{m-1}\Nn_i(u_m(s)), S_{m-1}\Nn_i(u_m(s))\big) \big)ds\\
				& \le  \frac{1}{2}\sum_{i=1}^M \int_0^{t\wedge\tau_m^k} \bigg( 2\| f_i\|_{H_0^1(\bO)}^2  + \frac{2}{p}\big( \|f_i\|_{L^p(\bO)}^{\frac{p}{2}} + \|f_i\|_{L^p(\bO)}^{p} \big) + 2\|f_i\|_{L^2(\bO)}^2 \|u_m(s)\|_{H_0^1(\bO)}^2\\
				&\qquad\qquad + \bigg( \|f_i\|_{L^2(\bO)}^2 \|S_{m-1}\|_{\Ls(L^p(\bO))}^2 + \frac{p-2}{p} \|S_{m-1}\|_{\Ls(L^p(\bO))}^{\frac{2p}{p-2}} \|f_i\|_{L^p(\bO)}^{\frac{p}{p-2}}\\
				&\qquad\qquad\quad + \frac{2p-2}{p} \|S_{m-1}\|_{\Ls(L^p(\bO))}^{\frac{2p}{p-1}}\|f_i\|_{L^2(\bO)}^{\frac{p}{p-1}} \bigg) \|u_m(s)\|_{L^p(\bO)}^p \bigg) ds\\
				& \leq \sum_{i=1}^M \bigg( \| f_i\|_{H_0^1(\bO)}^2  + \frac{1}{p}\big( \|f_i\|_{L^p(\bO)}^{\frac{p}{2}} + \|f_i\|_{L^p(\bO)}^{p} \big) \bigg)T\\
				& \quad +\sum_{i=1}^M C^i_2 \int_0^{t\wedge\tau_m^k} \big(\|u_m(s)\|_{H_0^1(\bO)}^2+ \|u_m(s)\|_{L^p(\bO)}^p \big) ds
				,\label{eqn-est-u-8}
			\end{align}
			where we have utilized the fact that $\|u_m (s\wedge\tau_m^k)\|_{L^2(\bO)}^2=1$ for all $s\in [0,T]$, see Lemma \ref{Lem-manifold}, and 
			\begin{align}
				C^i_2 & := \max\bigg\{ \|f_i\|_{L^2(\bO)}^2 ,  \frac{\|S_{m-1}\|_{\Ls(L^p(\bO))}^2}{2} \|f_i\|_{L^2(\bO)}^2 + \frac{p-2}{2p} \|S_{m-1}\|_{\Ls(L^p(\bO))}^{\frac{2p}{p-2}} \|f_i\|_{L^p(\bO)}^{\frac{p}{p-2}} \\
				&\qquad\qquad\qquad\qquad\quad + \frac{p-1}{p} \|S_{m-1}\|_{\Ls(L^p(\bO))}^{\frac{2p}{p-1}}\|f_i\|_{L^2(\bO)}^{\frac{p}{p-1}} \bigg\}.\label{eqn-C^i_2}
			\end{align}
			\begin{remark}
				Observe that, when $M= \infty$, we require the series $\sum_{i=1}^M\|f_i\|_{L^p(\bO)}\break <  \infty$ to converge\dela{ and $p>2$}.
			\end{remark}
			
			\vskip 1mm
			\noindent
			\textbf{Step IV.} Therefore, by utilizing the estimates \eqref{eqn-est-u-7} and \eqref{eqn-est-u-8} in \eqref{eqn-est-u-5}, we deduce
			\begin{align}
				&	\frac{1}{2}\|u_m(t\wedge\tau_m^k)\|_{H_0^1(\bO)}^2+ \frac{1}{p}	\|u_m(t\wedge\tau_m^k)\|_{L^p(\bO)}^p + \int_0^{t\wedge\tau_m^k} \big\| - \Delta u_m(s) + P_m(|u_m(s)|^{p-2}u_m(s))\\
				&\quad - \big(\|\nabla u_m(s)\|_{L^2(\bO)}^2+ \|u_m(s)\|_{L^p(\bO)}^p\big)u_m(s)\big\|_{L^2(\bO)}^2ds\\
				\nonumber& \leq \frac{1}{2} \|u_m(0)\|_{H_0^1(\bO)}^2+ \frac{1}{p}\|u_m(0)\|_{L^p(\bO)}^p + \sum_{i=1}^M \bigg( \frac{5}{4} \|f_i\|_{H_0^1(\bO)}^2 +  \frac{1}{p} \|f_i\|_{L^p(\bO)}^{\frac{p}{2}} + \frac{3}{2p}\|f_i\|_{L^p(\bO)}^{p}  \bigg) T\\
				& \quad + \sum_{i=1}^M (C^i_1 + C^i_2) \int_0^{t\wedge\tau_m^k}\big(\|u_m(s)\|_{H_0^1(\bO)}^2 + \|u_m(s)\|_{L^p(\bO)}^p\big)ds\\
				& \quad + \sum_{i=1}^M \int_0^{t\wedge\tau_m^k}\big(-\Delta u_m(s) +|u_m(s)|^{p-2}u_m(s),\Nn_i^m(u_m(s))\big)dW_i(s). \label{eqn-est-u-9}
			\end{align}
			Further, by using the fact that $\frac{1}{p} \le \frac{1}{2}$, for all $p\in [2,\infty)$, on the left-hand side of identity \eqref{eqn-est-u-9}, we assert
			\begin{align}
				&	\frac{1}{p}\big(\|u_m(t\wedge\tau_m^k)\|_{H_0^1(\bO)}^2+ 	\|u_m(t\wedge\tau_m^k)\|_{L^p(\bO)}^p\big) + \frac{1}{2}\int_0^{t\wedge\tau_m^k} \big\| - \Delta u_m(s) + P_m(|u_m(s)|^{p-2}u_m(s)) \\
				&\quad - \big(\|\nabla u_m(s)\|_{L^2(\bO)}^2+ \|u_m(s)\|_{L^p(\bO)}^p\big)u_m(s)\big\|_{L^2(\bO)}^2d s\\
				& \leq \frac{1}{2} \|u_m(0)\|_{H_0^1(\bO)}^2+ \frac{1}{p}\|u_m(0)\|_{L^p(\bO)}^p + \sum_{i=1}^M \bigg( \frac{5}{4} \|f_i\|_{H_0^1(\bO)}^2 +  \frac{1}{p} \|f_i\|_{L^p(\bO)}^{\frac{p}{2}} + \frac{3}{2p}\|f_i\|_{L^p(\bO)}^{p}  \bigg) T\\
				& \quad +  \sum_{i=1}^M\big( C^i_1 \dela{C_{\max}(\|f_i\|_{L^2(\bO)}^2+1)} + C^i_2 \big) \int_0^{t\wedge\tau_m^k} \big(\|u_m(s)\|_{H_0^1(\bO)}^2 + \|u_m(s)\|_{L^p(\bO)}^p\big)ds\\
				& \quad + \sum_{i=1}^M \int_0^{t\wedge\tau_m^k}\big(-\Delta u_m(s) +|u_m(s)|^{p-2}u_m(s), \Nn_i^m(u_m(s))\big)dW_i(s), \label{eqn-est-u-10}
			\end{align}
			where $C_{\max} := \max\{1, C(p-1)\}$.
			Since we know that, for each $1\le i \le M$,
			$$\eta_i(t\wedge\tau_m^k) = \int_0^{t\wedge\tau_m^k} \big(-\Delta u_m(s) +|u_m(s)|^{p-2}u_m(s), \Nn_i^m(u_m(s))\big)dW_i(s)$$
			is a martingale and hence $\E[\eta_i(t\wedge\tau_m^k)]=0$. Taking expectation in \eqref{eqn-est-u-10} and using the above fact yield
			\begin{align}
				& \frac{1}{p}\E \left[ \|u_m(t\wedge\tau_m^k)\|_{H_0^1(\bO)}^2+ 	\|u_m(t\wedge\tau_m^k)\|_{L^p(\bO)}^p \right] + \E\bigg[\int_0^{t\wedge\tau_m^k} \big\| - \Delta u_m(s)\\
				& + {P_m}(|u_m(s)|^{p-2}u_m(s)) - \big(\|\nabla u_m(s)\|_{L^2(\bO)}^2+ \|u_m(s)\|_{L^p(\bO)}^p\big)u_m(s)\big\|_{L^2(\bO)}^2ds\bigg]\\
				\nonumber & \leq  \frac{1}{2}\|u_m(0)\|_{H_0^1(\bO)}^2 + \frac{1}{p}\|u_m(0)\|_{L^p(\bO)}^p  + \sum_{i=1}^M \bigg( \frac{5}{4} \|f_i\|_{H_0^1(\bO)}^2 +  \frac{1}{p} \|f_i\|_{L^p(\bO)}^{\frac{p}{2}} + \frac{3}{2p}\|f_i\|_{L^p(\bO)}^{p}  \bigg) T\\
				& \quad + \sum_{i=1}^M \big( C^i_1 \dela{C_{\max}(\|f_i\|_{L^2(\bO)}^2+1)} + C^i_2 \big) \int_0^{t}\E\big[\|u_m(s\wedge\tau_m^k)\|_{H_0^1(\bO)}^2 + \|u_m(s\wedge\tau_m^k)\|_{L^p(\bO)}^p\big]ds. \label{eqn-est-u-11}
			\end{align}
			An application of Gronwall's inequality in \eqref{eqn-est-u-11} gives for all $t\in[0,T]$
			\begin{align}
				& \frac{1}{p} \E\left[\|u_m(t\wedge\tau_m^k)\|_{H_0^1(\bO)}^2+ 	\|u_m(t\wedge\tau_m^k)\|_{L^p(\bO)}^p\right] + \E\bigg[\int_0^{t\wedge\tau_m^k}	\big\| - \Delta u_m(s)\\
				& + {P_m}(|u_m(s)|^{p-2}u_m(s)) - \big(\|\nabla u_m(s)\|_{L^2(\bO)}^2+ \|u_m(s)\|_{L^p(\bO)}^p\big)u_m(s)\big\|_{L^2(\bO)}^2ds\bigg]\\
				& \leq \bigg( \frac{1}{2}\|u_0\|_{H_0^1(\bO)}^2 + \frac{1}{p}\|u_0\|_{L^p(\bO)}^p + \sum_{i=1}^M\bigg( \frac{5}{4} \|f_i\|_{H_0^1(\bO)}^2 +  \frac{1}{p} \|f_i\|_{L^p(\bO)}^{\frac{p}{2}} + \frac{3}{2p}\|f_i\|_{L^p(\bO)}^{p} \bigg)T\bigg)\\
				& \quad \times\exp\left(\sum_{i=1}^M \big( C^i_1 \dela{C_{\max}(\|f_i\|_{L^2(\bO)}^2+1)} + C^i_2 \big)T \right)\\
				& \leq C\Big(\|u_0\|_{\Vp},  T, \sum_{i=1}^M \|f_i\|_{\Vp} \Big). \label{eqn-est-u-12}
			\end{align}
			
			\vspace{2mm}
			\noindent
			\textbf{Step V.} Let us observe for the indicator function $\chi$ that
			$$\E\left[{\chi}_{\{\tau_m^k<t\}}\right]=\Pr\big\{\omega\in\Omega:\tau_m^k(\omega)<t\big\},$$
			and using the definition of {the stopping time defined in} \eqref{eqn-stop-1}, we obtain
			\begin{align*}
				\E\left[ \|u_m(t\wedge\tau_m^k)\|_{H_0^1(\bO)}^2 + \|u_m(t\wedge\tau_m^k)\|_{L^p(\bO)}^p \right]
				& = \E\left[ \big( \|u_m(\tau_m^k)\|_{H_0^1(\bO)}^2+ \|u_m(\tau_m^k)\|_{L^p(\bO)}^p \big){\chi}_{\{\tau_m^k<t\}}\right]\\
				& \quad + \E\left[\big( \|u_m(t)\|_{H_0^1(\bO)}^2+ \|u_m(t)\|_{L^p(\bO)}^p\big) {\chi}_{\{\tau_m^k\geq t\}}\right]\nonumber\\
				& \geq
				\E\left[ \big( \|u_m(\tau_m^k)\|_{H_0^1(\bO)}^2+ 	\|u_m(\tau_m^k)\|_{L^p(\bO)}^p\big){\chi}_{\{\tau_m^k<t\}}\right]\\
				& \geq
				( k^2 + k^p )\Pr\big\{\omega\in\Omega:\tau_m^k<t\big\}.
			\end{align*}
			Using the energy bound (\ref{eqn-est-u-12}), we deduce
			\begin{align*}
				\Pr\big\{\omega\in\Omega:\tau_m^k<t\big\}
				& \leq \frac{1}{\left(k^2 + k^p\right)} \E\left[ \|\nabla u_m(t\wedge\tau_m^k)\|_{L^2(\bO)}^2 +  \| u_m(t\wedge\tau_m^k)\|_{L^p(\bO)}^p \right]\\ 
				& \leq
				\frac{1}{ ( k^2 + k^p )} C\Big(\|u_0\|_{\Vp},  T, \sum_{i=1}^M \|f_i\|_{\Vp} \Big).
			\end{align*}
			Therefore, for each $m\in\N$, we have
			\begin{align*}
				\lim_{k\to\infty}\Pr\big\{\omega\in\Omega:\tau_m^k<t\big\}=0, \ \textrm{
					for all }\ t\in [0,T],
			\end{align*}
			and \[t\wedge\tau_m^k\to t\ \text{ as } \ k\to\infty.\]
			Passing limit $k\to\infty$ in
			\eqref{eqn-est-u-12} and utilizing the {Monotone Convergence Theorem}, we obtain
			\begin{align}
				& \sup_{m\in\N}\sup_{t\in[0,T]}	\E\left[ \|u_m(t)\|_{H_0^1(\bO)}^2+	\|u_m(t)\|_{L^p(\bO)}^p\right]	\\
				& \quad + {\frac{p}{2}}\sup_{m\in\N}\E\bigg[\int_0^{T}	\big\| - \Delta u_m(s) + {P_m}(|u_m(s)|^{p-2}u_m(s)) \\
				& \qquad\qquad\qquad\qquad\quad - \big(\|\nabla u_m(s)\|_{L^2(\bO)}^2 + \|u_m(s)\|_{L^p(\bO)}^p \big)u_m(s)\big\|_{L^2(\bO)}^2ds\bigg]\\
				& \leq C\Big(\|u_0\|_{\Vp},  T, \sum_{i=1}^M \|f_i\|_{\Vp} \Big).\label{eqn-est-u-13}
			\end{align}
			Since $t\wedge\tau_m^k\to t$ as $k\to\infty$, it follows, by taking limit $k\to\infty$ in \eqref{eqn-man}, for all $m\in\N$, that
			\begin{align}\label{eqn-u_m-L^2=1}
				\|u_m(t)\|_{L^2(\bO)}^2=1\ \text{ for all }\ t\in[0,T],
			\end{align}  
			In other words, it implies $u_m(t)\in\bM$ for every $t\in[0,T]$.

			\vspace{2mm}
			\noindent	
			\textbf{Step VI.}
			Let us now find higher order moments by showing that
			\begin{align}
				\sup_{m\in\N}	\E\left[\sup_{t\in[0,T]}\left( {\frac{1}{2}}\|u_m(t)\|_{H_0^1(\bO)}^{4} + \frac{1}{p}	\|u_m(t)\|_{L^p(\bO)}^{2p}\right)\right] \leq C\Big(\|u_0\|_{\Vp},  T, \sum_{i=1}^M \|f_i\|_{\Vp} \Big).\label{eqn-est-u-14}
			\end{align}
			We use the finite dimensional It\^o formula for the function $v \mapsto v^2$ and to the processes $\frac{1}{2}\|u_m(\cdot)\|_{H_0^1(\bO)}^2 + \frac{1}{p}\|u_m(\cdot)\|_{L^p(\bO)}^p $ satisfying \eqref{eqn-est-u-5}. Then, for every $t\in[0,T]$, $\Pr-$a.s., we have 
			\begin{align}
				& \left(\frac{1}{2}\|u_m(t\wedge\tau_m^k)\|_{H_0^1(\bO)}^{2} + \frac{1}{p}	\|u_m(t\wedge\tau_m^k)\|_{L^p(\bO)}^{p} \right)^2\\
				& = \left(\frac{1}{2} \|u_m(0)\|_{H_0^1(\bO)}^{2} + \frac{1}{p}\|u_m(0)\|_{L^p(\bO)}^{p} \right)^2 -2\int_0^{t\wedge\tau_m^k} \left(\frac{1}{2}\|u_m(s)\|_{H_0^1(\bO)}^{2} + \frac{1}{p} \|u_m(s)\|_{L^p(\bO)}^{p} \right)\\
				\nonumber& \quad \times \big\| - \Delta u_m(s) + {P_m}(|u_m(s)|^{p-2}u_m(s))  -\big(\|\nabla u_m(s)\|_{L^2(\bO)}^2 + \|u_m(s)\|_{L^p(\bO)}^p \big)u_m(s)\big\|_{L^2(\bO)}^2ds\\
				& \quad + \sum_{i=1}^M \int_0^{t\wedge\tau_m^k}	\Big(\big(-\Delta u_m(s) +|u_m(s)|^{p-2}u_m(s), \kappa_i^m(u_m(s))\big)\\
				& \qquad\qquad\qquad\qquad + \big(-\Delta \Nn_i^m(u_m(s)) + ({p-1})|u_m(s)|^{p-2}\Nn_i^m(u_m(s)), \Nn_i^m(u_m(s))\big) \Big) \\
				&\qquad\qquad\qquad \times \left(\frac{1}{2}\|u_m(s)\|_{H_0^1(\bO)}^{2} + \frac{1}{p} \|u_m(s)\|_{L^p(\bO)}^{p} \right) ds\\
				& \quad + \sum_{i=1}^M\int_0^{t\wedge\tau_m^k}  \big|\big(-\Delta u_m(s) + |u_m(s)|^{p-2}u_m(s), \Nn_i^m(u_m(s))\big)\big|^2ds\\
				& \quad + 2\sum_{i=1}^M \int_0^{t\wedge\tau_m^k}\left(\frac{1}{2}\|u_m(s)\|_{H_0^1(\bO)}^{2} + \frac{1}{p}	\|u_m(s)\|_{L^p(\bO)}^{p}\right)\\
				& \qquad\qquad \qquad \times \big(-\Delta u_m(s) +|u_m(s)|^{p-2}u_m(s), \Nn_i^m(u_m(s))\big) dW_i(s).\label{eqn-sq-1}
			\end{align}
			However, by utilizing the fact that $\|u_m(t)\|_{L^2(\bO)} =1$ for every $t \in [0,T]$ (see \eqref{eqn-u_m-L^2=1}), property (v) of Proposition \ref{Prop-P_m} (namely $P_m S_{m-1} = S_{m-1}$), and the H\"older inequality twice (with the exponents $(2,2)$ and $(\frac{p}{p-1},p)$), we estimate as follows:
			 \begin{align}
					\big|\big(A u_m , \Nn_i^m(u_m)\big)\big|^2 & \le \|\nabla u_m\|_{L^2(\bO)}^2 \|\nabla \Nn_i^m(u_m)\|_{L^2(\bO)}^2 \\
					& = \|\nabla u_m\|_{L^2(\bO)}^2 \|S_{m-1} \nabla (f_i  - (f_i, u_m) u_m) \|_{L^2(\bO)}^2\\
					& \le 2 \|u_m\|_{H_0^1(\bO)}^2 \big(\| f_i\|_{H_0^1(\bO)}^2 +  \|f_i\|_{L^2(\bO)}^2  \|u_m\|_{H_0^1(\bO)}^2 \big),\label{eqn-A-N_i}
				\end{align}
				and by using the triangle, H\"older's and then Young's inequalities, we get
				\begin{align}
					& \big|\big(|u_m|^{p-2}u_m, \Nn_i^m(u_m)\big)\big|^2 
					\le \|u_m\|_{L^p(\bO)}^{2p-2} \| S_{m-1}(f_i - (f_i, u_m) u_m)\|_{L^p(\bO)}^2\\
					& \le 2 \|S_{m-1}\|_{\Ls(L^p(\bO))}^2\big( \|f_i\|_{L^p(\bO)}^2 + \|f_i\|_{L^2(\bO)}^2\|u_m\|_{L^p(\bO)}^2 \big)  \|u_m\|_{L^p(\bO)}^{2p-2}\\
					& \le \bigg(\frac{2p-2}{p} \|S_{m-1}\|_{\Ls(L^p(\bO))}^{\frac{2p}{p-1}} \|f_i\|_{L^p(\bO)}^{\frac{p}{p-1}} + 2 \|S_{m-1}\|_{\Ls(L^p(\bO))}^2 \|f_i\|_{L^2(\bO)}^2 \bigg)\|u_m\|_{L^p(\bO)}^{2p}\\
					&\quad +\frac{2}{p}\|f_i\|_{L^p(\bO)}^p.\label{eqn-nonli-N_i}
			\end{align}
			We estimate the second last term of \eqref{eqn-sq-1} by using the above two inequalities
				\begin{align*}
					&\sum_{i=1}^M\int_0^{t\wedge\tau_m^k} \big|\big(Au_m(s) + |u_m(s)|^{p-2}u_m(s), \Nn_i^m(u_m(s))\big)\big|^2 ds\\
					& \le 2  \sum_{i=1}^M\int_0^{t\wedge\tau_m^k}\Big( \big|\big(Au_m(s), \Nn_i^m(u_m(s))\big)\big|^2 + \big|\big(|u_m(s)|^{p-2}u_m(s), \Nn_i^m(u_m(s))\big)\big|^2\Big) ds\\
					& \le 2 \sum_{i=1}^M\int_0^{t\wedge\tau_m^k} \bigg( \frac{2}{p} \|f_i\|_{L^p(\bO)}^p + 2 \|\nabla u_m(s)\|_{L^2(\bO)}^2 \| f_i\|_{H_0^1(\bO)}^2 + 2\|f_i\|_{L^2(\bO)}^2  \|\nabla u_m(s)\|_{L^2(\bO)}^4\\
					&\quad  + \bigg(\frac{2p-2}{p} \|S_{m-1}\|_{\Ls(L^p(\bO))}^{\frac{2p}{p-1}} \|f_i\|_{L^p(\bO)}^{\frac{p}{p-1}} + 2 \|S_{m-1}\|_{\Ls(L^p(\bO))}^2 \|f_i\|_{L^2(\bO)}^2 \bigg) \|u_m(s)\|_{L^p(\bO)}^{2p} \bigg)ds\\
					& \le 4 \sum_{i=1}^M  \| f_i\|_{H_0^1(\bO)}^2 \int_0^{t\wedge\tau_m^k} \|u_m(s)\|_{H_0^1(\bO)}^2 ds +\frac{4T}{p} \sum_{i=1}^M  \| f_i\|_{L^p(\bO)}^p \\
					&\quad + \sum_{i=1}^M C^i_3 \int_0^{t\wedge\tau_m^k}\big( \|u_m\|_{H_0^1(\bO)}^4 + \|u_m\|_{L^p(\bO)}^{2p}\big)ds\\
					& \le  C\bigg( \| u_0\|_{\Vp}, T, \sum_{i=1}^M \| f_i\|_{\Vp} \bigg) + \sum_{i=1}^M C^i_3 \int_0^{t\wedge\tau_m^k}\big( \|u_m\|_{H_0^1(\bO)}^2 + \|u_m\|_{L^p(\bO)}^{p}\big)^2ds,
				\end{align*}
				where 
				\begin{equation}\label{eqn-C^i_3}
					C^i_3 := \max\left\{ 4\|f_i\|_{L^2(\bO)}^2, 4 \bigg(\frac{p-1}{p} \|S_{m-1}\|_{\Ls(L^p(\bO))}^{\frac{2p}{p-1}} \|f_i\|_{L^p(\bO)}^{\frac{p}{p-1}} +  \|S_{m-1}\|_{\Ls(L^p(\bO))}^2 \|f_i\|_{L^2(\bO)}^2 \bigg) \right\}.
			\end{equation}
			By arguments analogous to those in \eqref{eqn-est-u-7}--\eqref{eqn-est-u-9} and using the inequalities $\frac{1}{p} \le \frac{1}{2}$ (for any $p \in [2,\infty)$), along with the estimate obtained above, we infer
			\begin{align}
				\nonumber& \frac{1}{p^2} \big( \|u_m(t\wedge\tau_m^k)\|_{H_0^1(\bO)}^{2} + 	\|u_m(t\wedge\tau_m^k)\|_{L^p(\bO)}^{p} \big)^2 + 2\int_0^{t\wedge\tau_m^k} \left(\frac{1}{2}\|u_m(s)\|_{H_0^1(\bO)}^{2} + \frac{1}{p}	\|u_m(s)\|_{L^p(\bO)}^{p} \right)\\
				& \ \times \big\| - \Delta u_m(s) + {P_m}(|u_m(s)|^{p-2}u_m(s)) - \big(\|\nabla u_m(s)\|_{L^2(\bO)}^2+ \|u_m(s)\|_{L^p(\bO)}^p \big)u_m(s)\big\|_{L^2(\bO)}^2ds\\ 
				\nonumber& \leq \left(\frac{1}{2} \|u_m(0)\|_{H_0^1(\bO)}^{2} + \frac{1}{p}\|u_m(0)\|_{L^p(\bO)}^{p} \right)^2  + \sum_{i=1}^M \bigg( \frac{5}{2} \|f_i\|_{H_0^1(\bO)}^2 +  \frac{2}{p} \|f_i\|_{L^p(\bO)}^{\frac{p}{2}} + \frac{7}{p}\|f_i\|_{L^p(\bO)}^{p} \bigg)\\
				&\qquad \times \int_0^{t\wedge\tau_m^k} \left(\frac{1}{2}\|u_m(s)\|_{H_0^1(\bO)}^{2} + \frac{1}{p}	\|u_m(s)\|_{L^p(\bO)}^{p} \right)ds\\
				& \quad + \sum_{i=1}^M \Big( 2C^i_1 +2C^i_2 + C^i_3\Big) \int_0^{t\wedge\tau_m^k} \big(\|u_m(s)\|_{H_0^1(\bO)}^2 + \|u_m(s)\|_{L^p(\bO)}^p\big)^2ds\\
				& \quad + 2\sum_{i=1}^M \int_0^{t\wedge\tau_m^k}\left(\frac{1}{2}\|u_m(s)\|_{H_0^1(\bO)}^{2} + \frac{1}{p}	\|u_m(s)\|_{L^p(\bO)}^{p} \right)\\
				& \qquad\qquad\qquad\qquad \times\big(-\Delta u_m(s) +|u_m(s)|^{p-2}u_m(s),  \Nn_i^m(u_m(s))\big)dW_i(s), \label{eqn-est-u-15}
			\end{align}
			where $C_i^j$ are from \eqref{eqn-C^i_1}, \eqref{eqn-C^i_1},\eqref{eqn-C^i_3}, for $j =1,2,3$.
			
			\vspace{1mm}
			\noindent
			\textbf{Step VII.} Taking expectation on both sides of \eqref{eqn-est-u-15} and utilizing the fact that the last term appearing in the right-hand side of \eqref{eqn-est-u-15} is a martingale (that means its expectation will be zero), we arrive at
			\begin{align}
				& { \frac{1}{p^2} \E \big[ \big(\|u_m(t\wedge\tau_m^k)\|_{H_0^1(\bO)}^{2} +	\|u_m(t\wedge\tau_m^k)\|_{L^p(\bO)}^{p} \big)^2 \big]}\\
				& \quad + 2\E \bigg[ \int_0^{t\wedge\tau_m^k} \Big(\frac{1}{2}\|u_m(s)\|_{H_0^1(\bO)}^{2} + {\frac{1}{p}}	\|u_m(s)\|_{L^p(\bO)}^{p}\Big) \big\| - \Delta u_m(s) + {P_m}(|u_m(s)|^{p-2}u_m(s))\\
				&\qquad\qquad\qquad -\big(\|\nabla u_m(s)\|_{L^2(\bO)}^2+ \|u_m(s)\|_{L^p(\bO)}^p\big)u_m(s)\big\|_{L^2(\bO)}^2ds\bigg]\\
				\nonumber& \leq \left(\frac{1}{2} \|u_m(0)\|_{H_0^1(\bO)}^{2} + \frac{1}{p}\|u_m(0)\|_{L^p(\bO)}^{p} \right)^2 + \sum_{i=1}^M \left( \frac{5}{2} \|f_i\|_{H_0^1(\bO)}^2 +  \frac{2}{p} \|f_i\|_{L^p(\bO)}^{\frac{p}{2}} + \frac{7}{p}\|f_i\|_{L^p(\bO)}^{p} \right)\\
				& \qquad \times \E \bigg[\int_0^{t\wedge\tau_m^k}	\left(\frac{1}{2}\|u_m(s)\|_{H_0^1(\bO)}^{2} + \frac{1}{p}	\|u_m(s)\|_{L^p(\bO)}^{p} \right)ds\bigg]\\
				& \quad + \sum_{i=1}^M \Big(  2C^i_1 + 2C^i_2 + C^i_3 \Big) \E\bigg[\int_0^{t\wedge\tau_m^k} \big(\|u_m(s)\|_{H_0^1(\bO)}^2 + \|u_m(s)\|_{L^p(\bO)}^p\big)^2ds\bigg] \\
				& \leq C\left( \|u_0\|_{\Vp},T,  \sum_{i=1}^M\|f_i\|_{\Vp} \right)\\
				& \quad + \sum_{i=1}^M \big( 2C^i_1 + 2C^i_2 + C^i_3 \big) \E \bigg[\int_0^{t\wedge\tau_m^k} \big(\|u_m(s)\|_{H_0^1(\bO)}^{2} +	\|u_m(s)\|_{L^p(\bO)}^{p}\big)^2ds\bigg], \label{eqn-est-u-16}
			\end{align}
			where we have used uniform bound \eqref{eqn-est-u-13}. Thus, an application of Gronwall's inequality in \eqref{eqn-est-u-16} yields for all $t\in[0,T]$
			\begin{align*}
				& \frac{1}{p^2}\E\Big[\big(\|u_m(t\wedge\tau_m^k)\|_{H_0^1(\bO)}^{2} +	\|u_m(t\wedge\tau_m^k)\|_{L^p(\bO)}^{p}\big)^2\Big]\\
				& \quad + 2 \E\Big[\int_0^{t\wedge\tau_m^k} \left(\frac{1}{2}\|u_m(s)\|_{H_0^1(\bO)}^{2} + {\frac{1}{p}}	\|u_m(s)\|_{L^p(\bO)}^{p}\right)\\
				& \quad\times\big\| - \Delta u_m(s) + {P_m}(|u_m(s)|^{p-2}u_m(s)) -\big(\|\nabla u_m(s)\|_{L^2(\bO)}^2+ \|u_m(s)\|_{L^p(\bO)}^p\big)u_m(s)\big\|_{L^2(\bO)}^2ds\Big]\\
				& \leq C\left( \|u_0\|_{\Vp},T, \sum_{i=1}^M\|f_i\|_{\Vp} \right) \exp\left(\sum_{i=1}^M \big( 2C^i_1 + 2C^i_2 + C^i_3 \big)T\right),
			\end{align*}
			where $C_i^j$'s are from \eqref{eqn-C^i_1},\eqref{eqn-C^i_2},\eqref{eqn-C^i_3} for $j=1,2,3$.
			In particular, by applying the inequality $a^2 + b^2 \le (a+b)^2$ for $a,b \ge 0$ to the left-hand side of the above estimate gives
			\begin{equation*}
				\E\Big[\|u_m(t\wedge\tau_m^k)\|_{H_0^1(\bO)}^{4} +	\|u_m(t\wedge\tau_m^k)\|_{L^p(\bO)}^{2p}\Big]  \leq C\left( \|u_0\|_{\Vp},T, \sum_{i=1}^M\|f_i\|_{\Vp} \right).
			\end{equation*}
			Letting $k\to\infty$, by utilizing the same approach as in Step V (since $t\wedge \tau_m^k \to t$), we obtain
			\begin{equation}
				\E\Big[\|u_m(t)\|_{H_0^1(\bO)}^{4} +	\|u_m(t)\|_{L^p(\bO)}^{2p}\Big] \leq C\left( \|u_0\|_{\Vp},T, \sum_{i=1}^M\|f_i\|_{\Vp} \right).\label{eqn-est-u-17}
			\end{equation}
			
			\vskip 2mm
			\noindent
			\textbf{Step VIII.} In order to finish the proof, we first take supremum over $[0,T]$, then take expectation, and finally use a calculation similar to Step VII to obtain
			\begin{align}
				& \E \bigg[\sup_{t\in[0,T\wedge\tau_m^k]} \Big(\frac{1}{2}\|u_m(t)\|_{H_0^1(\bO)}^{2} + \frac{1}{p}	\|u_m(t)\|_{L^p(\bO)}^{p} \Big)^2 \bigg]\\ 
				&  + \E \bigg[ \int_0^{T\wedge\tau_m^k} \left( \frac{1}{2}\|u_m(t)\|_{H_0^1(\bO)}^{2} + \frac{1}{p}	\|u_m(t)\|_{L^p(\bO)}^{p} \right) \big\| - \Delta u_m(t) + {P_m}(|u_m(t)|^{p-2}u_m(t))\\
				& \qquad\qquad\qquad\qquad\qquad - \big( \|\nabla u_m(t)\|_{L^2(\bO)}^2+ \|u_m(t)\|_{L^p(\bO)}^p \big) u_m(t)\big\|_{L^2(\bO)}^2 dt \bigg]\\
				& \leq C\left(\|u_0\|_{\Vp},T, \sum_{i=1}^M \|f_i\|_{\Vp}\right)\\ 
				& \quad + \sum_{i=1}^M( 2C^i_1 + 2C^i_2 + C^i_3 ) \E\bigg[\int_0^{T\wedge\tau_m^k}\left(\frac{1}{2}\|u_m(t)\|_{H_0^1(\bO)}^{2} + \frac{1}{p}	\|u_m(t)\|_{L^p(\bO)}^{p}\right)^2dt\bigg]  \\ 
				& \quad + 2\sum_{i=1}^M \E\bigg[\sup_{t\in[0, T\wedge\tau_m^k]}\bigg|\int_0^{t}\left(\frac{1}{2}\|u_m(s)\|_{H_0^1(\bO)}^{2} + \frac{1}{p}	\|u_m(s)\|_{L^p(\bO)}^{p} \right)\\
				& \qquad\qquad\qquad\qquad\qquad\times\big(-\Delta u_m(s) +|u_m(s)|^{p-2}u_m(s), \Nn_i^m(u_m(s))\big)dW_i(s)\bigg|\bigg]. \label{eqn-bef-BDGI}
			\end{align}
			{Applying the Burkholder-Davis-Gundy inequality (with exponents $1$), see \cite[Theorem 1]{AI-86}, we can estimate the term containing the It\^o integral as follows:
				\begin{align}
					& \E\bigg[\sup_{t\in[0, T\wedge\tau_m^k]}\Big|\int_0^{t}\left(\frac{1}{2}\|u_m(s)\|_{H_0^1(\bO)}^{2} + \frac{1}{p}	\|u_m(s)\|_{L^p(\bO)}^{p} \right)\\
					& \qquad\qquad\qquad \times\big(-\Delta u_m(s) +|u_m(s)|^{p-2}u_m(s),\Nn_i^m(u_m(s))\big)dW_i(s)\Big|\bigg]\\
					& \le 3 \E\bigg[\Big(\int_0^{T\wedge\tau_m^k} \Big(\frac{1}{2}\|u_m(t)\|_{H_0^1(\bO)}^{2} + \frac{1}{p}	\|u_m(t)\|_{L^p(\bO)}^{p} \Big)^2\\
					& \qquad\qquad\qquad \times \big| \big(-\Delta u_m(t) +|u_m(t)|^{p-2}u_m(t), \Nn_i^m(u_m(t))\big) \big|^2 dt \Big)^{\frac{1}{2}} \bigg].
				\end{align}
				Further, making use of the above estimate along with\dela{H\"older's inequality in time (with exponent),} the estimates \eqref{eqn-A-N_i} and \eqref{eqn-nonli-N_i}, and Young's inequality yield
				\begin{align*}
					& 2 \sum_{i=1}^M \E\bigg[\sup_{t\in[0, T\wedge\tau_m^k]}\Big|\int_0^{t}\left(\frac{1}{2}\|u_m(s)\|_{H_0^1(\bO)}^{2} + \frac{1}{p}	\|u_m(s)\|_{L^p(\bO)}^{p} \right)\\
					& \qquad\qquad\qquad\times\big(-\Delta u_m(s) +|u_m(s)|^{p-2}u_m(s),\Nn_i^m(u_m(s))\big)dW_i(s)\Big|\bigg]\\
					& \le 6 \sum_{i=1}^M\E\bigg[\bigg(\int_0^{T\wedge\tau_m^k} \Big(\frac{1}{2}\|u_m(t)\|_{H_0^1(\bO)}^{2} + \frac{1}{p}	\|u_m(t)\|_{L^p(\bO)}^{p} \Big)^2 \norm{S_{m-1}}_{L^2(\bO)}^2 \norm{\Nn_i(u_m(t))}_{L^2(\bO)}^2\\
					& \qquad \times 2\big( \| -\Delta u_m(t) + P_m(|u_m(t)|^{p-2}u_m(t)) - (\norm{\nabla u_m(t)}_{L^2(\bO)}^2 + \norm{ u_m(t)}_{L^p(\bO)}^p) u_m(t) \|_{L^2(\bO)}^2\\
					& \qquad\qquad + \big(\norm{\nabla u_m(t)}_{L^2(\bO)}^2 + \norm{ u_m(t)}_{L^p(\bO)}^p)^2 \big) dt \bigg)^{\frac{1}{2}} \bigg]\\
					& \le 12 \sum_{i=1}^M\norm{f_i}_{L^2(\bO)} \E \bigg[ \bigg( \sup_{t \in [0,T\wedge\tau_m^k]}\Big(\frac{1}{2}\|u_m(t)\|_{H_0^1(\bO)}^{2} + \frac{1}{p}	\|u_m(t)\|_{L^p(\bO)}^{p} \Big)^2 \bigg)^{\frac{1}{2}}  \\
					& \quad \times \bigg(\int_0^{T\wedge\tau_m^k} \big( \| -\Delta u_m(t) + P_m(|u_m(t)|^{p-2}u_m(t))\\
					&\quad  -(\norm{\nabla u_m(t)}_{L^2(\bO)}^2 + \norm{ u_m(t)}_{L^p(\bO)}^p) u_m(t) \|_{L^2(\bO)}^2 + \big(\norm{\nabla u_m(t)}_{L^2(\bO)}^2 + \norm{ u_m(t)}_{L^p(\bO)}^p)^2 \big) dt \bigg)^{\frac{1}{2}} \bigg]\\
					& \le 12 \sum_{i=1}^M\norm{f_i}_{L^2(\bO)} \eps \E\bigg[\sup_{t \in [0,T\wedge\tau_m^k]} \Big(\frac{1}{2}\|u_m(t)\|_{H_0^1(\bO)}^{2} + \frac{1}{p}	\|u_m(t)\|_{L^p(\bO)}^{p} \Big)^2 \bigg]\\
					&\quad + \frac{12 \sum_{i=1}^M\norm{f_i}_{L^2(\bO)}}{4\eps} \E\bigg[ \int_0^{T\wedge\tau_m^k} \big( \big(\norm{u_m(t)}_{H_0^1(\bO)}^2 + \norm{ u_m(t)}_{L^p(\bO)}^p)^2\\
					& \quad + \| -\Delta u_m(t) + P_m(|u_m(t)|^{p-2}u_m(t)) -(\norm{\nabla u_m(t)}_{L^2(\bO)}^2 + \norm{ u_m(t)}_{L^p(\bO)}^p) u_m(t) \|_{L^2(\bO)}^2 \big) dt \bigg].
				\end{align*}
				\textbf{Step IX.} Choosing $\eps = \frac{1}{24  \sum_{i=1}^M\|f_i\|_{L^2(\bO)}}$, provided $f_i \ne 0$ for $1\le i \le M$, thanks to the linearity of the expectation and utilizing the bound \eqref{eqn-est-u-13}  in the above inequality give
				\begin{align}
					& 2 \sum_{i=1}^M \E\bigg[\sup_{t\in[0, T\wedge\tau_m^k]} \bigg| \int_0^{t}\left(\frac{1}{2}\|u_m(s)\|_{H_0^1(\bO)}^{2} + \frac{1}{p}	\|u_m(s)\|_{L^p(\bO)}^{p} \right)\\
					& \qquad\qquad\times\big(-\Delta u_m(s) +|u_m(s)|^{p-2}u_m(s),\Nn_i^m(u_m(s))\big)dW_i(s)\bigg| \bigg]\\
					& \le  \frac{1}{2} \E\bigg[\sup_{t \in [0,T\wedge\tau_m^k]} \Big(\frac{1}{2}\|u_m(t)\|_{H_0^1(\bO)}^{2} + \frac{1}{p}	\|u_m(t)\|_{L^p(\bO)}^{p} \Big)^2 \bigg]\\
					& \quad + 72 \sum_{i=1}^M \norm{f_i}_{L^2(\bO)} \E\bigg[ \int_0^{T\wedge\tau_m^k}  \big(\norm{u_m(t)}_{H_0^1(\bO)}^2 + \norm{ u_m(t)}_{L^p(\bO)}^p)^2 dt\bigg]\\
					&\quad + C\left( \|u_0\|_{\Vp},T,  \sum_{i=1}^M \|f_i\|_{\Vp}\right).\label{eqn-BDGI}
				\end{align}
				By using the inequality \eqref{eqn-BDGI} in \eqref{eqn-bef-BDGI}, we infer
				\begin{align*}
					& \frac{1}{2}\E \bigg[\sup_{t\in[0,T\wedge\tau_m^k]} \Big(\frac{1}{2}\|u_m(t)\|_{H_0^1(\bO)}^{2} + \frac{1}{p}	\|u_m(t)\|_{L^p(\bO)}^{p} \Big)^2 \bigg]\\ 
					& \quad + \E \bigg[ \int_0^{T\wedge\tau_m^k} \left( \frac{1}{2}\|u_m(t)\|_{H_0^1(\bO)}^{2} + \frac{1}{p} \|u_m(t)\|_{L^p(\bO)}^{p} \right)\\
					\nonumber& \qquad\quad \times \big\| - \Delta u_m(t) +|u_m(t)|^{p-2}u_m(t) - \big( \|\nabla u_m(t)\|_{L^2(\bO)}^2+ \|u_m(t)\|_{L^p(\bO)}^p \big) u_m(t)\big\|_{L^2(\bO)}^2 dt \bigg]\\
					& \leq C\left( \|u_0\|_{\Vp},T, \sum_{i=1}^M \|f_i\|_{\Vp}\right)\\ 
					& \quad + \sum_{i=1}^M \big( 72 \|f_i\|_{L^2(\bO)}^2 + 2C^i_1 + 2C^i_2 + C^i_3 \big)\E\bigg[\int_0^{T\wedge\tau_m^k}\left(\frac{1}{2}\|u_m(t)\|_{H_0^1(\bO)}^{2} + \frac{1}{p}	\|u_m(t)\|_{L^p(\bO)}^{p}\right)^2 dt \bigg].
				\end{align*}
				Ultimately, by applying Gronwall’s inequality, we conclude
				\begin{equation}
					\E \bigg[\sup_{t\in[0,T]} \Big(\frac{1}{2}\|u_m(t)\|_{H_0^1(\bO)}^{2} + \frac{1}{p}	\|u_m(t)\|_{L^p(\bO)}^{p} \Big)^2 \bigg]
					\leq C\left( \|u_0\|_{\Vp},T,  \sum_{i=1}^M \|f_i\|_{\Vp}\right).\label{eqn-high-m}
				\end{equation}
				In particular, we obtain the required higher moments estimate
				\begin{equation}
					\E \bigg[\sup_{t\in[0,T]} \Big(\|u_m(s)\|_{H_0^1(\bO)}^{4} + 	\|u_m(s)\|_{L^p(\bO)}^{2p} \Big) \bigg]
					\leq C\left(\|u_0\|_{\Vp},T,  \sum_{i=1}^M \|f_i\|_{\Vp}\right).
				\end{equation}

				\vspace{2mm}
				\noindent
				\textbf{Step X.} Observe that, since $\|u_m(t)\|_{L^2(\bO)}^2=1$ for all $t \in [0,T]$, see \eqref{eqn-u_m-L^2=1}, the identity \eqref{eqn-L^2-relation} holds. This further implies that
				\begin{align}
					\big\| - \Delta u_m + {P_m( |u_m|^{p-2}u_m)} \big\|_{L^2(\bO)}^2 & = \|\Delta u_m\|_{L^2(\bO)}^2+2{(p-1)}\||u_m|^{\frac{p-2}{2}}\nabla u_m\|_{L^2(\bO)}^2\\
					& \quad + \|{P_m(|u_m|^{p-2}u_m)}\|_{L^{2}(\bO)}^{2}.\label{eqn-fin-1}
				\end{align}
				Using \eqref{eqn-L^2-relation} and \eqref{eqn-fin-1}, we immediately deduce
				\begin{align}
					& \|\Delta u_m\|_{L^2(\bO)}^2 + 2{(p-1)}\||u_m|^{\frac{p-2}{2}}\nabla u_m\|_{L^2(\bO)}^2+ \|{P_m(|u_m|^{p-2}u_m)}\|_{L^{2}(\bO)}^{2}
					\\ & \leq \big\| - \Delta u_m +|u_m|^{p-2}u_m -\big(\|\nabla u_m\|_{L^2(\bO)}^2+ \|u_m\|_{L^p(\bO)}^p\big)u_m\big\|_{L^2(\bO)}^2\\
					&\quad  + \big(\|\nabla u_m\|_{L^2(\bO)}^2+ \|u_m\|_{L^p(\bO)}^p\big)^2. \label{eqn-fin-2}
				\end{align}
				Finally, by invoking inequality \eqref{eqn-fin-2} in the left-hand side of \eqref{eqn-est-u-13} and utilizing the estimates \eqref{eqn-high-m}, we obtain
				\begin{align}
					& \E\left[\int_0^T\|\Delta u_m(t)\|_{L^2(\bO)}^2dt\right]+ 2(p-1)\E\left[\int_0^T\||u_m(t)|^{\frac{p-2}{2}} \nabla u_m(t)\|_{L^2(\bO)}^2dt\right]\\
					& \quad + \E\left[\int_0^T\|{P_m(|u_m(t)|^{p-2}u_m(t))}\|_{L^{2}(\bO)}^{2}dt\right] \leq C\left(\|u_0\|_{\Vp},T,  \sum_{i=1}^M \|f_i\|_{\Vp}\right),
				\end{align}
				which completes the proof of Lemma \ref{Lem-energy}.}
		\end{proof}

		\subsection{Tightness}
		By utilizing the a-priori energy estimates established in Lemma \ref{Lem-energy} and Corollary \ref{Cor-Aldous}, this subsection aims to demonstrate the tightness of the sequence of laws $\{\bL(u_m)\}_{m\in\N}$, associated with the sequence of solutions $\{u_m\}_{m\in\N}$ to the SDE \eqref{eqn-Faedo-Galerkin}, defined on $(\mathcal{Z}_T, \mathcal{T})$. 
		
		\begin{lemma}\label{lem-tight}
			The set of measures $\{\mathscr{L}(u_m)\}_{m\in\N}$ is tight on $(\mathcal{Z}_T, \mathcal{T})$.
		\end{lemma}
		
		\begin{proof}
			The proof follows with the help of Corollary \ref{Cor-Aldous}. 
			
			\vspace{2mm}
			\noindent
			\textbf{Step I.} For each fixed $m\in \N$, let $u_m$ be the solution to the SDE \eqref{eqn-Faedo-Galerkin}. Choosing $X_m = u_m$ in Corollary \ref{Cor-Aldous}, the a-priori estimates established in Lemma \ref{Lem-energy} ensure that the conditions (a) {and (b)} of Corollary \ref{Cor-Aldous} hold. Therefore, we are left to show that $\{u_m\}_{m\in\N}$ satisfies [$\mathbf{A}$] in $L^2(\bO)$, which will ultimately imply the tightness of the corresponding laws. 
			
			Let $\{\tau_m\}_{m\in\N}$ be a sequence of stopping times such that $0\leq \tau_m\leq T$. Then, from the SDE \eqref{eqn-Faedo-Galerkin}, we infer for all $t\in[0,T]$
			\begin{align}\label{eqn-Faedo-Galerkin-1}
				u_m(t) & = u_m(0) + \int_0^t \Delta u_m(s)ds - \int_0^t P_m(|u_m(s)|^{p-2}u_m(s))ds\\
				& \quad + \int_0^t\big( \norm{\nabla u_m(s)}_{L^2(\bO)}^2 + \norm{u_m(s)}_{L^p(\bO)}^p \big) u_m(s)ds + \frac{1}{2}\sum_{i=1}^M \int_0^t {S_{m-1}}\kappa_i(u_m(s))ds\\
				& \quad + \sum_{i=1}^M\int_0^t {S_{m-1}}\Nn_i(u_m(s))dW_i(s)\\
				& =: I_m^1 + \sum_{i=2}^6 I_m^i(t).
			\end{align}
			\textbf{Step II.} First, let us {choose and fix} $\theta>0$. So as to prove that $\{u_m\}_{m\in\N}$ satisfies [$\mathbf{A}$], i.e., the Aldous condition, we calculate each $I_m^i$ for $2\le i\le6$ separately.
			
			\vspace{2mm}
			\noindent
			\emph{Case $I_m^2$:} Since $A :D(A)\to L^2(\bO)$, by H\"older's inequality and the estimate \eqref{energy-estimate}, it follows that
			\begin{align*}
				\E\left[\|I_m^2(\tau_m+ \theta) -I_m^2(\tau_m)\|_{L^2(\bO)}\right]
				& = \E\bigg[\bigg\|\int_{\tau_m}^{\tau_m+ \theta}\Delta u_m(s)ds\bigg\|_{L^2(\bO)}\bigg]\\ 
				& \leq\E\bigg[\int_{\tau_m}^{\tau_m+ \theta}\|\Delta u_m(s)\|_{L^2(\bO)}ds\bigg]\\ 
				& \leq\theta^{\frac{1}{2}}\bigg(\E\bigg[\int_{\tau_m}^{\tau_m+ \theta}\|\Delta u_m(s)\|_{L^2(\bO)}^2ds\bigg]\bigg)^{\frac{1}{2}}\\ 
				& \leq \theta^{\frac{1}{2}} \bigg(\E\bigg[\int_{0}^{T}\|\Delta u_m(s)\|_{L^2(\bO)}^2ds\bigg]\bigg)^{\frac{1}{2}}\\ & \leq C_2\theta^{\frac{1}{2}}.
			\end{align*}
			\emph{Case $I_m^3$:} Similarly, an application of H\"older's inequality and estimate \eqref{energy-estimate} yield
			\begin{align*}
				\E\left[\|I_m^3(\tau_m+ \theta) -I_m^3(\tau_m)\|_{L^2(\bO)}\right]
				& = \E\bigg[\bigg\|\int_{\tau_m}^{\tau_m+ \theta}P_m(|u_m(s)|^{p-2}u_m(s))ds\bigg\|_{L^2(\bO)}\bigg]\\ 
				& \leq\E\bigg[\int_{\tau_m}^{\tau_m+ \theta}\|P_m(|u_m(s)|^{p-2}u_m(s))\|_{L^2(\bO)}ds\bigg]\\ 
				& \leq\theta^{\frac{1}{2}}\bigg(\E\bigg[\int_{\tau_m}^{\tau_m+ \theta}\|P_m(|u_m(s)|^{p-2}u_m(s))\|_{L^{2}(\bO)}^{2}ds\bigg]\bigg)^{\frac{1}{2}}\\ 
				& \leq C_3\theta^{\frac{1}{2}}.
			\end{align*}
			\emph{Case $I_m^4$:} Now, by utilizing the estimate \eqref{energy-estimate},  and the invariance property $\|u_m(t)\|_{L^2(\bO)}=1$ for all $t\in[0,T]$, we deduce
			\begin{align*}
				& \E\left[\|I_m^4(\tau_m+ \theta) -I_m^4(\tau_m)\|_{L^2(\bO)}\right]\\& = \E\bigg[\bigg\|\int_{\tau_m}^{\tau_m+ \theta} \big( \norm{\nabla u_m(s)}_{L^2(\bO)}^2 + \norm{u_m(s)}_{L^p(\bO)}^p \big) u_m(s)ds\bigg\|_{L^2(\bO)}\bigg]\\ & \leq \E\bigg[\int_{\tau_m}^{\tau_m+ \theta} \big( \norm{\nabla u_m(s)}_{L^2(\bO)}^2 + \norm{u_m(s)}_{L^p(\bO)}^p  \big) ds\bigg]\\
				& \leq
				\theta \E\bigg[\sup_{s\in[{\tau_m},{\tau_m+ \theta}]} \big( \norm{\nabla u_m(s)}_{L^2(\bO)}^2 + \norm{u_m(s)}_{L^p(\bO)}^p \big)\bigg] \\ & \leq \theta  \E\bigg[\sup_{s\in[0,T]} \big( \norm{\nabla u_m(s)}_{L^2(\bO)}^2 +2 \norm{u_m(s)}_{L^p(\bO)}^p \big) \bigg]\\
				& \leq C_4\theta.
			\end{align*}
			\emph{Case $I_m^5$:}   Again using the fact that $\|u_m(t)\|_{L^2(\bO)}=1$ for all $t\in[0,T]$, the estimate \eqref{eqn-kappa-est} for $p=2$ and the bound $\|S_{m-1}\|_{\Ls(L^2(\bO))} \le 1$, see Proposition \ref{Prop-S_m}, we get
			\begin{align}
				\E\left[\|I_m^5(\tau_m+ \theta) -I_m^5(\tau_m)\|_{L^2(\bO)}\right]
				& = \frac{1}{2}\E\bigg[\bigg\|\int_{\tau_m}^{\tau_m+ \theta}\sum_{i=1}^M{S_{m-1}} \kappa_i(u_m(s))ds\bigg\|_{L^2(\bO)}\bigg]\\
				& \leq\frac{1}{2}\sum_{i=1}^M\E\bigg[\int_{\tau_m}^{\tau_m+ \theta}\| {S_{m-1}} \kappa_i(u_m(s))\|_{L^2(\bO)}ds\bigg]\\
				& \leq \sum_{i=1}^M\E\bigg[\int_{\tau_m}^{\tau_m+ \theta}\|f_i\|_{L^2(\bO)}^2(1+ \|u_m\|_{L^2(\bO)}^2)\|u_m\|_{L^2(\bO)}ds\bigg]\\
				& =2\theta \sum_{i=1}^M\|f_i\|_{L^2(\bO)}^2\leq C_5\theta.\label{eqn-I_m^5}
			\end{align}
			\emph{Case $I_m^6$:}  Finally, applying the It\^o isometry, the estimate \eqref{B-L^p-bound} for $p=2$ and the invariance property $\|u_m(t)\|_{L^2(\bO)}=1$ for all $t\in[0,T]$, we arrive at
			\begin{align}
				\E\left[\|I_m^6(\tau_m+ \theta) -I_m^6(\tau_m)\|_{L^2(\bO)}^2\right]
				& =	\E\bigg[\bigg\|\sum_{i=1}^M\int_{\tau_m}^{\tau_m+ \theta}{S_{m-1}} \Nn_i(u_m(s))dW_i(s)\bigg\|_{L^2(\bO)}^2\bigg]\\
				& = \sum_{i=1}^M\E\bigg[\bigg\|\int_{\tau_m}^{\tau_m+ \theta}{S_{m-1}} \Nn_i(u_m(s))ds\bigg\|_{L^2(\bO)}^2\bigg]\\
				& \leq \sum_{i=1}^M\E\bigg[\bigg(\int_{\tau_m}^{\tau_m+ \theta}\|{S_{m-1}} \Nn_i(u_m(s))\|_{L^2(\bO)}ds\bigg)^2\bigg]\\
				& \leq \sum_{i=1}^M\E\bigg[\bigg(\int_{\tau_m}^{\tau_m+ \theta}\|f_i\|_{L^2(\bO)}(1+ \|u_m\|_{L^2(\bO)}^2)ds\bigg)^2\bigg]\\
				& =4\theta \sum_{i=1}^M\|f_i\|_{L^2(\bO)}^2\leq C_6\theta. \label{eqn-final-est}
			\end{align}
			\textbf{Step III.} Next, we fix $\eta>0$ and $\varepsilon>0$. Then, by Chebyshev's inequality and using the estimates obtained above in Step 2, we deduce the following for all $m\in\N$
			\begin{align}
				\Pr\left\{\|I_m^i(\tau_m+ \theta) -I_m^i(\tau_m)\|_{L^2(\bO)}\geq \eta \right\}\leq\frac{1}{\eta}
				\E\left[\|I_m^i(\tau_m+ \theta) -I_m^i(\tau_m)\|_{L^2(\bO)}\right]\leq\frac{C_i\theta^{\frac{1}{2}}}{\eta},
			\end{align}
			where $i=2,\ldots, {4}$. Let us take $\delta_i=\frac{\eta^2}{C^i_2}\varepsilon^2$, so that
			\begin{align}
				\sup_{m\in\N}\sup_{0\leq\theta\leq\delta_i}\Pr\left\{\|I_m^i(\tau_m+ \theta) -I_m^i(\tau_m)\|_{L^2(\bO)}\geq \eta \right\}\leq\varepsilon, \ i=2,\ldots, {4}.
			\end{align}
			{Similarly, for all $m\in\N$, the estimate \eqref{eqn-I_m^5} yields
				\begin{align}
					\Pr\left\{\|I_m^5(\tau_m+ \theta) -I_m^5(\tau_m)\|_{L^2(\bO)}\geq \eta \right\}\leq\frac{1}{\eta}
					\E\left[\|I_m^5(\tau_m+ \theta) -I_m^5(\tau_m)\|_{L^2(\bO)}\right]\leq\frac{C_5\theta}{\eta},
				\end{align}
				and assume $\delta_5 = \frac{\eta^2}{C_5^2}\eps^2$ that implies
				\begin{align}
					\sup_{m\in\N}\sup_{0\leq\theta\leq\delta_5}\Pr\left\{\|I_m^5(\tau_m+ \theta) -I_m^5(\tau_m)\|_{L^2(\bO)}\geq \eta \right\}\leq\varepsilon.
			\end{align}}
			Once again using the Chebyshev inequality together with the bound \eqref{eqn-final-est}, we have for all $m\in\N$
			\begin{equation}
				\Pr\left\{\|I_m^6(\tau_m+ \theta) -I_m^6\tau_m)\|_{L^2(\bO)}\geq \eta \right\}\leq\frac{1}{\eta^2}
				\E\left[\|I_m^6(\tau_m+ \theta) -I_m^6(\tau_m)\|_{L^2(\bO)}^2\right]\leq\frac{C_6\theta}{\eta^2}.
			\end{equation}
			Let us now take $\delta_6=\frac{\eta^2}{C_6}\varepsilon$, so that
			\begin{align}
				\sup_{m\in\N}\sup_{0\leq\theta\leq\delta_6}\Pr\left\{\|I_m^6(\tau_m+ \theta) -I_m^6(\tau_m)\|_{L^2(\bO)}\geq \eta \right\}\leq\varepsilon.
			\end{align}
			Hence, the Aldous condition [$\mathbf{A}$] holds,  for each term $I_m^i$, $i=2,\ldots, 6$. Therefore, by the triangle inequality, [$\mathbf{A}$] also holds for the sequence $\{u_m\}_{m\in\N}$. This concludes the proof by invoking Corollary \ref{Cor-Aldous}.
		\end{proof}

		\subsection{Proof of Theorem \ref{Thm-main}}
		The aim of this subsection is to utilize the tightness and energy estimates obtained in the previous subsections to collect a few necessary results, such as the fact that the solution of the SDE \eqref{eqn-Faedo-Galerkin} and its limit is a martingale, $L^p-$It\^o formula, and related arguments.
		The proof begins at the end of this subsection, after establishing the required results through the following eight steps:
		\vskip 2mm
		\noindent
		\textbf{Step I.} An direct application of Lemma \ref{lem-tight} yields that the set of measures {$\{\bL(u_m)\}_{m\in\N}$ is tight on the space $(\mathcal{Z}_T, \mathcal{T})$}, where $\mathcal{Z}_T$ is from \eqref{def-zt}. Therefore, {by choosing $\eta_m = u_m$ in} Corollary \ref{cor-Skoro} guarantees the existence of a subsequence $\{m_k\}_{k\in\N}$, a probability space $(\wtilde{\Omega},\wtilde{\Fn},\wtilde{\Fb},\wtilde{\Pr})$, $\mathcal{Z}_T-$valued random variables $\wtilde{u}$, $\wtilde{u}_{m_k}$, for $k\geq 1$, defined on $(\wtilde{\Omega},\wtilde{\Fn},\wtilde{\Fb},\wtilde{\Pr})$, such that
		\begin{align}\label{eqn-conv}
			\bL(\wtilde{u}_{m_k}) = \bL(u_{m_k})\ \text{ and }\ \wtilde{u}_{m_k}\to\wtilde{u}\ \text{ in }\ \mathcal{Z}_T, \ \wtilde{\Pr} - \text{a.s.},
		\end{align}
		i.e., $\wtilde{\Pr}-$a.s. the following convergences hold:
		\begin{equation}\label{eqn-cgs}
			\left\{
			\begin{aligned}
				\wtilde{u}_{m_k} & \to\wtilde{u}\ \text{ in }\ C([0,T];L^2(\bO))\cap L^2(0,T;H_0^1(\bO)),\\
				\wtilde{u}_{m_k} & \rightharpoonup\wtilde{u}\ \text{ in }\ L^2(0,T;D(A)),\\ 
				\wtilde{u}_{m_k} & \rightharpoonup\wtilde{u}\ \text{ in }\ C([0,T]; \Vp).
			\end{aligned}
			\right.
		\end{equation}
		
		\begin{remark}
			Note that for $1\leq d\leq 4$, we have $D(A)\embed L^p(\bO)\embed L^2(\bO)$, for $2\leq p<\infty$, with the compact embedding $D(A)\embed L^p(\bO)$. Therefore, in this case one obtains  the convergence $\wtilde{u}_{m_k}\rightharpoonup\wtilde{u} \text{ in }  L^2(0,T;L^p(\bO))$ also. We point out that we are not using this convergence to obtain the required results, due to the consideration $d \ge1$.
		\end{remark}
		
		Let us, with a slight abuse of notation, denote the subsequence $\{u_{m_k}\}_{k\in\N}$ again by $\{u_{m}\}_{m\in\N}$.		
		Since $u_m\in C([0, T]; L^2_m(\bO))$, $\Pr-$a.s., the space $C([0, T]; L^2_m(\bO))$ is a Borel subset of 
		\[C([0, T]; L^2(\bO)) \cap L^2(0, T; H_0^1(\bO)),\]
		and the random variables $\wtilde{u}_m,u_m$ have the same laws on $\mathcal{Z}_T$, it implies that for each $m\in\N$
		\begin{align*}
			\bL(u_m)(C([0, T]; L^2_m(\bO)))& =1, \ m\in\N,\\
			\|\wtilde{u}_m(t)\|_{L^2(\bO)}=	\|u_m(t)\|_{L^2(\bO)}& =1.
		\end{align*}
		Thanks to the strong convergence in \eqref{eqn-cgs}, which implies convergence in norm, together with \eqref{eqn-u_m-L^2=1}, i.e., $\wtilde{u}_m(t) \in\bM$ for all $t\in[0, T]$, it follows that
		\begin{align}\label{eqn-manifold}
			\wtilde{u}(t)\in\bM \ \text{ for all }\ t\in[0,T].
		\end{align}
		On the other hand, from \eqref{energy-estimate}, we already have
		\begin{equation}
			\sup_{m\in\N}	\wtilde{\E}\bigg[\sup_{t\in[0,T]}\left(\|\wtilde{u}_m(t)\|_{H_0^1(\bO)}^{4} + \|\wtilde{u}_m(t)\|_{L^p(\bO)}^{2p}\right)\bigg]
			\leq C, \label{eqn-new-en-1}
		\end{equation}
		and
		\begin{equation}
			\sup_{m\in\N} \bigg(	\wtilde{\E}\left[\int_0^T\|A \wtilde{u}_m(t)\|_{L^2(\bO)}^2dt\right]+ \wtilde{\E}\left[\int_0^T {\| P_m(|\wtilde{u}_m(t)|^{p-2}\wtilde{u}_m(t))\|_{L^{2}(\bO)}^{2}}dt\right]\bigg)\leq C.\label{eqn-new-en-2}
		\end{equation}
		In particular, the estimate \eqref{eqn-new-en-2} infers that the sequences $\{\wtilde{u}_m\}_{m\in\N}$ and $\{P_m(|\wtilde{u}_m|^{p-2}\wtilde{u}_m)\}_{m\in\N}$ contains a subsequence, {by using the same notation}, that are converging weakly in the spaces $$L^2([0,T]\times\wtilde{\Omega};D(A))\ {\text{ and }\ L^{2}([0,T]\times\wtilde{\Omega};L^{2}(\bO))},$$
		respectively. 
		From \cite[Proposition 4.7]{AB+ZB+MTM-25+}, it follows that 
		\begin{equation}
			P_m(|\wtilde{u}_m|^{p-2} \wtilde{u}_m) \rightharpoonup |\wtilde{u}|^{p-2}\wtilde{u}\
			\text{  in }\ L^2([0,T]\times\wtilde{\Omega}; L^2(\bO)).\label{eqn-P_m(|u_m|^{p-2}u_m)-cgs}
		\end{equation}
		Thus, the weakly lower semicontinuity property of norms, see \cite[Remark 4.8]{AB+ZB+MTM-25+}, yields
		\begin{align*}
			\norm{\abs{\wtilde{u}}^{p-2}\wtilde{u}}_{L^2([0,T]\times\wtilde{\Omega}; L^2(\bO))}^2 \leq \liminf_{m\to \infty} \norm{P_m(\abs{\wtilde{u}_m}^{p-2}\wtilde{u}_m)}_{L^2([0,T]\times\wtilde{\Omega}; L^2(\bO))}^2 < \infty.
		\end{align*}
		In particular, we obtain an extra regularity for the solution $\wtilde{u}$, i.e.,
		\begin{align}
			\wtilde{\E}\left[\int_0^T {\| \wtilde{u}(t)\|_{L^{2p-2}(\bO)}^{2p-2}} dt\right] < \infty,\label{eqn-L^{2p-2}-tilde-u}
		\end{align}
		and from \eqref{eqn-cgs}, we obtain 
		\begin{align}
			\wtilde{ \E}\left[\int_0^T\|A \wtilde{u}(t)\|_{L^2(\bO)}^2dt\right] <\infty.\label{eqn-new-en-3}
		\end{align}
		Similarly, using the uniform bound \eqref{eqn-new-en-1},  we can extract a subsequence of $\{\wtilde{u}_m\}_{m\in\N}$ that converges in  weak star topology of the space $L^2(\wtilde{\Omega};L^{\infty}(0,T;\Vp))$. Together with \eqref{eqn-cgs},  we deduce
		\begin{align}\label{eqn-new-en-4}
			\wtilde{\E}\Big[\sup_{t\in[0,T]}\big(\|\wtilde{u}(t)\|_{H_0^1(\bO)}^{4} + \|\wtilde{u}(t)\|_{L^p(\bO)}^{2p}\big)\Big]<\infty.
		\end{align}

		{The next result is essential for obtaining convergence of the nonlinear terms in the SDE \eqref{eqn-Faedo-Galerkin}. Furthermore, it is crucial in proving that the limiting process $\wtilde{u}$ is in fact a martingale solution to the SPDE \eqref{eqn-main-prob-Ito}.}
		\begin{lemma}\label{lem-strong-conv}
			Let $\psi\in \Vp$ and $0\le s\leq t\le T$ be fixed. Then, the following limits exist $\wtilde{\Pr} - $a.s.,
			\begin{align}
				& \lim_{m\to\infty}(\wtilde{u}_m(t),P_m\psi)=(\wtilde{u}(t),\psi),\label{eqn-conv-1}\\
				& \lim_{m\to\infty}\int_s^t(A\wtilde{u}_m(\sigma),P_m\psi)d\sigma=\int_s^t(A\wtilde{u}(\sigma),\psi)d\sigma,\label{eqn-conv-2}\\
				& \lim_{m\to\infty}\int_s^t(P_m(|\wtilde{u}_m(\sigma)|^{p-2}\wtilde{u}_m(\sigma)),P_m\psi)d\sigma=\int_s^t(|\wtilde{u}(\sigma)|^{p-2}\wtilde{u}(\sigma),\psi)d\sigma,\label{eqn-conv-3}\\
				& \lim_{m\to\infty}\int_s^t\big( \big( \norm{\nabla \wtilde{u}_m(\sigma)}_{L^2(\bO)}^2 + \norm{\wtilde{u}_m(\sigma)}_{L^p(\bO)}^p \big) \wtilde{u}_m(\sigma),P_m\psi\big)d\sigma\\
				& \quad = \int_s^t\big( \big( \norm{\nabla \wtilde{u}(\sigma)}_{L^2(\bO)}^2 + \norm{\wtilde{u}(\sigma)}_{L^p(\bO)}^p \big) \wtilde{u}(\sigma),\psi\big)d\sigma,\label{eqn-conv-5}\\
				& \lim_{m\to\infty}\frac{1}{2}\sum_{i=1}^M\int_s^t(\kappa_i(\wtilde{u}_m(\sigma)),S_{m-1}\psi)d\sigma=\frac{1}{2}\sum_{i=1}^M\int_s^t(\kappa_i(\wtilde{u}(\sigma)),\psi)d\sigma. \label{eqn-conv-6}
			\end{align}
		\end{lemma}
		
		\begin{proof}[Proof of Lemma \ref{lem-strong-conv}]
			Let us choose and fix $\psi\in \Vp$ and assume $0\le s\leq t\le T$. 
			
			\vskip 1mm
			\noindent
			\textbf{1.} First, let us recall from \eqref{eqn-conv} that $\wtilde{\Pr}-$a.s.
			\begin{equation}\label{Lem-cgs-1}
				\wtilde{u}_{m} \to\wtilde{u}\ \text{ in }\ C([0,T];L^2(\bO))\cap L^2(0,T;H_0^1(\bO))\cap L^2_w(0,T;D(A))\cap C_w([0,T]; \Vp).
			\end{equation}
			Since $\wtilde{u}_{m}\to\wtilde{u}\ \text{ in }\ C([0,T];L^2(\bO))$, $\wtilde{\Pr} - $a.s. and $P_m\psi\to\psi$ in $L^2(\bO)$, we have for all $t\in[0,T]$, $\wtilde{\Pr} - $a.s
			\begin{align}
				\lim_{m\to\infty}(\wtilde{u}_m(t),P_m\psi) - (\wtilde{u}(t),\psi) & = \lim_{m\to\infty}(\wtilde{u}_m(t) - \wtilde{u}(t) + \wtilde{u}(t), P_m\psi) - (\wtilde{u}(t),\psi)\\
				& = \lim_{m\to\infty}(\wtilde{u}_m(t) -\wtilde{u}(t), {P_m}\psi) + \lim_{m\to\infty}(\wtilde{u}(t),P_m\psi-\psi)=0,
			\end{align}
			and the convergence \eqref{eqn-conv-1} follows. Similarly, the convergence $\wtilde{u}_{m}\to\wtilde{u}$ in $L^2(0,T; {H_0^1(\bO)})$, $\wtilde{\Pr} - $a.s. implies that $\{\wtilde{u}_{m}\}_{m\in\N}$ is a uniformly bounded sequence in $L^2(0,T; {H_0^1(\bO)})$, $\wtilde{\Pr} - $a.s. By using this fact and the convergence  $P_m\psi\to\psi$ in $H_0^1(\bO)$, we obtain
			\begin{align*}
				& \bigg| \int_s^t (A\wtilde{u}_m(\sigma),P_m\psi) d\sigma - \int_s^t (A\wtilde{u}(\sigma),\psi) d\sigma \bigg| \\ 
				& = \bigg| \int_s^t (A\wtilde{u}_m(\sigma),P_m\psi) d\sigma - \int_s^t (A\wtilde{u}(\sigma) - A\wtilde{u}_m(\sigma) + A\wtilde{u}_m(\sigma),\psi) d\sigma \bigg| \\ 
				& \leq\bigg|\int_s^t(A\wtilde{u}_m(\sigma) -A\wtilde{u}(\sigma), \psi)d\sigma\bigg|+ \bigg|\int_s^t (A\wtilde{u}_m(\sigma),P_m\psi-\psi)d\sigma\bigg|\\
				& = \bigg|\int_s^t (\nabla(\wtilde{u}_m(\sigma) -\wtilde{u}(\sigma)), \nabla\psi)d\sigma\bigg|+ \bigg|\int_s^t (\nabla \wtilde{u}_m(\sigma), \nabla(P_m\psi-\psi))d\sigma\bigg|\\ 
				& \leq\|\nabla\psi\|_{L^2(\bO)}T^{\frac{1}{2}}\|\wtilde{u}_m - \wtilde{u}\|_{L^2(0,T;H_0^1(\bO))} + \|P_m\psi-\psi\|_{H_0^1(\bO)}T^{\frac{1}{2}}\|\wtilde{u}_m\|_{L^2(0,T;H_0^1(\bO))}\\ 
				& \to 0\ \text{ as } \ m\to\infty,
			\end{align*}
			so that  \eqref{eqn-conv-2} follows. 
			
			\vskip 1mm
			\noindent
			\textbf{2.} {Let us now show the convergence \eqref{eqn-conv-3} hold. 
				By an application of the Cauchy-Schwarz inequality and the weak convergence \eqref{eqn-P_m(|u_m|^{p-2}u_m)-cgs}, we find
				\begin{align}
					& \bigg| \int_s^t (P_m(|\wtilde{u}_m(\sigma)|^{p-2}\wtilde{u}_m(\sigma)), P_m\psi) - (|\wtilde{u}(\sigma)|^{p-2}\wtilde{u}(\sigma), \psi) d\sigma\bigg|\\
					& \le \bigg|\int_s^t ( P_m(|\wtilde{u}_m(\sigma)|^{p-2}\wtilde{u}_m(\sigma)), P_m \psi) \\
					&\quad - (|\wtilde{u}(\sigma)|^{p-2}\wtilde{u}(\sigma) + P_m(|\wtilde{u}_m(\sigma)|^{p-2}\wtilde{u}_m(\sigma)) - P_m(|\wtilde{u}_m(\sigma)|^{p-2}\wtilde{u}_m(\sigma)), \psi) d\sigma\bigg|\\
					& \le \int_s^t \big|( P_m(|\wtilde{u}_m(\sigma)|^{p-2}\wtilde{u}_m(\sigma)), P_m \psi - \psi)\big| d\sigma\\
					&\quad   + \int_s^t \big|(P_m(|\wtilde{u}_m(\sigma)|^{p-2}\wtilde{u}_m(\sigma)) - |\wtilde{u}(\sigma)|^{p-2}\wtilde{u}(\sigma), \psi)\big|d\sigma\\
					& \le \norm{P_m \psi - \psi}_{L^2(\bO)} \int_s^t \|P_m(|\wtilde{u}_m(\sigma)|^{p-2}\wtilde{u}_m(\sigma))\|_{L^{2}(\bO)} d\sigma \\
					&\quad   + \int_s^t \big|(P_m(|\wtilde{u}_m(\sigma)|^{p-2}\wtilde{u}_m(\sigma)) - |\wtilde{u}(\sigma)|^{p-2}\wtilde{u}(\sigma), \psi)\big|d\sigma\\
					& \to 0\ \text{ as }\ m\to\infty.
			\end{align}}
			\vskip 1mm
			\noindent
			\textbf{3.} Next, we aim to prove that convergence \eqref{eqn-conv-5} hold. Since $\wtilde{u}_m(t),\wtilde{u}(t)\in\bM$ for all $t\in[0,T]$, and $\{\wtilde{u}_{m}\}_{m\in\N}$ is a uniformly bounded in $L^{\infty}(0,T; \Vp)$, $\wtilde{\Pr} - $a.s., it implies that
			\begin{align}
				&	\bigg| \int_s^t \big( \big( \norm{\nabla \wtilde{u}_m(\sigma)}_{L^2(\bO)}^2 + \norm{\wtilde{u}_m(\sigma)}_{L^p(\bO)}^p \big) \wtilde{u}_m(\sigma),P_m\psi \big)d\sigma\\
				&\quad -\int_s^t \big( \big( \norm{\nabla \wtilde{u}(\sigma)}_{L^2(\bO)}^2 + \norm{\wtilde{u}(\sigma)}_{L^p(\bO)}^p \big) \wtilde{u}(\sigma),\psi\big)d\sigma\bigg|\\
				& = \bigg| \int_s^t \big( \big( \norm{\nabla \wtilde{u}_m(\sigma)}_{L^2(\bO)}^2 + \norm{\wtilde{u}_m(\sigma)}_{L^p(\bO)}^p \big) \wtilde{u}_m(\sigma), P_m\psi \big)d\sigma\\
				&\quad -\int_s^t \big( \big( \norm{\nabla \wtilde{u}(\sigma)}_{L^2(\bO)}^2 + \norm{\wtilde{u}(\sigma)}_{L^p(\bO)}^p \big) \wtilde{u}(\sigma), P_m\psi  - P_m\psi + \psi\big)d\sigma\bigg|\\
				& \leq 	\bigg|\int_s^t\big( \big( \norm{\nabla \wtilde{u}_m(\sigma)}_{L^2(\bO)}^2 + \norm{\wtilde{u}_m(\sigma)}_{L^p(\bO)}^p \big) \wtilde{u}_m(\sigma)\\
				& \qquad- \big( \norm{\nabla \wtilde{u}(\sigma)}_{L^2(\bO)}^2 + \norm{\wtilde{u}(\sigma)}_{L^p(\bO)}^p \big) \wtilde{u}(\sigma), {P_m}\psi\big)d\sigma\bigg|\\
				& \quad + \bigg|\int_s^t\big( \big( \norm{\nabla \wtilde{u}(\sigma)}_{L^2(\bO)}^2 + \norm{\wtilde{u}(\sigma)}_{L^p(\bO)}^p\big) \wtilde{u}(\sigma),P_m\psi-\psi\big)d\sigma\bigg|\\
				& \leq {T\|\psi\|_{L^2(\bO)} \sup_{\sigma\in[0,T]}\big( \norm{\nabla \wtilde{u}_m(\sigma)}_{L^2(\bO)}^2 + \norm{\nabla\wtilde{u}_m(\sigma)}_{L^2(\bO)}^2 \big) \| \wtilde{u}_m-\wtilde{u}\|_{L^{\infty}(0,T; L^2(\bO))}}\\
				& \quad {+ p 2^{p-2} T \|\psi\|_{L^2(\bO)} \sup_{\sigma \in [0,T]}\big( \norm{\nabla \wtilde{u}_m(\sigma)}_{L^2(\bO)}^2 + \norm{\wtilde{u}_m(\sigma)}_{L^p(\bO)}^p \big) \| \wtilde{u}_m-\wtilde{u}\|_{L^{\infty}(0,T; L^2(\bO))}}\\
				& \quad + {T\|P_m\psi-\psi\|_{L^2(\bO)} \sup_{\sigma\in[0,T]}\big( \norm{\nabla \wtilde{u}(\sigma)}_{L^2(\bO)}^2 + \norm{\wtilde{u}(\sigma)}_{L^p(\bO)}^p \big)}\\
				& \to 0\ \text{ as }\ m \to \infty,
			\end{align}
			where we have utilized the following fact
			\begin{align}
				&	\int_s^t\Big( \big( \norm{\nabla \wtilde{u}_m(\sigma)}_{L^2(\bO)}^2 + \norm{\wtilde{u}_m(\sigma)}_{L^p(\bO)}^p \big) - \big( \norm{\nabla \wtilde{u}(\sigma)}_{L^2(\bO)}^2 + \norm{\wtilde{u}(\sigma)}_{L^p(\bO)}^p \big) \Big)d\sigma\\
				& \leq\int_0^T \big( \norm{\nabla \wtilde{u}_m(\sigma)}_{L^2(\bO)} + \norm{\nabla \wtilde{u}(\sigma)}_{L^2(\bO)} \big) \norm{\nabla( \wtilde{u}_m(\sigma) -\wtilde{u}(\sigma))}_{L^2(\bO)}d\sigma\\
				& \qquad + \int_0^T \big( \norm{ \wtilde{u}_m(\sigma)}_{L^p(\bO)}^p - \norm{ \wtilde{u}(\sigma)}_{L^p(\bO)}^p \big)d\sigma\\
				& \leq T^{\frac{1}{2}} \sup_{\sigma \in [0,T]} \big( \norm{\nabla \wtilde{u}_m(\sigma)}_{L^2(\bO)} + \norm{\nabla \wtilde{u}(\sigma)}_{L^2(\bO)} \big) \bigg(\int_0^T \norm{\nabla( \wtilde{u}_m(\sigma) -\wtilde{u}(\sigma))}_{L^2(\bO)}^2 d\sigma\bigg)^{\frac{1}{2}}\\
				& \quad + p 2^{p-2} \int_0^T \big( \norm{ \wtilde{u}_m(\sigma)}_{L^p(\bO)}^{p-1} + \norm{ \wtilde{u}(\sigma)}_{L^p(\bO)}^{p-1} \big) \norm{ \wtilde{u}_m(\sigma) -\wtilde{u}(\sigma)}_{L^p(\bO)} d\sigma\\
				& \leq T^{\frac{1}{2}} \sup_{\sigma \in [0,T]} \big( \norm{\nabla \wtilde{u}_m(\sigma)}_{L^2(\bO)} + \norm{\nabla \wtilde{u}(\sigma)}_{L^2(\bO)} \big) \norm{ \wtilde{u}_m -\wtilde{u}}_{L^2(0,T; H_0^1(\bO))}\\
				& \quad + p 2^{p-2} T \sup_{\sigma \in [0,T]}\big( \norm{ \wtilde{u}_m(\sigma)}_{L^p(\bO)}^{p-1} + \norm{ \wtilde{u}(\sigma)}_{L^p(\bO)}^{p-1} \big) \norm{ \wtilde{u}_m -\wtilde{u}}_{L^\infty(0,T; L^p(\bO))} d\sigma.
			\end{align}
				\textbf{4.} The last convergence \eqref{eqn-conv-6} can be established in the following way. Applying H\"older's inequality, the estimates \eqref{eqn-kappa-est} and \eqref{eqn-est-3} for $p=2$, using the fact that $\wtilde{u}_m(t),\wtilde{u}(t)\in\bM$ for all $t\in[0,T]$ and the convergence $\wtilde{u}_{m}\to\wtilde{u} \text{ in }  C([0,T];L^2(\bO)),$ $\wtilde{\Pr} - $a.s., we obtain
				\begin{align}
					& \bigg| \int_s^t (\kappa_i(\wtilde{u}_m(\sigma)),S_{m-1}\psi) d\sigma - \int_s^t(\kappa_i(\wtilde{u}(\sigma)),\psi) d\sigma\bigg|\\
					& = \bigg| \int_s^t \big( (\kappa_i(\wtilde{u}_m(\sigma)), S_{m-1}\psi) - (\kappa_i(\wtilde{u}(\sigma)) - \kappa_i(\wtilde{u}_m(\sigma)) + \kappa_i(\wtilde{u}_m(\sigma)), \psi)\big) d\sigma\bigg|\\
					& \leq\bigg|\int_s^t(\kappa_i(\wtilde{u}_m(\sigma)) -\kappa_i(\wtilde{u}(\sigma)),\psi)d\sigma\bigg|+ \bigg|\int_s^t(\kappa_i(\wtilde{u}_m(\sigma)),S_{m-1}\psi-\psi)d\sigma\bigg|\\ 
					\nonumber& \leq\|\psi\|_{L^2(\bO)}\int_s^t\|\kappa_i(\wtilde{u}_m(\sigma)) -\kappa_i(\wtilde{u}(\sigma))\|_{L^2(\bO)}ds+ \|S_{m-1}\psi-\psi\|_{L^2(\bO)}\int_s^t\|\kappa_i(\wtilde{u}_m(\sigma))\|_{L^2(\bO)}ds\\
					& \leq 2\|\psi\|_{L^2(\bO)} \int_s^t\|f_i\|_{L^2(\bO)}^2\left(1+ (\|\wtilde{u}_m(\sigma)\|_{L^2(\bO)} + \|\wtilde{u}(\sigma)\|_{L^2(\bO)})^2\right)\|\wtilde{u}_m(\sigma) -\wtilde{u}(\sigma)\|_{L^2(\bO)}d\sigma\\ 
					& \quad +2\|S_{m-1}\psi-\psi\|_{L^2(\bO)} \int_s^t \|f_i\|_{L^2(\bO)}^2(1+ \|\wtilde{u}_m(\sigma)\|_{L^2(\bO)}^2)\|\wtilde{u}_m(\sigma)\|_{L^2(\bO)}d\sigma\\ 
					& \leq  10T\|\psi\|_{L^2(\bO)}\|f_i\|_{L^2(\bO)}^2\|\wtilde{u}_m-\wtilde{u}\|_{L^{\infty}(0,T;L^2(\bO))} +4\|S_{m-1}\psi-\psi\|_{L^2(\bO)}\|f_i\|_{L^2(\bO)}^2\\ 
					& \to 0\ \text{ as }\ m\to\infty,
				\end{align}
				and the convergence \eqref{eqn-conv-6} follows. It concludes the proof of Lemma \ref{lem-strong-conv}.
			\end{proof}

			Let us now come back and continue the proof of Theorem \ref{Thm-main}.

			\vskip 2mm
			\noindent
			\textbf{Step II.} Next, let us choose and fix $m\in \N$. Suppose that $\wtilde{M}_m$ is an $L^2_m(\bO)-$valued continuous process. In particular, it belongs to $C([0,T];L^2(\bO))$, and is defined for every $t\in[0, T]$ by
			\begin{align}
				\wtilde{M}_m(t)
				& = \wtilde{u}_m(t) -\wtilde{u}_m(0) - \int_0^t \bigg( \Delta \wtilde{u}_m(s) -P_m(|\wtilde{u}_m(s)|^{p-2}\wtilde{u}_m(s))  \\
				& \qquad\quad + \big( \norm{\nabla \wtilde{u}_m(s)}_{L^2(\bO)}^2 + \norm{\wtilde{u}_m(s)}_{L^p(\bO)}^p \big) \wtilde{u}_m(s) + \frac{1}{2}\sum_{i=1}^M {S_{m-1}}\kappa_i(\wtilde{u}_m(s)) \bigg) ds.\hspace{1cm}\label{eqn-martingale}
			\end{align}
			{Now, we have a result, motivated from \cite[Section 8.4, p.~229]{GDP+JZ-14}, which provide that the process $\wtilde{u}_m$ is martingale.}
			
			\begin{lemma}\label{lem-martingale}
				{For each fixed $m\in \N$}, $\wtilde{M}_m(t)$ defined in \eqref{eqn-martingale} is a square integrable martingale with respect to the filtration $\wtilde{\Fb}_m=\{\wtilde{\Fn}_{m,t}\}_{t\in[0,T]}$, here $\wtilde{\Fn}_{m,t}=\sigma\{\wtilde{u}_m(s), s\leq t\}$ with the quadratic variation
				\begin{align}\label{eqn-quadratic}
					\llbracket \wtilde{M}_m \rrbracket_t = \sum_{i=1}^M \int_0^t  \|{S_{m-1}} \Nn_i(\wtilde{u}_m(s)) \|_{L^2(\bO)}^2 ds.
				\end{align}
			\end{lemma}
			\begin{proof}
				{Since $M_m(t) = \sum_{i=1}^M \int_0^t S_{m-1} \Nn_i (u_m(s)) dW(s)$ is a square integrable martingale, so is
					\begin{align*}
						{M}_m(t) & = {u}_m(t) -{u}_m(0)\\
						& \quad-\int_0^t \bigg( \Delta {u}_m(s) -P_m(|{u}_m(s)|^{p-2}{u}_m(s))  \\
						& \qquad\qquad + \big( \norm{\nabla {u}_m(s)}_{L^2(\bO)}^2 + \norm{{u}_m(s)}_{L^p(\bO)}^p \big) {u}_m(s) + \frac{1}{2}\sum_{i=1}^M {S_{m-1}}\kappa_i({u}_m(s)) \bigg) ds.
					\end{align*}
					As we know from \eqref{eqn-conv} that ${u}_m$ and $\wtilde{u}_m$ have the same laws, together with the fact $\E \big[ \|M_m(t)\|_{L^2(\bO)}^2\big]< \infty$ this implies 
					$$\wtilde{\E} \big[ \|\wtilde{M}_m(t)\|_{L^2(\bO)}^2\big]< \infty.$$ 
					Let us fix a real-valued bounded continuous function $h$ defined on $C([0, s]; L^2(\bO))$, where $0\le s\leq t\le T$. Since $M_m$ is an $\Fb_t-$measurable martingale, and the increments of a martingale are themselves martingales, it follows that for all $\psi\in L^2(\bO)$
					\begin{align}\label{eqn-mar-zero}
						{\E}\left[({M}_m(t) - {M}_m(s),\psi)h({u}_{m}|_{[0,s]})\right]=0, \ \text{ for all }\ h.
					\end{align}
					Thus, again the uniqueness of laws \eqref{eqn-conv} asserts
					\begin{align}\label{eqn-mar-zero-tilde}
						\wtilde{\E}\big[(\wtilde{M}_m(t) -\wtilde{M}_m(s),\psi)h(\wtilde{u}_{m}|_{[0,s]})\big]=0,
					\end{align}
					which give that $\{\wtilde{M}_m\}_{t\in[0,T]}$ is also a martingale. Similarly, for any $\psi, \xi \in L^2(\bO)$, we have
					\begin{align}\label{eqn-mar-zero-}
						{\E}\bigg[ & \bigg(({M}_m(t),\psi)({M}_m(t),\xi) - ({M}_m(s),\psi)({M}_m(s),\xi) \\
						& - \sum_{i=1}^M \int_s^t (S_{m-1} \Nn_i (u_m), \psi) (S_{m-1} \Nn_i(u_m), \xi) d\sigma \bigg) h({u}_{m}|_{[0,s]})\bigg]=0.
					\end{align}
					Thus, uniqueness of laws gives
					\begin{align}
						\wtilde{\E}\bigg[ & \bigg((\wtilde{M}_m(t),\psi)(\wtilde{M}_m(t),\xi) - (\wtilde{M}_m(s),\psi)(\wtilde{M}_m(s),\xi)  \\
						& - \sum_{i=1}^M\int_s^t (S_{m-1} \Nn_i (\wtilde{u}_m), \psi) (S_{m-1} \Nn_i(\wtilde{u}_m), \xi)  d\sigma \bigg) h(\wtilde{u}_{m}|_{[0,s]})\bigg]=0,\label{eqn-mar-zero-1}
					\end{align}
					which implies \eqref{eqn-quadratic} and hence concludes the proof of the lemma.}
			\end{proof}
			
			\begin{lemma}
				For any $t\in[0,T]$, let us define a process $\wtilde{M}$ as
				\begin{align}
					\wtilde{M}(t)& := \wtilde{u}(t) -\wtilde{u}(0) -\int_0^t \bigg( \Delta \wtilde{u}(s) -|\wtilde{u}(s)|^{p-2}\wtilde{u}(s) + \big( \norm{\nabla \wtilde{u}(s)}_{L^2(\bO)}^2 + \norm{\wtilde{u}(s)}_{L^p(\bO)}^p \big) \wtilde{u}(s)\\ 
					& \qquad\qquad\qquad\qquad \quad + \frac{1}{2}\sum_{i=1}^M\kappa_i(\wtilde{u}(s)) \bigg) ds.\label{eqn-martingale-1}
				\end{align}
				Then, $\wtilde{M}$ is an $L^2(\bO)-$valued continuous process.
			\end{lemma}
			\begin{proof}
				{Let us recall from \eqref{eqn-cgs} that $\wtilde{u}\in C([0, T]; L^2(\bO))$. It suffices to verify that the remaining four terms in \eqref{eqn-martingale-1} are well-defined in $L^2(\bO)$.}
				
				Using H\"older's inequality and the estimates \eqref{eqn-new-en-3}, \eqref{eqn-manifold} and the fact \eqref{eqn-kappa-est} for $p=2$, we have
				\begin{align}
					& \wtilde{\E} \bigg[\int_0^T\|A\wtilde{u}(t)\|_{L^2(\bO)}dt\bigg] \leq T^{\frac{1}{2}} \bigg( \wtilde{\E} \bigg[ \int_0^T \|A\wtilde{u}(t)\|_{L^2(\bO)}^2dt \bigg] \bigg)^{\frac{1}{2}} < \infty,\\
					& \wtilde{\E} \bigg[ \int_0^T \||\wtilde{u}(t)|^{p-2}\wtilde{u}(t)\|_{L^2(\bO)} dt \bigg]
					\leq T^{\frac{1}{2}} \bigg(\wtilde{\E} \bigg[ \int_0^T \|\wtilde{u}(t)\|_{L^{2p-2}(\bO)}^{2p-2} dt \bigg]\bigg)^{\frac{1}{2}}<\infty,\\
					& \wtilde{\E} \bigg[ \int_0^T \big( \norm{\nabla \wtilde{u}(t)}_{L^2(\bO)}^2 + \norm{\wtilde{u}(t)}_{L^p(\bO)}^p \big) \|\wtilde{u}(t)\|_{L^2(\bO)}dt\bigg]\\
					& \quad \leq T\wtilde{\E}\Big[\sup_{t\in[0,T]} \big( \norm{\nabla \wtilde{u}(t)}_{L^2(\bO)}^2 + \norm{\wtilde{u}(t)}_{L^p(\bO)}^p \big)\Big]<\infty, \\
					& \wtilde{\E}\bigg[\int_0^T\bigg\|\sum_{i=1}^M\kappa_i(\wtilde{u}(t))\bigg\|_{L^2(\bO)}dt\bigg] \leq \sum_{i=1}^M \wtilde{\E} \bigg[ \int_0^T\|\kappa_i(\wtilde{u}_i(t))\|_{L^2(\bO)}dt \bigg]\\
					& \quad \leq \sum_{i=1}^M\wtilde{\E}\bigg[ \int_0^T\|f_i\|_{L^2(\bO)}^2(1+ \|\wtilde{u}(t)\|_{L^2(\bO)}^2)\|\wtilde{u}(t)\|_{L^2(\bO)}dt \bigg]  = 2{T} \sum_{i=1}^M\|f_i\|_{L^2(\bO)}^2<\infty.
				\end{align}
				It completes the proof.
			\end{proof}
			\noindent
			\textbf{Step III.} {Now, we shift our focus towards applying the Lemma \ref{lem-strong-conv} in order to prove the convergence of increments stated in the following proposition.}
			Let us fix a bounded continuous function $h$ on $C([0, T]; L^2(\bO))$ and let $\wtilde{\Fb} = \{\wtilde{\Fn}_t\}$ denote the filtration generated by the process $\wtilde{u}$, where $\{\wtilde{\Fn}_t\}=\sigma\{\wtilde{u}(s),s\leq t\}$.

			\begin{lemma}\label{lem-martingale-1}
				Let us recall the processes $\wtilde{M}_m$ and $\wtilde{M}$ defined in \eqref{eqn-martingale} and \eqref{eqn-martingale-1}, respectively. Then, for fixed $\psi \in L^2(\bO)$ and $0\le s\leq t \le T$
				\begin{align}\label{eqn-mar}
					\lim_{m\to\infty} \wtilde{\E}\big[(\wtilde{M}_m(t) -\wtilde{M}_m(s),\psi)h(\wtilde{u}_{m}|_{[0,s]})\big] = \wtilde{\E}\big[(\wtilde{M}(t) -\wtilde{M}(s),\psi)h(\wtilde{u}|_{[0,s]})\big].
				\end{align}
			\end{lemma}
			\begin{proof}
				Let us fix $s, t\in[0, T]$, $s\leq t$ and $\psi \in \Vp$. Then, by \eqref{eqn-martingale}, we have
				\begin{align*}
					&(\wtilde{M}_m(t) -\wtilde{M}_m(s),\psi)\\
					& = (\wtilde{u}_m(t) -\wtilde{u}_m(s), \psi) + \int_s^t (A \wtilde{u}_m(\sigma), \psi) d\sigma + \int_s^t (P_m(|\wtilde{u}_m(\sigma)|^{p-2}\wtilde{u}_m(\sigma)), \psi) d\sigma\\
					& \quad- \int_s^t \bigg( \big( \big( \norm{\nabla \wtilde{u}_m(\sigma)}_{L^2(\bO)}^2 + \norm{\wtilde{u}_m(\sigma)}_{L^p(\bO)}^p\big) \wtilde{u}_m(\sigma), \psi \big) + \frac{1}{2}\sum_{i=1}^M \big(S_{m-1}\kappa_i(\wtilde{u}_m(\sigma)), \psi\big) \bigg) d\sigma\\
					& = {(\wtilde{u}_m(t) -\wtilde{u}_m(s), P_m\psi) + \int_s^t (A \wtilde{u}_m(\sigma), P_m\psi) d\sigma + \int_s^t (P_m(|\wtilde{u}_m(\sigma)|^{p-2}\wtilde{u}_m(\sigma)), P_m\psi) d\sigma}\\
					& \quad {- \int_s^t \bigg( \big( \norm{\nabla \wtilde{u}_m(\sigma)}_{L^2(\bO)}^2 + \norm{\wtilde{u}_m(\sigma)}_{L^p(\bO)}^p\big) (\wtilde{u}_m(\sigma), P_m\psi ) + \frac{1}{2}\sum_{i=1}^M \big(\kappa_i(\wtilde{u}_m(\sigma)), S_{m-1}\psi\big)\bigg) d\sigma}.
				\end{align*}
				{Thus, by taking limits on both sides and applying Lemma \ref{lem-strong-conv}, we obtain}
				\begin{align}\label{eqn-man-conv}
					\lim_{m\to\infty}(\wtilde{M}_m(t) -\wtilde{M}_m(s),\psi)=(\wtilde{M}(t) -\wtilde{M}(s),\psi),\ 	\wtilde{\Pr} - \text{a.s.}
				\end{align}
				To complete the proof by verifying that \eqref{eqn-mar} holds, we first note that the convergence established in \eqref{eqn-cgs}, namely
				$\wtilde{u}_m\to \wtilde{u}$ in $C([0, T];L^2(\bO))$ combined with the fact that $h$ is a bounded continuous function on $C([0, T]; L^2(\bO))$, implies that
				\begin{align}\label{eqn-h-conv}
					\lim_{m\to\infty}	h(\wtilde{u}_{m}|_{[0,s]})=h(\wtilde{u}|_{[0,s]}), \ \wtilde{\Pr} - \text{a.s.},
				\end{align}
				and
				\begin{align}
					{\sup_{m \in \N} \wtilde{\E} \big[ |h(\wtilde{u}_{m}|_{[0,s]})| \big] \le} \sup_{m\in\N}|h(\wtilde{u}_{m}|_{[0,s]})|_{L^{\infty}(\wtilde{\Omega})}<\infty.
				\end{align}
				Next, we define an $\mathbb{R}-$valued sequence of random variables:
				\begin{align}\label{Def-f_m}
					f_m(\omega):=(\wtilde{M}_m(t,\omega) -\wtilde{M}_m(s,\omega),\psi)h(\wtilde{u}_{m}|_{[0,s]}(\omega)),\ \omega\in\wtilde{\Omega}.
				\end{align}
				In preparation for applying Vitali's Theorem, our aim is to show  that the functions $f_m$, are uniformly integrable, i.e., 
				\begin{align}\label{eqn-uniform}
					\sup_{m\in\N}\wtilde{\E}\left[|f_m|^2\right]<\infty.
				\end{align}
				{Using the definition of $f_m$ from \eqref{Def-f_m} and applying the Cauchy-Schwarz inequality, we obtain}
				\begin{align}\label{eqn-uniform-1}
					\wtilde{\E}\left[|f_m|^2\right]\leq 2\|h\|_{L^{\infty}(\wtilde{\Omega})}\|\psi\|_{L^2(\bO)}^2\wtilde{\E}\big[\|\wtilde{M}_m(t)\|_{L^2(\bO)}^2+ \|\wtilde{M}_m(s)\|_{L^2(\bO)}^2\big].
				\end{align}
				On the other hand, from Lemma \ref{lem-martingale}, we know that the process $\wtilde{M}_m$ is a continuous martingale with quadratic variation given by \eqref{eqn-quadratic}, and that the operator ${S_{m-1}}:L^2(\bO)\to L^2(\bO)$ is a contraction. Combining these facts with the Burkholder-Davis-Gundy inequality, the estimate \eqref{B-L^p-bound} for $p=2$, and the property that $\wtilde{u}_m(t)\in\bM$ for all $t\in[0,T]$, we deduce
				\begin{align}\label{eqn-quad}
					\wtilde{\E}\bigg[\sup_{t\in[0,T]}\|\wtilde{M}_m(t)\|_{L^2(\bO)}^2\bigg]
					& {= \wtilde{\E} \bigg[ \sup_{t\in[0,T]} \bigg\| \sum_{i=1}^M\int_0^t {S_{m-1}}\Nn_i(u_m(s)) d W_i(s)\bigg\|_{L^2(\bO)}^2 \bigg]}\\
					& \leq C\wtilde{\E} \bigg[ \sum_{i=1}^M\int_0^T\|{S_{m-1}}\Nn_i(u_m(t))\|_{L^2(\bO)}^2dt\bigg]\\ 
					& \leq C\wtilde{\E}\bigg[\sum_{i=1}^M\int_0^T\|f_i\|_{L^2(\bO)}^2 (1+ \|u_m(t)\|_{L^2(\bO)}^2)^2\bigg]\\ 
					& \leq CT\sum_{i=1}^M\|f_i\|_{L^2(\bO)}^2<\infty.
				\end{align}
				{Therefore, substituting \eqref{eqn-quad} into \eqref{eqn-uniform-1}, and noting that the right-hand side does not dependent on $m$, we immediately obtain the desired uniform bound \eqref{eqn-uniform}. As a result, the sequence $\{f_m\}_{m\in\N}$ is uniformly integrable. Combined with the $\wtilde{\Pr} - $a.s. pointwise convergence from in \eqref{eqn-man-conv}, the proof of the lemma follows from an application of Vitali's Theorem.}
			\end{proof}
			\noindent
			\textbf{Step IV.} {As a byproduct of Lemmas \ref{lem-martingale} and \ref{lem-martingale-1}, we deduce the following result: the limiting process $\wtilde{u}$ satisfies the martingale property.}
			\begin{corollary}\label{cor-mart-zero}
				Assume $\wtilde{M}$ be the process defined in \eqref{eqn-martingale-1}. Then, for all $0\le s\leq t \le T$, we have
				\begin{align}
					\wtilde{\E}\big(\wtilde{M}(t) -\wtilde{M}(s)\big|\wtilde{\Fn}_t\big)=0.
				\end{align}
			\end{corollary}
			
			\begin{lemma}\label{lem-mart-zero}
				Let {us consider the processes} $\wtilde{M}_m$ and $\wtilde{M}$ defined in \eqref{eqn-martingale} and \eqref{eqn-martingale-1}, respectively. Then, for $\psi,\xi\in L^2(\bO)$ and $0\le s\leq t\le T$
				\begin{align}
					& \lim_{m\to\infty}\wtilde{\E}\big[\big((\wtilde{M}_m(t),\psi)(\wtilde{M}_m(t),\xi) - (\wtilde{M}_m(s),\psi)(\wtilde{M}_m(s),\xi)\big)h(\wtilde{u}_{m}|_{[0,s]})\big]\\
					& = \wtilde{\E}\big[\big((\wtilde{M}(t),\psi)(\wtilde{M}(t),\xi) - (\wtilde{M}(s),\psi)(\wtilde{M}(s),\xi)\big)h(\wtilde{u}|_{[0,s]})\big].
				\end{align}
			\end{lemma}
			\begin{proof}
				The spirit of the proof is similar to the proof of the Lemma \ref{lem-martingale-1}.
				
				We begin by choosing and fixing $0 \le s\leq t\le T$, and $\psi,\xi\in L^2(\bO)$. Then, let us define random variables $g_m$ and $g$ as follows:
				\begin{align}
					g_m(\omega)&:=\big((\wtilde{M}_m(t,\omega),\psi)(\wtilde{M}_m(t,\omega),\xi) - (\wtilde{M}_m(s,\omega),\psi)(\wtilde{M}_m(s,\omega),\xi)\big)h(\wtilde{u}_{m}|_{[0,s]}(\omega)),\\
					g(\omega)&:=\big((\wtilde{M}(t,\omega),\psi)(\wtilde{M}(t,\omega),\xi) - (\wtilde{M}(s,\omega),\psi)(\wtilde{M}(s,\omega),\xi)\big)h(\wtilde{u}|_{[0,s]}(\omega)), \ \omega\in\wtilde{\Omega}.
				\end{align}
				{Observe that from the convergences \eqref{eqn-man-conv} and \eqref{eqn-h-conv} obtained in Lemma \ref{lem-martingale-1}, it follows that} $$\lim_{m\to\infty}g_m(\omega)=g(\omega),\ \text{ for }\ \wtilde{\Pr} - \text{a.e. }\ \omega\in\wtilde{\Omega}.$$ 
				We aim to demonstrate that the random variables {$g_m$'s} are uniformly integrable, i.e.,
				\begin{align}\label{eqn-uniform-2}
					\sup_{m\in\N}\wtilde{\E}\left[|g_m|^2\right]<\infty.
				\end{align}
				{Let us fix $m\in\N$. Then}, a calculation similar to \eqref{eqn-uniform-1} yields
				\begin{align}
					\wtilde{\E}\left[|g_m|^2\right]\leq 2\|h\|_{L^{\infty}(\wtilde{\Omega})}^2\|\psi\|_{L^2(\bO)}^2\|\xi\|_{L^2(\bO)}^2\wtilde{\E}\left[\|\wtilde{M}_m(t)\|_{L^2(\bO)}^4+ \|\wtilde{M}_m(s)\|_{L^2(\bO)}^4\right]
				\end{align}
				{However, from Lemma \ref{lem-martingale}, we know that the process $\wtilde{M}_m$ is a continuous martingale with quadratic variation given by \eqref{eqn-quadratic}, and that the operator ${S_{m-1}}:L^2(\bO)\to L^2(\bO)$ is a contraction.  Therefore, applying the Burkholder-Davis-Gundy inequality, the estimate \eqref{B-L^p-bound} for $p=2$, and using the fact that $\wtilde{u}_m(t)\in\bM$ for all $t\in[0,T]$, we infer}
				\begin{align}\label{eqn-quad-1}
					\wtilde{\E}\bigg[\sup_{t\in[0,T]}\|\wtilde{M}_m(t)\|_{L^2(\bO)}^4\bigg]
					& {= \wtilde{\E} \bigg[ \sup_{t\in[0,T]} \bigg\| \sum_{i=1}^M\int_0^t S_{m-1}\Nn_i(u_m(s)) d W_i(s)\bigg\|_{L^2(\bO)}^4 \bigg]}\\
					& \leq C\wtilde{\E}\bigg[\bigg(\sum_{i=1}^M\int_0^T\|{S_{m-1}}\Nn_i(u_m(t))\|_{L^2(\bO)}^2dt\bigg)^2 \bigg]\\ 
					& \leq C\wtilde{\E}\bigg[\bigg(\sum_{i=1}^M\int_0^T\|f_i\|_{L^2(\bO)}^2(1+ \|u_m(t)\|_{L^2(\bO)}^2)^2dt\bigg)^2\bigg]\\ 
					& \leq CT^2\big(\sum_{i=1}^M\|f_i\|_{L^2(\bO)}^2\big)^2<\infty.
				\end{align}
				Using the above two estimates, it is immediate that  \eqref{eqn-uniform-2} holds true. Hence, by the Vitali Theorem, we have
				\begin{align}
					\lim_{m\to\infty}\wtilde{\E}\left[g_m\right] = \wtilde{\E} \left[g\right],
				\end{align}
				which completes the proof.
			\end{proof}
			
			\begin{lemma}[Convergence of the quadratic term]\label{lem-mart-conv}
				Let {$0\le s\le t\le T$} and $\psi,\xi\in \Vp$. Then, for {$h\in C_b(C([0, T]; L^2(\bO)); \R)$}
				\begin{align}
					& \lim_{m\to\infty} \wtilde{\E}\bigg[\bigg(\sum_{i=1}^M\int_s^t (S_{m-1} \Nn_i (\wtilde{u}_m), \psi) (S_{m-1} \Nn_i(\wtilde{u}_m), \xi) d\sigma\bigg) h(\wtilde{u}_{m}|_{[0,s]})\bigg]\\
					& = \wtilde{\E}\bigg[\bigg(\sum_{i=1}^M\int_s^t ( \Nn_i (\wtilde{u}), \psi) ( \Nn_i(\wtilde{u}), \xi) d\sigma\bigg) h(\wtilde{u}|_{[0,s]})\bigg].
				\end{align}
			\end{lemma}
			
			\begin{proof}
				First, we select and fix $\psi,\xi\in L^2(\bO)$ and consider the following sequence of random variables by
				\begin{align}
					\ell_m(\omega):=\bigg(\sum_{i=1}^M\int_s^t (S_{m-1} \Nn_i (\wtilde{u}_m), \psi) (S_{m-1} \Nn_i(\wtilde{u}_m), \xi) d\sigma\bigg) h(\wtilde{u}_{m}|_{[0,s]}(\omega)),\  \omega\in\wtilde{\Omega}.
				\end{align}
				Our next goal is to {show that the random variables $\ell_m$'s are uniformly integrable and converges pointwise to some random variable $\ell(\omega)$, $\wtilde{\Pr} - $a.s}. Let us first show that
				\begin{align}\label{eqn-uniform-3}
					\sup_{m\in\N}\wtilde{\E}\left[|\ell_m|^2\right]<\infty.
				\end{align}
				Applying the Cauchy-Schwarz inequality, using the fact that $\wtilde{u}_m(t)\in\bM$ for all $t\in[0,T]$ and \eqref{B-L^p-bound} for $p=2$, we deduce
				\begin{align}
					\wtilde{\E}\left[|\ell_m|^2\right]
					& = \wtilde{\E} \bigg[ \bigg( \bigg( \sum_{i=1}^M \int_s^t   (S_{m-1} \Nn_i (\wtilde{u}_m), \psi) (S_{m-1} \Nn_i(\wtilde{u}_m), \xi) d\sigma\bigg) h(\wtilde{u}_{m}|_{[0,s]})\bigg)^2\bigg]\\ 
					& \leq \|h\|_{L^{\infty} (\wtilde{\Omega})}^2 \wtilde{\E} \bigg[ \bigg(\sum_{i=1}^M \int_s^t \|S_{m-1} \Nn_i(\wtilde{u}_m)\|_{L^2(\bO)}^2\|\psi\|_{L^2(\bO)} \|\xi\|_{L^2(\bO)} d\sigma \bigg)^2\bigg]\\ 
					& \leq \|h\|_{L^{\infty}(\wtilde{\Omega})}^2 \|\psi\|_{L^2(\bO)}^2\|\xi\|_{L^2(\bO)}^2\wtilde{\E}\bigg[\bigg(\sum_{i=1}^M\int_s^t \|f_i\|_{L^2(\bO)}^2 \big( 1 + \|u_m(\sigma)\|_{L^2(\bO)}^2 \big)^2 d\sigma\bigg)^2\bigg]\\ 
					& \leq 16T^2 \|h\|_{L^{\infty}(\wtilde{\Omega})}^2 \|\psi\|_{L^2(\bO)}^2 \|\xi\|_{L^2(\bO)}^2 \bigg(\sum_{i=1}^M\|f_i\|_{L^2(\bO)}^2\bigg)^2 < \infty,
				\end{align}
				so that \eqref{eqn-uniform-3} holds true. 
				
				In order to show pointwise convergence,  for each fixed $\omega\in\wtilde{\Omega}$, we show the following limit holds
				\begin{equation}
					\lim_{m\to\infty} \sum_{i=1}^M \int_s^t  (S_{m-1} \Nn_i (\wtilde{u}_m), \psi) (S_{m-1} \Nn_i(\wtilde{u}_m), \xi) d\sigma 
					= \sum_{i=1}^M \int_s^t  ( \Nn_i (\wtilde{u}), \psi) ( \Nn_i(\wtilde{u}), \xi)  d\sigma.\label{eqn-noise-conv}
				\end{equation}
				Let us fix a $\omega\in\wtilde{\Omega}$ such that
				$\wtilde{u}_m(\cdot,\omega)\to \wtilde{u}(\cdot,\omega)$ in $C([0,T];L^2(\bO))$. Note that for all $m\in\N$, $\|\wtilde{u}_m(t,\omega)\|_{L^2(\bO)}=1$ for all $t\in[0,T]$, $\wtilde{\Pr} - $a.e. $\omega\in\wtilde{\Omega}$. In order to prove \eqref{eqn-noise-conv}, it is sufficient to show that
				\begin{equation}
					S_{m-1} \Nn_i (\wtilde{u}_m) \to \Nn_i(\wtilde{u})\ \text{ in } \ L^2(s,t;L^2(\bO)),\ \text{ for every }\ 1\le i \le M.
				\end{equation} 
				It can be seen by the definition of $\Nn_i$, see \eqref{eqn-def-Nn}, the property $\wtilde{u}_m(t), \wtilde{u}(t)\in \bM$ for $t\in[0,T]$ and the part (v) of Proposition \ref{Prop-P_m} that
				\begin{align}
					& \int_s^t\| S_{m-1} \Nn_i (\wtilde{u}_m(\sigma)) - \Nn_i(\wtilde{u}(\sigma)) \|_{L^2(\bO)}^2d\sigma \\
					&= \int_s^t \| S_{m-1} f_i - (f_i, \wtilde{u}_m(\sigma)) (S_{m-1}\wtilde{u}_m (\sigma) - \wtilde{u} (\sigma)) - f_i + (f_i, (\wtilde{u}(\sigma) - \wtilde{u}_m (\sigma)))\wtilde{u}(\sigma) \|_{L^2(\bO)}^2d\sigma\\
					&\le 2 \int_s^t \big(\| S_{m-1} f_i - f_i \|_{L^2(\bO)}^2 + 2 \|f_i\|_{L^2(\bO)}^2 \|\wtilde{u} (\sigma)\|_{L^2(\bO)}^2 \|\wtilde{u}_m(\sigma)  - \wtilde{u}(\sigma))\|_{L^2(\bO)}^2\big) d\sigma\\
					& \quad  + 4 \int_s^t \|f_i\|_{L^2(\bO)}^2 \|\wtilde{u}_m (\sigma)\|_{L^2(\bO)}^2 {\|S_{m-1} \wtilde{u}_m(\sigma) - S_{m-1}\wtilde{u}(\sigma) + S_{m-1}\wtilde{u}(\sigma)  - \wtilde{u}(\sigma))\|_{L^2(\bO)}^2} d\sigma\\
					&\le 2 \int_s^t \big(\| S_{m-1} f_i - f_i \|_{L^2(\bO)}^2 + 2 \|f_i\|_{L^2(\bO)}^2 \|\wtilde{u} (\sigma)\|_{L^2(\bO)}^2 \|\wtilde{u}_m(\sigma)  - \wtilde{u}(\sigma))\|_{L^2(\bO)}^2\big) d\sigma\\
					& \quad  + \int_s^t \|f_i\|_{L^2(\bO)}^2 \|\wtilde{u}_m (\sigma)\|_{L^2(\bO)}^2 { \|\wtilde{u}_m- \wtilde{u}(\sigma)\|_{L^2(\bO)}^2 +  \| S_{m-1}\wtilde{u}(\sigma)  - \wtilde{u}(\sigma)\|_{L^2(\bO)}^2 \big)} d\sigma\\
					& \to 0\ \text{ as } m\to\infty,
				\end{align}
				which completes the proof.
			\end{proof}
			\noindent
			\textbf{Step V.} {Note that by applying Lemma \ref{lem-martingale-1} together with Lemmas \ref{lem-mart-zero} and \ref{lem-mart-conv}, we can pass to the limit in \eqref{eqn-mar-zero} and \eqref{eqn-mar-zero-1}, respectively. After taking these limits, we deduce that for all $\psi,\xi\in L^2(\bO)$:}
			\begin{equation*}
				\wtilde{\E} \big[(\wtilde{M}(t) -\wtilde{M}(s),\psi)h(\wtilde{u}|_{[0,s]}) \big] = 0,
			\end{equation*}
			and
			\begin{equation*}
				\wtilde{\E}\bigg[(\wtilde{M}(t),\psi)(\wtilde{M}(t),\xi) - (\wtilde{M}(s),\psi)(\wtilde{M}(s),\xi) - \sum_{i=1}^M\int_s^t (\Nn_i(\wtilde{u}), \psi)(\Nn_i(\wtilde{u}),\xi))d\sigma  h(\wtilde{u}|_{[0,s]})\bigg]=0.
			\end{equation*}
			We now state a result that is derived from Lemmas \ref{lem-martingale}, \ref{lem-mart-zero} and \ref{lem-mart-conv}.
			
			\begin{corollary}\label{cor-mart}
				Let $\wtilde{M}$ be the process defined in \eqref{eqn-martingale-1}. Then, for $t\in[0, T]$, the quadratic variation of $\wtilde{M}$ is given by
				\begin{align}\label{eqn-quadratic-new}
					\llbracket \wtilde{M}\rrbracket_t=\int_0^t\sum_{i=1}^M\|\Nn_i(u(s))\|_{L^2(\bO)}^2ds.
				\end{align}
			\end{corollary}

			\begin{proof}[Proof of Proposition \ref{Prop-pre-martingale}]
				Observe from Lemma \ref{lem-strong-conv} and Corollary \ref{cor-mart-zero} that an $L^2(\bO) -\break$valued continuous square integrable process $\{\wtilde{M}(t)\}_{ t\in[0, T]}$ is a martingale with respect to the filtration $\wtilde{\Fb}=\{\Fn\}_{t\geq 0}$. Furthermore, it follows from Corollary \ref{cor-mart} that the process $\wtilde{M}$ has the quadratic variation as described in \eqref{eqn-quadratic-new}. Thus, an application Martingale representation Theorem, {see \cite[Section 8.2]{GDP+JZ-14}}, asserts that there exist
				\begin{itemize}
					\item[$\boldsymbol{\ast}$] a stochastic basis $(\wtilde{\wtilde{\Omega}},\wtilde{\wtilde{\Fn}},\wtilde{\wtilde{\Fb}},\wtilde{\wtilde{\Pr}})$,
					\item[$\boldsymbol{\ast}$] an $\mathbb{R}^M-$valued $\wtilde{\wtilde{\Fb}} - $Wiener process 		$\wtilde{\wtilde{W}}$,
					\item[$\boldsymbol{\ast}$] and a $\wtilde{\wtilde{\Fb}} - $progressively measurable process $u :[0, T] \times\wtilde{\wtilde{\Omega}}\to D(A)$ with $\wtilde{\wtilde{\Pr}} - $a.e. paths
					$$\wtilde{\wtilde{u}}(\cdot,\omega)\in C_w([0,T]; \Vp)\cap L^2(0,T;D(A)) {\cap L^{2p-2}(0,T;L^{2p-2}(\bO))}$$
					such that for all $t\in[0, T]$ and all $\psi\in L^2(\bO)$, $\wtilde{\wtilde{\Pr}} - $a.s.
					\begin{align*}
						(\wtilde{\wtilde{u}}(t),\psi) & = (\wtilde{\wtilde{u}}_0,\psi) + \int_0^t \Big ( \Delta \wtilde{\wtilde{u}}(s) - |\wtilde{\wtilde{u}}(s)|^{p-2}\wtilde{\wtilde{u}}(s) + \big( \|\nabla \wtilde{\wtilde{u}}(s)\|_{L^2(\bO)}^2 + \|\wtilde{\wtilde{u}}(s)\|_{L^p(\bO)}^p \big) \wtilde{\wtilde{u}}(s), \psi \Big)ds\\
						& \quad + \frac{1}{2}\sum_{i=1}^M\int_0^t(\kappa_i(\wtilde{\wtilde{u}}(s)),\psi)ds+ \sum_{i=1}^M\int_0^t(\Nn_i(\wtilde{\wtilde{u}}(s)),\wtilde{\wtilde{u}}(s))d\wtilde{\wtilde{W}}_i(s). 
					\end{align*}
				\end{itemize}
				Hence, the hypothesis of Definition \ref{Def-Martingale} hold provided $$(\widehat{\Omega}, \widehat{\Fn}, \{\widehat{\Fn}_t\}_{t\geq 0}, \what{\Pr}) = (\wtilde{\wtilde{\Omega}}, \wtilde{\wtilde{\Fn}}, \{\wtilde{\wtilde{\Fn}}_t\}_{t\geq 0},\wtilde{\wtilde{\Pr}}),$$ $\widehat{W}=\wtilde{\wtilde{W}}$ and $\widehat{u}=\wtilde{\wtilde{u}}$, which completes the proof of Proposition \ref{Prop-pre-martingale}.
			\end{proof}

			Previously, in Proposition \ref{Prop-pre-martingale}, we have shown that the paths of the weak solution $u$ to \eqref{eqn-main-prob-Ito} lie in the space 
			$$C_w([0,T]; \Vp)\cap L^2(0,T;D(A)) \cap L^{2p-2}(0,T;L^{2p-2}(\bO)).$$
			Next, with the help of the It\^o formula, we aim to establish that the weak solution actually lies in the above space but equipped with the strong topology. 
			
			\vspace{2mm}
			\noindent
			\textbf{Step VI.} We recall that the space $\Vp$ is the intersection of $H_0^1(\bO)$ and $L^p(\bO)$, for $p\in [2,\infty)$. As a first step, we establish the $L^p(\bO)-$It\^o formula for the weak solution of the SPDE \eqref{eqn-main-prob-Ito}, i.e., the proof of Lemma \ref{Lem-Ito} after stating an important remark as follows. 
			\begin{remark}
				As mentioned in the introduction, see Remark \ref{Rmk-Krylov}, we elaborate here that the Krylov’s result \cite[Lemma 5.1]{NVK-10} does not apply to nonlinear problems without any restriction to the nonlinearity exponent; therefore, the original argument of our problem does not carry over. Consequently, the full regime of our setting cannot be covered by Krylov’s framework, for instance, consider the nonlinear term $f= |u|^{p-2}u$, then by Krylov’s result it should hold that 
				\begin{equation*}
					f= |u|^{p-2}u \in L^p(\bO) \ \text{ which equivalently requires }\ u \in L^{p^2 -p}(\bO).
				\end{equation*}			
				But, for $p\ge 2$, the embedding $L^p(\bO) \embed L^{p^2 -p}(\bO)$ only holds for $p \le 2$. Therefore, the Krylov's result is applicable, but with the restriction $p=2$. Hence it does not cover the whole regime, i.e., $2 \le p < \infty$, in our settings.
			\end{remark}
			
			\subsubsection{$L^p-$It\^o formula}
			Next, we aim to prove the $L^p-$It\^o formula. The proof motivated from the theme of Krylov's work, see \cite[Lemma 5.1]{NVK-10}, here we do not use regularization technique, instead, we exploit the sequence of self-adjoint operators $\{S_m\}_{m\in\N}$ available with us, see Proposition \ref{Prop-S_m}.

			\begin{proof}[Proof of Lemma \ref{Lem-Ito}]
				{Assume $u_0\in \Vp\cap \bM$, for any $2\le p < \infty$, be fixed. Suppose $u$ denotes a weak solution, guaranteed by Proposition \ref{Prop-pre-martingale}, that satisfies the following equation for all $t\in [0,T]$ $\Pr-$a.s. in $L^2(\bO)$:}
				\begin{align}
					u(t) & = u(0) + \int_0^t \bigg( \Delta u(s) - |u(s)|^{p-2}u(s) + \big( \norm{\nabla u(s)}_{L^2(\bO)}^2 + \norm{u(s)}_{L^p(\bO)}^p \big) u(s)\\ 
					& \qquad\qquad\qquad\quad + \frac{1}{2}\sum_{i=1}^M \kappa_i(u(s)) \bigg) ds + \sum_{i=1}^M\int_0^t \Nn_i({u}(s))d{{W}}_i(s).\label{eqn-ito-1}
				\end{align}
				\textbf{1.} Fix $m\in\N$. Thanks to the self-adjoint operator $S_m$, see Proposition \ref{Prop-S_m}, applying $S_m$ on both sides of \eqref{eqn-ito-1} yields, for every $t\in [0,T]$, $u_m = S_m u$ solves the following in $L_m^2(\bO)$, $\wtilde{\Pr} - $a.s.:
				\begin{align*}
					u_m(t) & = u_m(0) + \int_0^t \bigg( \Delta u_m(s) - S_m(|u(s)|^{p-2}u(s)) + \big( \norm{\nabla u(s)}_{L^2(\bO)}^2 + \norm{u(s)}_{L^p(\bO)}^p \big) u_m(s)  \\ 
					& \qquad\qquad\qquad\qquad + \frac{1}{2}\sum_{i=1}^M S_m\kappa_i(u(s)) \bigg) ds + \sum_{i=1}^M\int_0^t S_m\Nn_i({u}(s))d{{W}}_i(s).
				\end{align*}
				Apply the finite dimensional It\^o formula to the function $\R \ni v \mapsto \|v\|_{L^p(\bO)}^p \in \R$, since it is continuous for $p\in [2,\infty)$, and the process $u_m(\cdot)$, obtaining, $\wtilde{\Pr}-$a.s.
				\begin{align}
					\int_{\bO}	|u_m(t)|^p dx & = \int_{\bO}	|u_m(0)|^pdx\\
					& \quad + p \int_0^{t} \int_{\bO} |u_m(s)|^{p-2}u_m(s)\big(\Delta u_m(s) - S_m(|u(s)|^{p-2}u(s)) \big)dxds\\
					& \quad + p\int_0^{t}\int_{\bO}|u_m(s)|^{p} \big( \norm{\nabla u(s)}_{L^2(\bO)}^2 + \norm{u(s)}_{L^p(\bO)}^p\big)dxds\\
					& \quad + \frac{p}{2} \sum_{i=1}^M \int_0^{t} \int_{\bO}|u_m(s)|^{p-2}u_m(s) S_m (\kappa_i(u(s)))dxds\\
					& \quad + \frac{p(p-1)}{2}\sum_{i=1}^M \int_0^{t}\int_{\bO}|u_m(s)|^{p-2} (S_m (\Nn_i(u(s))))^2dxds\\
					& \quad + p\sum_{i=1}^M \int_0^{t} \int_{\bO}|u_m(s)|^{p-2}u_m(s) S_m (\Nn_i(u(s)))dx\,dW_i(s),\label{eqn-Ito-eps}
				\end{align}
				Next, our objective is to show that as $m \to \infty$, we obtain our desired It\^o formula \eqref{eqn-Ito-L^p}.
				\vskip 2mm
				\noindent 
				\textbf{2.} Since, for each fixed $m\in\N$, $u_m$ is finite dimensional function and the sequence of self-adjoint operators $\{S_m\}_{m\in\N}$ converges in $L^p(\bO)$ for $2\le p < \infty$, see Proposition \ref{Prop-S_m}, i.e., for all $t\in[0,T]$
				\begin{equation}\label{eqn-L^p-conv}
					S_m u(t) = u_m(t) \to u(t) \ \text{ in }\ L^p(\bO)\ \text{ as }\ m \to \infty,\ \wtilde{\Pr}-\text{a.s}.
				\end{equation}
				Therefore, the left-hand side term and the first term on the right-hand side of \eqref{eqn-Ito-eps} converge $\wtilde{\Pr}-$a.s. 
				Next, let us show that we can pass to the limit $m \to \infty$ in the remaining terms of the right hand side of \eqref{eqn-Ito-eps} $\wtilde{\Pr}-$a.s. 
				
				\vskip 1mm
				\noindent
				\textbf{(i)} By utilizing the $L^{2p-2}(\bO)-$convergence of $u_m$, see Remark \ref{Rmk-S_m-L^{2p-2}}, on a set of full probability for a.e. $t\in[0,T]$, we have as $m\to\infty$
				\begin{equation*}
					\||u_m(t)|^{p-2}u_m(t)\|_{L^{2}(\bO)} = \| u_m(t) \|_{L^{2p-2}(\bO)}  \to \| u(t) \|_{L^{2p-2}(\bO)} = \| |u(t)|^{p-2}u(t)\|_{L^{2}(\bO)}.
				\end{equation*}
				From the strong convergence \eqref{eqn-L^p-conv}, we also know that 
				\begin{equation*}
					|u_m(t)|^{p-2}u_m(t) \to |u(t)|^{p-2}u(t),\ \text{ for a.e. }\ t\in[0,T]\ \; \wtilde{\Pr}-\text{a.s.}
				\end{equation*}
				By the Lions Lemma \cite[Lemma 1.3]{JLL-69}, it follows $\wtilde{\Pr}-$a.s. that
				\begin{align}
					|u_m(t)|^{p-2}u_m(t) \rightharpoonup |u(t)|^{p-2}u(t) \ \text{ in }\ L^{2}(\bO), \ \text{ for a.e.} \ t\in [0,T].
				\end{align}
				Invoking the fact that the combination of norm and weak convergence implies strong convergence, we obtain
				\begin{align}
					|u_m(t)|^{p-2}u_m(t) \to |u(t)|^{p-2}u(t) \ \text{ in }\ L^{2}(\bO), \ \text{ for a.e.} \ t\in [0,T].\label{eqn-nonlin-conv-L^2}
				\end{align}
				Using H\"older's inequality twice in space and time (with exponents $2$ and $2$), we assert $\Pr-$a.s.
				\begin{align*}
					& \int_0^{t} \int_{\bO} |u_m(s)|^{p-2}u_m(s)\big(\Delta u_m(s) - S_m(|u(s)|^{p-2}u(s)) \big)dxds\\
					& \le \int_0^{t} \|u_m(s)\|_{L^{2p-2}(\bO)}^{p-1}\big( \|\Delta u_m(s)\|_{L^2(\bO)} + \|S_m(|u(s)|^{p-2}u(s))\|_{L^2(\bO)} \big)ds\\
					& \le \|S_m\|_{\Ls(L^{2p-2}(\bO))}^{p-1} \|S_m\|_{\Ls(L^2(\bO))} \int_0^{t} \|u(s)\|_{L^{2p-2}(\bO)}^{p-1} \big( \|\Delta u(s)\|_{L^2(\bO)} + \|u(s)\|_{L^{2p-2}(\bO)}^{p-1} \big)ds\\
					& \le C \|u\|_{L^{2p-2}(0,T; L^{2p-2}(\bO))}^{p-1} \big( \|u\|_{L^2(0,T; D(A))} + \|u\|_{L^{2p-2}(0,T; L^{2p-2}(\bO))}^{p-1} \big) < \infty,
				\end{align*}
				where $C$ depends on $\displaystyle \sup_{m\in\N}\|S_m\|_{\Ls(L^{2p-2}(\bO))} < \infty$, see \eqref{eqn-S_m-L^{2p-2}}.
				The H\"older inequality (with exponents $2$ and $2$), convergences \eqref{eqn-L^p-conv} and \eqref{eqn-nonlin-conv-L^2}, and the Lebesgue Dominated Convergence Theorem yield $\Pr-$a.s.
				\begin{align*}
					& \bigg| \int_0^{t} \int_{\bO} |u_{m}(s)|^{p-2}u_m(s)\Delta u_m(s)dxds - \int_0^{t} \int_{\bO} |u(s)|^{p-2}u(s)\Delta u(s)dxds \bigg|\\
					& \le \int_0^{t}  \big| (|u_m(s)|^{p-2}u_m(s) - |u(s)|^{p-2}u(s), \Delta u_m(s)) + (|u(s)|^{p-2}u(s),\Delta u_m(s) - \Delta u(s)) \big| ds\\
					& \le \|S_m\|_{\Ls(L^2(\bO))}\||u_m|^{p-2}u_m - |u|^{p-2}u\|_{L^2(0,T; L^2(\bO))}\|\Delta u\|_{L^2(0,T; L^2(\bO))}\\
					&\quad + \|u\|_{L^{2p-2}(0,T; L^{2p-2}(\bO))}^{p-1} \|S_m(\Delta u) - \Delta u\|_{L^2(0,T; L^2(\bO))} \\
					& \to 0, \ \text{ as }\ m\to \infty.
				\end{align*}
				Again, by H\"older's inequality (with exponents $2$ and $2$) and using the convergences \eqref{eqn-L^p-conv} and \eqref{eqn-nonlin-conv-L^2} along with the Lebesgue Dominated Convergence Theorem, we infer $\Pr-$a.s.
				\begin{align*}
					& \bigg| \int_0^{t} \int_{\bO} |u_{m}(s)|^{p-2}u_m(s) S_m(|u(s)|^{p-2}u(s)) dxds - \int_0^{t} \int_{\bO} (|u(s)|^{p-2}u(s))^2 dxds \bigg|\\
					& \le \int_0^{t}  \big|(|u_m(s)|^{p-2}u_m(s), S_m(|u(s)|^{p-2}u(s)) - |u(s)|^{p-2}u(s))\\
					&\quad\qquad + (|u(s)|^{p-2}u(s), |u_m(s)|^{p-2}u_m(s) - |u(s)|^{p-2}u(s))\big| ds\\
					& \le \|S_m\|_{\Ls(L^{2p-2}(\bO))}  \|u\|_{L^{2p-2}(0,T; L^{2p-2}(\bO))} \|S_m(|u|^{p-2}u) - |u|^{p-2}u\|_{L^2(0,T; L^2(\bO))}\\
					&\quad +\|u\|_{L^{2p-2}(0,T; L^{2p-2}(\bO))} \||u_m|^{p-2}u_m - |u|^{p-2}\|_{L^2(0,T; L^2(\bO))}\\
					& \to 0, \ \text{ as }\ m\to \infty. 
				\end{align*}
				\textbf{(ii)} Next, by applying the H\"older inequality (with the exponents $2$ and $2$) in time as well as in space, we deduce $\Pr-$a.s.
				\begin{align*}
					& \int_0^{t}\int_{\bO}|u_m(s)|^{p-2}u_m(s) \big( \norm{\nabla u(s)}_{L^2(\bO)}^2 + \norm{u(s)}_{L^p(\bO)}^p\big) u_m(s)dxds\\
					& \le \int_0^{t} \|u_m(s)\|_{L^{p}(\bO)}^{p} \big( \norm{\nabla u(s)}_{L^2(\bO)}^2 + \norm{u(s)}_{L^p(\bO)}^p\big)ds\\
					& \le  \|S_m\|_{\Ls(L^{p}(\bO))}^{p} \|u\|_{L^p(0,T; L^{p}(\bO))}^{p}\big(\|u\|_{L^{\infty}(0,T; H_0^1(\bO))}^{2} + \|u\|_{L^p(0,T; L^p(\bO))}^{p} \big)
					< \infty.
				\end{align*}
				Thus, by utilizing the convergence \eqref{eqn-L^p-conv} (strong convergence implies norm convergence) together with the Lebesgue Dominated Convergence Theorem, we obtain $\Pr-$a.s.
				\begin{align*}
					& \bigg|\int_0^{t}\big( \norm{\nabla u(s)}_{L^2(\bO)}^2 + \norm{u(s)}_{L^p(\bO)}^p\big)\int_{\bO} \big( |u_m(s)|^p  - |u(s)|^p\big) dxds \bigg|\\
					& \le \int_0^{t}\big( \norm{\nabla u(s)}_{L^2(\bO)}^2 + \norm{u(s)}_{L^p(\bO)}^p\big) \big| \norm{u_m(s)}_{L^p(\bO)}^p  - \norm{u(s)}_{L^p(\bO)}^p\big| ds\\
					& \le  \big(\|u\|_{L^{\infty}(0,T; H_0^1(\bO))}^{2} + \|u\|_{L^\infty(0,T; L^p(\bO))}^{p} \big) \int_0^{t}\big| \norm{u_m(s)}_{L^p(\bO)}^p  - \norm{u(s)}_{L^p(\bO)}^p\big|ds \\
					& \to 0, \ \text{ as }\ m \to \infty.
				\end{align*}
				\textbf{(iii)} Consider the fourth term of \eqref{eqn-Ito-eps} and use H\"older's inequality both in time and space (with exponents  $2$ and $2$) and the estimate \eqref{eqn-kappa-est} which $\Pr-$a.s. yield
				\begin{align*}
					& \int_0^{t} \int_{\bO}  |u_m(s)|^{p-2}u_m(s) S_m (\kappa_i(u(s))dxds\\
					& \le 2 \|S_m\|_{\Ls(L^{2p-2}(\bO))} \|f_i\|_{L^2(\bO)}^2 \int_0^{t}  \|u(s)\|_{L^{2p-2}(\bO)} \big(1 + \|u(s)\|_{L^{2}(\bO)}^2 \big)\|u(s)\|_{L^{2}(\bO)}ds\\
					& \le C  \|f_i\|_{L^2(\bO)}^2 \|u\|_{L^{2p-2}(0,T; L^{2p-2}(\bO))}^{p}
					< \infty,
				\end{align*}
				where $C$ depends on $T$ and $\displaystyle \sup_{m\in\N}\|S_m\|_{\Ls(L^{2p-2}(\bO))}$.
				Therefore, the Lebesgue Dominated Convergence Theorem $\Pr-$a.s. deduce
				\begin{align*}
					& \bigg| \int_0^{t} \int_{\bO} \big( |u_m(s)|^{p-2}u_m(s) S_m (\kappa_i(u(s))) - |u(s)|^{p-2}u(s) \kappa_i(u(s)) \big)dxds \bigg|\\
					& \le  \||u_m|^{p-2}u_m\|_{L^{2}(0,T; L^{2}(\bO)} \|S_m (\kappa_i(u)) -  \kappa_i(u)\|_{L^{2}(0,T; L^{2}(\bO))}\\
					&\quad + \||u_m|^{p-2}u_m - |u|^{p-2}u\|_{L^{2}(0,T; L^{2}(\bO)} \|\kappa_i(u)\|_{L^{2}(0,T; L^{2}(\bO))}\\
					& \le  \|S_m\|_{\Ls(L^{2p-2}(\bO))} \|u\|_{L^{2p-2}(0,T; L^{2p-2}(\bO)}^{p-1} \|S_m (\kappa_i(u)) -  \kappa_i(u)\|_{L^{2}(0,T; L^{2}(\bO))}\\
					&\quad + C\||u_m|^{p-2}u_m - |u|^{p-2}u\|_{L^{2}(0,T; L^{2}(\bO)} \|\kappa_i(u)\|_{L^{2}(0,T; L^{2}(\bO))}\\
					& \to 0, \ \text{ as }\ m \to \infty,
				\end{align*}
				where $C$ is a constant depending on $\displaystyle \sup_{m\in\N}\|S_m\|_{\Ls(L^{2p-2}(\bO))}$.
				
				\noindent
				\textbf{(iv)} As in the case (i), the $L^p-$convergence of $u_m$, on a set of full probability for a.e. $t\in[0,T]$ also yields as $m\to\infty$
				\begin{equation}
					\||u_m(t)|^{p-2}\|_{L^{\frac{p}{p-2}}(\bO)} = \| u_m(t) \|_{L^p(\bO)}  \to \| u(t) \|_{L^p(\bO)} = \| |u(t)|^{p-2}\|_{L^{\frac{p}{p-2}}(\bO)}.
				\end{equation}
				and from the strong convergence \eqref{eqn-L^p-conv}, we know that 
				\begin{equation*}
					|u_m(t)|^{p-2} \to |u(t)|^{p-2},\ \text{ for a.e. }\ t\in[0,T]\ \, \wtilde{\Pr}-\text{a.s.}
				\end{equation*}
				By the Lions Lemma \cite[Lemma 1.3]{JLL-69}, it follows $\wtilde{\Pr}-$a.s. that
				\begin{equation}
					|u_m(t)|^{p-2} \rightharpoonup |u(t)|^{p-2} \ \text{ in }\ L^{\frac{p}{p-2}}(\bO), \ \text{ for a.e.} \ t\in [0,T].
				\end{equation}
				Since the norm and weak convergence together implies strong convergence, it follows that
				\begin{equation}
					|u_m(t)|^{p-2} \to |u(t)|^{p-2} \ \text{ in }\ L^{\frac{p}{p-2}}(\bO), \ \text{ for a.e.} \ t\in [0,T].\label{eqn-conv-p/p-2}
				\end{equation}
				Similar to the previous case, an application of H\"older's inequality (with exponents $\frac{p}{p-2}$ and $\frac{p}{2}$)  space implies for all $t\in[0,T]$ $\Pr-$a.s.
				\begin{align*}
					& \int_0^{t}\int_{\bO}  |u_m(s)|^{p-2} (S_m \Nn_i(u(s)))^2 dxds
					\le  \|S_m\|_{\Ls(L^{p}(\bO))}^{p} \int_0^{t}  \|u(s)\|_{L^{p}(\bO)}^{p-2} \|\Nn_i(u(s))\|_{L^{p}(\bO)}^2 ds\\
					& \le 2\|S_m\|_{\Ls(L^{p}(\bO))}^{p} \|u\|_{L^{\infty}(0,T; L^p(\bO))}^{p-2} \int_0^t \big( \|f_i\|_{L^p(\bO)}^2 + \|f_i\|_{L^2(\bO)}^2\|u(s)\|_{L^{p}(\bO)}^2 \big) ds < \infty.
				\end{align*}
				Thus, by utilizing the convergence \eqref{eqn-conv-p/p-2} together with the Lebesgue Dominated Convergence Theorem, we can pass to the limit $m\to \infty$ inside the integral, i.e., $\Pr-$a.s.
				\begin{align*}
					& \bigg| \int_0^{t}\int_{\bO} \big( |u_m(s)|^{p-2} (S_m \Nn_i(u(s)))^2 - |u(s)|^{p-2} (\Nn_i(u(s)))^2\big)dxds \bigg|\\
					& \le \bigg| \int_0^{t}\int_{\bO}  |u_m(s)|^{p-2} \big((S_m \Nn_i(u(s)))^2 - (\Nn_i(u(s)))^2\big) dx ds \bigg|\\
					&\quad + \bigg| \int_0^{t}\int_{\bO} \big(|u_m(s)|^{p-2} - |u(s)|^{p-2}\big)(\Nn_i(u(s)))^2 dx ds \bigg|\\
					& \le \|u_m\|_{L^{p}(0,T; L^p(\bO))}^{p-2} \|(S_m \Nn_i(u))^2 - (\Nn_i(u))^2\|_{L^{\frac{p}{2}}(0,T; L^{\frac{p}{2}}(\bO))} \\
					& \quad + \||u_m|^{p-2} - |u|^{p-2}\|_{L^{\frac{p}{p-2}}(0,T; L^{\frac{p}{p-2}}(\bO))} \|\Nn_i(u)\|_{L^{p}(0,T; L^{p}(\bO))}^2 \\
					& \le C \|u\|_{L^{p}(0,T; L^p(\bO))}^{p-2} \|S_m \Nn_i(u) - \Nn_i(u)\|_{L^p(0,T; L^{p}(\bO))} \|S_m \Nn_i(u) + \Nn_i(u)\|_{L^p(0,T; L^{p}(\bO))} \\
					& \quad + \||u_m|^{p-2} - |u|^{p-2}\|_{L^{\frac{p}{p-2}}(0,T; L^{\frac{p}{p-2}}(\bO))} \|\Nn_i(u)\|_{L^{p}(0,T; L^{p}(\bO))}^2 \\
					& \to 0, \ \text{ as }\ m \to \infty, 
				\end{align*}
				where $C$ is the constant depending on $\sup_{m \in \N} \|S_m\|_{\Ls(L^p(\bO))}^{p-2}< \infty$.
				Hence, we are now left to demonstrate the convergence of the stochastic term. 
				\vskip 1mm
				\noindent
				\textbf{3.} For the convergence of the stochastic integral term in \eqref{eqn-Ito-eps}, due to It\^o isometry, we need to show that as $m \to \infty$ for every $t\in[0,T]$ $\Pr-$a.s. 
				\begin{align*}
					\sum_{i=1}^M \int_0^{t} \bigg( \int_{\bO} \big( |u_m(s)|^{p-2}u_m(s) S_m \Nn_i(u(s)) - |u(s)|^{p-2}u(s)\Nn_i(u(s)) \big) dx\bigg)^2 ds \to 0.
				\end{align*}
				Applying H\"older's inequality (with the exponents $p/(p-1)$ and $p$) to find out
				\begin{align}
					& \sum_{i=1}^M \int_0^{t} \bigg( \int_{\bO} \big( |u_m(s)|^{p-2}u_m(s) S_m \Nn_i(u(s)) - |u(s)|^{p-2}u(s)\Nn_i(u(s)) \big) dx\bigg)^2 ds\\
					& \le 2 \sum_{i=1}^M \int_0^t \big( \| S_m \Nn_i(u(s)) \|_{L^p(\bO)}^2  \|u_m(s) \|_{L^{p}(\bO)}^{2p-2} + \| \Nn_i(u(s)) \|_{L^p(\bO)}^2  \|u(s) \|_{L^{p}(\bO)}^{2p-2} \big) ds\\
					& \le 2 \big(\|S_m\|_{\Ls(L^p(\bO))}^{2p} +1\big)\sum_{i=1}^M \int_0^t  \| \Nn_i(u(s)) \|_{L^p(\bO)}^2  \|u(s) \|_{L^{p}(\bO)}^{2p-2} ds\\
					& \le 4 \big(\|S_m\|_{\Ls(L^p(\bO))}^{2p} +1\big) \bigg(T\sum_{i=1}^M \|f_i\|_{L^p(\bO)}^2 + \sum_{i=1}^M \|f_i\|_{L^2(\bO)}^2\int_0^t \|u(s) \|_{L^{p}(\bO)}^{2p} ds \bigg)\\
					& \le 4 T\big(\|S_m\|_{\Ls(L^p(\bO))}^{2p} +1\big)\bigg(\sum_{i=1}^M \|f_i\|_{L^p(\bO)}^2 + \sum_{i=1}^M \|f_i\|_{L^2(\bO)}^2 \|u\|_{L^\infty(0,T; L^{p}(\bO))}^{2p} \bigg)< \infty,
				\end{align}
				where we have used the bound \eqref{B-L^p-bound}.
				Therefore, by Lebesgue Dominated Convergence Theorem, we can pass to the limit $m\to \infty$ inside the following integral
				\begin{align}
					& \sum_{i=1}^M \int_0^{t} \bigg( \int_{\bO} \big( |u_m(s)|^{p-2}u_m(s)\Nn_i(u(s)) - |u(s)|^{p-2}u(s)\Nn_i(u(s)) \big) dx\bigg)^2 ds\\
					& = \sum_{i=1}^M \int_0^{t} \bigg( \int_{\bO} \big( |u_m(s)|^{p-2}u_m(s) S_m\Nn_i(u(s)) - |u(s)|^{p-2}u(s) S_m\Nn_i(u(s))\\
					& \qquad\qquad\quad\qquad + |u(s)|^{p-2}u(s) S_m\Nn_i(u(s)) - |u(s)|^{p-2}u(s)\Nn_i(u(s)) \big) dx\bigg)^2 ds\\
					& \le 2 \sum_{i=1}^M \bigg(\int_0^t\| S_m\Nn_i(u(s)) \|_{L^p(\bO)}^2 ds \bigg) \bigg(\int_0^t \||u_m(s)|^{p-2}u_m(s) - |u(s)|^{p-2}u(s)\|_{L^{\frac{p}{p-1}}(\bO)}^2 ds\bigg)\\
					&\quad + 4 \|u\|_{L^{\infty}(0,T; L^{p}(\bO))}^{2p-2} \sum_{i=1}^M \int_0^t \big(\|S_m f_i - f_i \|_{L^p(\bO)}^2 + \|(f_i, u)(u_m - u) \|_{L^p(\bO)}^2\big) ds\\
					& \le 4CT \|S_m\|_{\Ls(L^p(\bO))}^2 T\sum_{i=1}^M \big( \|f_i\|_{L^p(\bO)}^2 + \|f_i\|_{L^2(\bO)}^2\|u\|_{L^\infty(0,T; L^{p}(\bO))}^2 \big)\\
					& \qquad\times \bigg(\int_0^t \||u_m(t)|^{p-2}u_m(t) - |u(t)|^{p-2}u(t)\|_{L^{\frac{p}{p-1}}(\bO)}^2 dt\bigg)\\
					&\quad + 4 \|u\|_{L^{\infty}(0,T; L^{p}(\bO))}^{2p-2}\\
					&\qquad \times \sum_{i=1}^M \bigg(T\|S_m f_i - f_i \|_{L^p(\bO)}^2 + C\|f_i\|_{L^2(\bO)}^2 \|u\|_{L^{\infty}(0,T; L^{2}(\bO))}^2 \int_0^t \|u_m(s) - u(s)\|_{L^p(\bO)}^2 ds \bigg)\\
					& \to 0\ \text{ as }\ m \to \infty.
				\end{align}
				Together with the convergences \eqref{eqn-L^p-conv} and \eqref{eqn-nonlin-conv-L^2}, we thus infer that as $m \to \infty$
				\begin{equation}
					\sum_{i=1}^M \int_0^{t} \bigg( \int_{\bO} \big( |u_m(s)|^{p-2}u_m(s) S_m \Nn_i(u(s)) - |u(s)|^{p-2}u(s)\Nn_i(u(s)) \big) dx\bigg)^2 ds\to 0.
				\end{equation}
				Thus, for each $t\in [0,T]$, \eqref{eqn-Ito-L^p} holds, i.e.,
				\begin{align}
					& \|u(t)\|_{L^p(\bO)}^p+p(p-1)\int_0^t\||u(s)|^{\frac{p-2}{2}}\nabla u(s)\|_{L^2(\bO)}^2ds+p\int_0^t\|u(s)\|_{L^{2p-2}(\bO)}^{2p-2}ds\\
					& = \|u_0\|_{L^p(\bO)}^p ds + p \int_0^{t} \big( \norm{\nabla u(s)}_{L^2(\bO)}^2 + \norm{u(s)}_{L^p(\bO)}^p\big) \|u(s)\|_{L^p(\bO)}^pds\\
					& \quad + \frac{p}{2}\sum_{i=1}^M\int_0^{t}(|u(s)|^{p-2}u(s),\kappa_i(u(s)))ds+ \frac{p(p-1)}{2}\sum_{i=1}^M \int_0^{t}\||u(s)|^{\frac{p-2}{2}}\Nn_i(u(s))\|_{L^2(\bO)}^2ds\\
					& \quad +p\sum_{i=1}^M \int_0^{t}(|u(s)|^{p-2}u(s),\Nn_i(u(s)))dW_i(s),\ \wtilde{\Pr}-\text{a.s}.\label{eqn-ito-3}
				\end{align}
				It completes the proof of the required result.
			\end{proof}			
			
			\noindent
			\textbf{Step VII.} We now turn our attention to deriving a classical Hilbertian It\^o formula in the framework of \eqref{eqn-main-prob-Ito} when the Hilbert space is $H_0^1(\bO)$. We also show that the weak solution obtained in Proposition \ref{Prop-pre-martingale} has almost everywhere continuous trajectories in $\Vp$.

			\subsubsection{Weak solution has continuous trajectories}
			First, we prove Lemma \ref{Lem-Cont}, to do so, we first demonstrate the It\^o formula for the $H_0^1(\bO)-$norm of the weak solution to the SPDE \eqref{eqn-main-prob-Ito}, then by utilizing the $L^p-$It\^o formula proved in Lemma \ref{Lem-Ito} this subsection is concluded. 
			Our ultimate goal is to prove Lemma \ref{Lem-Cont}, in other words, we show that the paths of the weak solution $u$ obtained in Proposition \ref{Prop-pre-martingale} are not merely weakly continuous but, in fact, strongly continuous; i.e., 
			$$u(\cdot, \omega) \in C([0,T];\Vp)\cap L^2(0,T;D(A)) {\cap L^{2p-2}(0,T;L^{2p-2}(\bO))\ \text{ for  a.e. }\ \omega \in \wtilde{\Omega}.}$$
			
			\begin{remark}
				The concept of a ``stochastic energy equality'', i.e., an It\^o formula applied to the square of the $H-$norm of solutions, where $H$ is a Hilbert space, was first discussed by Pardoux \cite{EP-75}. It was later generalized by Gy\"ongy and \v{S}i\v{s}ka \cite{IG+DS_2017} to the setting of finite intersections of Banach spaces and has been instrumental in proving existence and uniqueness results for SPDEs.
			\end{remark}

			\begin{remark}\label{Rmk-Gyongi+Krylov}
				Let $V$ be a separable Banach space with its dual $V^{\prime}$ and  $H$ be a Hilbert space such that $V\embed H\embed V^{\prime}$ with continuous and dense injections. Consider a $V^{\prime} - $valued semi-martingale $y$ of the form
				\begin{align}\label{eqn-000}
					y(t)=\int_0^t v^\prime(s)ds+h(t)
				\end{align}
				on a complete probability space $(\Omega,\Fn,\Pr)$ endowed with a filtration $\{\Fn_t\}_{t\geq 0}$, where $v^\prime$ is a $V^\prime-$valued progressively measurable process and $h$ is an $H-$valued locally square integrable martingale. Then, it has been shown in \cite[Theorem 1]{IG+NVK-82} that if $y=v$ (up to a $d\Pr\otimes dt$ null-set) for a $V-$valued progressively measurable process $v$, and that $\|v\|_{V}$, $\|v^\prime\|_{V^\prime}$ and $\|v\|_{V}\cdot \|v^\prime\|_{V^\prime}$ are $\Pr-$a.s. locally integrable with respect to $dt$, then up to	indistinguishability, $y$ is an $H-$valued adapted continuous process and the It\^o Lemma is valid for $\|y\|_{H}^2$.
			\end{remark}
			
			Motivated by the above remark, we first prove the first part of Lemma \ref{Lem-Cont}, while the second part follows from Lemma \ref{Lem-Ito}.
			
			\begin{proof}[Proof of Lemma \ref{Lem-Cont}]
				{Let us choose and fix $u_0\in \Vp\cap \bM$}, where $\Vp$ is defined in $\eqref{Def-Vp}$. Suppose $u$ is a weak solution of the problem \eqref{eqn-main-prob-Ito}, {guaranteed by Proposition \ref{Prop-pre-martingale}}, so that $$u\in C_w([0,T]; \Vp)\cap L^2(0,T;D(A))\cap L^{2p-2}(0,T;L^{2p-2}(\bO))\ \wtilde{\Pr}-\text{a.s.}$$
				Moreover, $u(t)\in\bM$ for all $t\in[0,T]$ and $u$ satisfies \eqref{eqn-ener-1}. Then, the proof is an consequence of Pardoux's result \cite[Theorem 1]{EP-75} and the It\^o Lemma \ref{Lem-Ito}.
				
				\vspace{2mm}
				\noindent
				\textbf{1.} {Due to the availability} of the energy estimate \eqref{eqn-ener-1}, we observe that each term on the right-hand side of \eqref{eqn-strong-form} is well defined. Since $u$ is a weak solution of \eqref{eqn-main-prob-Ito}, for every $\psi\in L^2(\bO)$ and $t\in[0, T]$, it satisfies the equality \eqref{eqn-strong-form}. Note that the equation  \eqref{eqn-strong-form} holds true for every $\psi\in C_c^{\infty}(\bO)$ and hence \eqref{eqn-strong-form-1} holds in the distribution sense. Since $C_c^{\infty}(\bO)$ is dense in $L^2(\bO)$, so that the equality \eqref{eqn-strong-form-1}  holds true almost everywhere.
				
				{In order to apply the Gy\"ongy $\&$ Krylov's result stated in the Remark \ref{Rmk-Gyongi+Krylov},} we work with the Gelfand triple $D(A)\embed H_0^1(\bO)\embed L^2(\bO)$, where the embeddings are continuous and dense, respectively. 
				For $t\in [0,T]$, we choose $y(t) := u(t)$  and
				\begin{equation*}
					v^\prime(t) := \Delta u(t) -|u(t)|^{p-2}u(t) + \big( \norm{\nabla u(t)}_{L^2(\bO)}^2 + \norm{u(t)}_{L^p(\bO)}^p \big) u(t) + \frac{1}{2}\sum_{i=1}^M\kappa_i(u(t)),
				\end{equation*}
		and $h(t)=\sum_{i=1}^M\int_0^t\Nn_i(u(s))dW_i(s)$	in \eqref{eqn-000}. In order to show the It\^o lemma for $\|u(\cdot)\|_{H}^2$, we are left to show that $\|v\|_{L^2(\bO)}$, $\|v^\prime\|_{L^2(\bO)}$ and $\|v\|_{L^2(\bO)}\cdot \|v^\prime\|_{L^2(\bO)}$ are $\wtilde{\Pr} - $a.s. locally integrable.
				Since $u\in L^2(0,T;D(A))$, $\wtilde{\Pr} - $a.s., it follows that  $\|v\|_{D(A)}:=\|\Delta u\|_{L^2(\bO)}$ is $\wtilde{\Pr} - $a.s. locally integrable. Similarly, from the fact that $u\in L^{\infty}(0,T;\Vp)\cap L^2(0,T;D(A))\cap L^{2p-2}(0,T;L^{2p-2}(\bO))$, $\wtilde{\Pr} - \text{a.s.}$, $\|u(t)\|_{L^2(\bO)}=1$ for all $t\in[0,T]$, the estimate \eqref{eqn-kappa-est} for $p=2$ and the H\"older inequality, we infer
				\begin{align*}
					\int_0^T\|v^\prime(t)\|_{L^2(\bO)}dt
					& \leq T^{\frac{1}{2}}\bigg( \int_0^T\|\Delta u(t)\|_{L^2(\bO)}^2dt+ \int_0^T\|u(t)\|_{L^{2p-2}(\bO)}^{2p-2}dt\\
					& \quad +T\sup_{t\in[0,T]}\big( \norm{\nabla u(t)}_{L^2(\bO)}^2 + \norm{u(t)}_{L^p(\bO)}^p \big) + 4 T \sum_{i=1}^M\|f_i\|_{L^2(\bO)}^2\bigg)^{\frac{1}{2}}\\
					&<\infty, \ \wtilde{\Pr} - \text{a.s.},
				\end{align*}
				Similarly, it follows that
				\begin{align*}
					\int_0^T\|v^\prime(t)\|_{L^2(\bO)}^2dt 
					& \leq C\bigg( \int_0^T\|\Delta u(t)\|_{L^2(\bO)}^2dt+ \int_0^T\|u(t)\|_{L^{2p-2}(\bO)}^{2p-2}dt \\
					& \quad + \int_0^T\left( \norm{\nabla u(t)}_{L^2(\bO)}^2 + \norm{u(t)}_{L^p(\bO)}^p \right)dt+ \sum_{i=1}^M\int_0^T\|\kappa_i(t)\|_{L^2(\bO)}^2dt \bigg) \\
					& \leq C\bigg( \int_0^T\|\Delta u(t)\|_{L^2(\bO)}^2dt+ \int_0^T\|u(t)\|_{L^{2p-2}(\bO)}^{2p-2}dt \\
					& \quad +T\sup_{t\in[0,T]}\left( \norm{\nabla u(t)}_{L^2(\bO)}^2 + \norm{u(t)}_{L^p(\bO)}^p \right) +T\sum_{i=1}^M\|f_i\|_{L^2(\bO)}^2\bigg) \\
					&<\infty, \ \wtilde{\Pr} - \text{a.s.}
				\end{align*}
				which further, by H\"older's inequality, implies that 
				\begin{align*}
					\int_0^T\|v(t)\|_{D(A)}\|v^\prime(t)\|_{L^2(\bO)}dt
					& \leq \bigg( \int_0^T \|\Delta u(t)\|_{L^2(\bO)}^2 dt \bigg)^{\frac{1}{2}} \bigg( \int_0^T \|v^\prime(t)\|_{L^2(\bO)}^2 dt\bigg)^{\frac{1}{2}}\\
					& < \infty,\   \wtilde{\Pr} - \text{a.s.},
				\end{align*}
				so that $ \|v\|_{D(A)}\cdot\|v^\prime\|_{L^2(\bO)}$ is locally integrable. 
				
				\vspace{2mm}
				\noindent
				\textbf{2.} Therefore, $u\in C([0,T];H_0^1(\bO)),$ $\wtilde{\Pr} - $a.s. and the following It\^o formula is valid $\wtilde{\Pr} - $a.s., for all $t\in[0,T]$:
				\begin{align*}
					\|u(t)\|_{H_0^1(\bO)}^2 & = \|u(0)\|_{H_0^1(\bO)}^2+2\int_0^{t}\big(-\Delta u(s),\Delta u(s) -|u(s)|^{p-2}u(s)\big)ds\\
					& \quad +2\int_0^{t}\big(-\Delta u(s), \big( \norm{\nabla u(s)}_{L^2(\bO)}^2 + \norm{u(s)}_{L^p(\bO)}^p \big) u(s)\big)ds\\
					& \quad + \sum_{i=1}^M\int_0^{t}(-\Delta u(s),\kappa_i(u(s)))ds+ \sum_{i=1}^M \int_0^{t}(-\Delta \Nn_i(u(s)),\Nn_i(u(s)))ds\\
					& \quad +2\sum_{i=1}^M \int_0^{t}(-\Delta u(s),\Nn_i(u(s)))dW_i(s).
				\end{align*}
				Upon simplifying, we get
				\begin{align*}
					& \|u(t)\|_{H_0^1(\bO)}^2+2\int_0^t\|\Delta u(s)\|_{L^2(\bO)}^2ds+2(p-1)\int_0^t\||u(s)|^{\frac{p-2}{2}}\nabla u(s)\|_{L^2(\bO)}^2ds\\
					& = \|u_0\|_{H_0^1(\bO)}^2 + 2\int_0^{t} \big(\norm{\nabla u(s)}_{L^2(\bO)}^2 + \norm{u(s)}_{L^p(\bO)}^p \big) \|\nabla u(s)\|_{L^2(\bO)}^2ds\\
					& \quad + \sum_{i=1}^M\int_0^{t}(\nabla u(s),\nabla\kappa_i(u(s)))ds+ \sum_{i=1}^M \int_0^{t}\| \Nn_i(u(s))\|_{L^2(\bO)}^2ds\\
					& \quad + 2\sum_{i=1}^M \int_0^{t}(\nabla u(s),\nabla \Nn_i(u(s)))dW_i(s),
				\end{align*}
				so that $u\in C([0,T];H_0^1(\bO)),$ for $\wtilde{\Pr}$ a.e. paths.
				
				On the other hand, by utilizing the It\^o Lemma \ref{Lem-Ito} for $L^p-$norm, we finally deduce that 
				\begin{equation*}
					u\in C([0,T];L^p(\bO)),\ \wtilde{\Pr} - \text{a.s.}.
				\end{equation*}
				Hence $u\in C([0,T];\Vp)$ $\wtilde{\Pr}-$a.s. This concludes the proof.
			\end{proof}
			
			\noindent
			\textbf{Step VIII.} We are now in a position to prove Theorem \ref{Thm-main} using the ingredients collected above.
			\begin{proof}[Proof of Theorem \ref{Thm-main}]
				By Proposition \ref{Prop-pre-martingale}, there exists a weak solution $(\wtilde{\Omega}, \wtilde{\Fn}, \{\wtilde{\Fn}\}_{t\ge0}, \wtilde{\Pr},\break \wtilde{W}, \wtilde{u})$ which solves \eqref{eqn-strong-form} and, with the help of Lemma \ref{Lem-Ito}, Lemma \ref{Lem-Cont} implies that 
				\begin{equation}
					\wtilde{u}(\cdot,\omega)\in C([0,T]; \Vp)\cap L^2(0,T;D(A)) {\cap L^{2p-2}(0,T;L^{2p-2}(\bO))}
				\end{equation} 
				Hence, there exists a martingale solution in the sense of Definition \ref{Def-Martingale}. This completes the proof.
			\end{proof}

			
			\section{Uniqueness and strong solution}\label{Sec-Path+Strong}
			In this section, we establish that the pathwise uniqueness of martingale solutions to the problem \eqref{eqn-main-prob-Ito}, and subsequently apply the Yamada--Watanabe result to deduce the existence of a strong solution and their uniqueness in law.

			\subsection{Pathwise uniqueness}
			Next result, by using Schmalfuss’s approach \cite{BS-97} based on the Itô formula applied to a suitable function, shows that the martingale solutions of \eqref{eqn-main-prob-Ito}  are pathwise unique.
			
			\begin{lemma}\label{Lem-unique}
				Let $T>0$ be fixed and $u_0\in \Vp\cap\bM$ be given. If $u_1$, $u_2$ are two martingale solutions of the SPDE \eqref{eqn-main-prob-Ito} defined on the same filtered probability space $(\wtilde{\Omega},\wtilde{\Fn},\wtilde{\Fb},\wtilde{\Pr})$ {and with the same initial data $u_0$}, then $\wtilde{\Pr} - $a.s. $u_1(t) =u_2(t)$ for every $t\in [0, T].$
			\end{lemma}
			
			\begin{proof}
				We introduce the notation $w$ as the difference of the two solutions $u_1$ and $u_2$, i.e., $w=u_1 - u_2$. Then, $w$ satisfies the following SPDE for $t\in[0, T]$ in $L^2(\bO)$:
				\begin{equation}
					\left\{
					\begin{aligned}
						d w(t)& = \big(\Delta w(t) + (|u_1(t)|^{p-2}u_1(t) -|u_2(t)|^{p-2}u_2(t)) \big) dt \\
						& \quad + \big( \big( \norm{\nabla u_1(t)}_{L^2(\bO)}^2 + \norm{u_1(t)}_{L^p(\bO)}^p\big) u_1(t) - \big( \norm{\nabla u_2(t)}_{L^2(\bO)}^2 + \norm{u_2(t)}_{L^p(\bO)}^p \big) u_2(t) \big) dt\\
						& \quad + \frac{1}{2}\sum_{i=1}^M(\kappa_i(u_1(t)) -\kappa_i(u_2(t)))dt+ \sum_{i=1}^M(\Nn_i(u_1(t)) -\Nn_i(u_2(t)))dW_i(t),\\
						w(0)& =0.
					\end{aligned}
					\right.
				\end{equation}
				For simplicity, let us define the operator
				\begin{equation}\label{eqn-non-op}
					\Gn(u) := \Delta u -|u|^{p-2}u +  \big(\norm{\nabla u}_{L^2(\bO)}^2  + \norm{u}_{L^p(\bO)}^p\big)u.
				\end{equation}
				Then, the above SPDE becomes
				\begin{equation}
					\left\{
					\begin{aligned}
						d w(t)& = \big(\Gn(u_1(t)) - \Gn(u_2(t))\big)dt + \frac{1}{2}\sum_{i=1}^M(\kappa_i(u_1(t)) -\kappa_i(u_2(t)))dt\\
						& \quad + \sum_{i=1}^M(\Nn_i(u_1(t)) -\Nn_i(u_2(t)))dW_i(t),\\
						w(0)& =0.
					\end{aligned}
					\right.
				\end{equation}
				\noindent
				\textbf{Step I.} Let us define a sequence stopping times
				\begin{equation*}
					\tau_R:= \inf_{t\geq 0}\left\{t:\|u_1(t)\|_{H_0^1(\bO)} + \|u_1(t)\|_{L^p(\bO)} + \|u_2(t)\|_{H_0^1(\bO)} + \|u_2(t)\|_{L^p(\bO)}\geq R\right\}\wedge T,
				\end{equation*}
				where $R\in\N.$ Since
				\begin{equation*}
					\wtilde{\E}\bigg[\sup_{t\in[0,T]} \big( \|u_i(t)\|_{H_0^1(\bO)}^2+ \|u_i(t)\|_{L^p(\bO)}^p\big) \bigg]\leq C, \ \text{ for }\ i=1,2,
				\end{equation*}
				it follows that $\tau_R\to T$ as $R\to\infty$.
				We apply now the infinite dimensional It\^o formula to the function $$\phi(t,v)=e^{-\rho(t)}\|v\|_{L^2(\bO)}^2,\ t\in[0,T],\ v\in \Vp,$$ and to the process $w$, where $\rho(t)$, $t\in[0, T]$ is a real-valued function which will be specified later in the proof.
				Note that for $h,k\in L^2(\bO)$,
				$$\frac{\partial\phi}{\partial t}=-\rho'(t)e^{-\rho(t)}\|v\|_{L^2(\bO)}^2,\ \frac{\partial\phi}{\partial v}(h)=2e^{-\rho(t)}(v,h),\ \frac{\partial^2\phi}{\partial v^2}(h,k)=2(h,k),$$ 
				we obtain for all $t\in[0, T]$, $\wtilde{\Pr} - $a.s.
				\begin{align}
					& e^{-\rho(t\wedge\tau_R)}\|w(t\wedge\tau_R)\|_{L^2(\bO)}^2\\
					& = \int_0^{t\wedge\tau_R} e^{-\rho(s)} \Big( -\rho'(s)\|w(s)\|_{L^2(\bO)}^2 {+} 2 \big( \Gn(u_1(s)) - \Gn(u_2(s)) ,w(s)\big) \Big) ds\\
					& \quad + \sum_{i=1}^M\int_0^{t\wedge\tau_R}e^{-\rho(s)}(\kappa_i(u_1(s)) -\kappa_i(u_2(s)),w(s))ds\\ 
					& \quad + \sum_{i=1}^M\int_0^{t\wedge\tau_R}e^{-\rho(s)}(\Nn_i(u_1(s)) -\Nn_i(u_2(s)),\Nn_i(u_1(s)) -\Nn_i(u_2(s)))ds\\
					& \quad +2\sum_{i=1}^M\int_0^{t\wedge\tau_R}e^{-\rho(s)}(\Nn_i(u_1(s)) -\Nn_i(u_2(s)),w(s))dW_i(s).\label{eqn-unique}
				\end{align}
				\noindent
				\textbf{Step II.} {It is known from \cite[Lemma 3.6, see Step 4]{AB+ZB+MTM-25+} that the operator $\Gn$ is monotone together with the estimate \eqref{eqn-nonlinear-est}, see Proposition \ref{Prop-Mono-Non-lin}, it satisfies the following properties:}
				\begin{align}
					& (\Gn(u_1) -\Gn(u_2),u_1-u_2)\\
					& \leq {- \frac{1}{2}\|\nabla(u_1-u_2)\|_{L^2(\bO)}^2 - \frac{1}{2^{p-2}}\|u_1 - u_2\|_{L^p(\bO)}^p} \\
					& \quad + \bigg( \norm{\nabla u_1}_{L^2(\bO)}^2 + \frac{1}{2}(\norm{\nabla u_1}_{L^2(\bO)} + \norm{\nabla u_2}_{L^2(\bO)})^2 \norm{ u_2}_{L^2(\bO)}^2\\
					& \qquad \quad + \norm{u_1}_{L^p(\bO)}^p + p^2 2^{2p-4} \big(\norm{u_1}_{L^p(\bO)}^p + \norm{u_2}_{L^p(\bO)}^p \big)\norm{u_2}_{L^2(\bO)}^2 \bigg) \norm{u_1-u_2}_{L^2(\bO)}^2.\label{eqn-unique-2}
				\end{align}
				{On the other hand, by} using the Cauchy-Schwarz inequality and the estimate \eqref{eqn-est-3} for $p=2$, we evaluate 
				\begin{align}
					|(\kappa_i(u_1) -\kappa_i(u_2),w)| 
					& \leq\|\kappa_i(u_1) -\kappa_i(u_2)\|_{L^2(\bO)}\|w\|_{L^2(\bO)}\\ 
					& \leq 2\|f_i\|_{L^2(\bO)}^2\left(1+ (\|u_1\|_{L^2(\bO)} + \|u_2\|_{L^2(\bO)})^2\right)\|w\|_{L^2(\bO)}^2.\label{eqn-unique-4}
				\end{align}
				{Similarly, the remaining term can be estimated using the bound in \eqref{eqn-b-differene-2} with $p=2$, as follow:}
				{\begin{align}
						\|\Nn_i(u_1) -\Nn_i(u_2)\|_{L^2(\bO)}^2
						& \leq \|f_i\|_{L^2(\bO)}^2 \left(\|u_1\|_{L^2(\bO)} + \|u_2\|_{L^2(\bO)}\right)^2 \|w\|_{L^2(\bO)}^2.\label{eqn-unique-5}
				\end{align}}
				{Substituting the estimates obtained above} \eqref{eqn-unique-2}-\eqref{eqn-unique-5}  in \eqref{eqn-unique}, we deduce
				\begin{align}
					& e^{-\rho(t\wedge\tau_R)}\|w(t\wedge\tau_R)\|_{L^2(\bO)}^2+ \int_0^{t\wedge\tau_R}e^{-\rho(s)}\|\nabla(u_1(s) -u_2(s))\|_{L^2(\bO)}^2ds\\
					& \quad + {\frac{1}{2^{p-1}}} \int_0^{t\wedge\tau_R}e^{-\rho(s)}\|u_1(s) -u_2(s)\|_{L^p(\bO)}^pds\\
					& \leq -\int_0^{t\wedge\tau_R}e^{-\rho(s)}\rho'(s)\|w(s)\|_{L^2(\bO)}^2ds\\
					& \quad +2\int_0^{t\wedge\tau_R}e^{-\rho(s)}\bigg( \norm{\nabla u_1(s)}_{L^2(\bO)}^2 + \frac{1}{2}(\norm{\nabla u_1(s)}_{L^2(\bO)} + \norm{\nabla u_2(s)}_{L^2(\bO)})^2 \norm{ u_2(s)}_{L^2(\bO)}^2\\
					& \qquad\qquad \qquad\qquad + \norm{u_1(s)}_{L^p(\bO)}^p + p^2 2^{2p-4}\left(\norm{u_1(s)}_{L^p(\bO)}^p + \norm{u_2(s)}_{L^p(\bO)}^p\right)\norm{u_2(s)}_{L^2(\bO)}^2\\
					& \qquad\qquad\qquad\qquad {+ \sum_{i=1}^M \|f_i\|_{L^2(\bO)}^2\big(2+ 3(\|u_1(s)\|_{L^2(\bO)} + \|u_2(s)\|_{L^2(\bO)})^2 \big)} \bigg) \norm{w(s)}_{L^2(\bO)}^2 {ds}\\
					& \quad +2 \sum_{i=1}^M \int_0^{t\wedge\tau_R}e^{-\rho(s)}(\Nn_i(u_1(s)) -\Nn_i(u_2(s)),w(s))dW_i(s).\label{eqn-unique-6}
				\end{align}
				\noindent
				\textbf{Step III.} Let us now choose
				\begin{align}
					\rho(t) & =2\int_0^t\bigg( \norm{\nabla u_1(s)}_{L^2(\bO)}^2 + \frac{1}{2} \big(\norm{\nabla u_1(s)}_{L^2(\bO)} + \norm{\nabla u_2(s)}_{L^2(\bO)} \big)^2 \norm{ u_2(s)}_{L^2(\bO)}^2\\
					& \qquad \qquad + \norm{u_1(s)}_{L^p(\bO)}^p + p^2 2^{2p-4} \big(\norm{u_1(s)}_{L^p(\bO)}^p + \norm{u_2(s)}_{L^p(\bO)}^p \big)\norm{u_2(s)}_{L^2(\bO)}^2 \\
					& \qquad\qquad {+ \sum_{i=1}^M \|f_i\|_{L^2(\bO)}^2\big(2+ 3(\|u_1(s)\|_{L^2(\bO)} + \|u_2(s)\|_{L^2(\bO)})^2 \big)} \bigg) \norm{w(s)}_{L^2(\bO)}^2 {ds}
				\end{align}
				and take expectation in \eqref{eqn-unique-6} and use the fact that
				$$\sum_{i=1}^M\int_0^{t\wedge\tau_R}e^{-\rho(s)}(\Nn_i(u_1(s)) -\Nn_i(u_2(s)),w(s))dW_i(s)$$
				is a martingale, so its expectation is zero, to deduce
				\begin{align*}
					{\E\bigg[e^{-\rho(t\wedge\tau_R)}\|w(t\wedge\tau_R)\|_{L^2(\bO)}^2+ \int_0^{t\wedge\tau_R}e^{-\rho(s)}\Big( \|\nabla w(s)\|_{L^2(\bO)}^2 + {\frac{1}{2^{p}}} \|w(s)\|_{L^p(\bO)}^p \Big) ds\bigg]\leq 0.}
				\end{align*}
				In particular, we have
				\begin{equation}
					\sup_{t\in[0,T]}\E\big[e^{-\rho(t\wedge\tau_R)}\|w(t\wedge\tau_R)\|_{L^2(\bO)}^2\big]= 0.\label{eqn-unique-7}
				\end{equation}
				Since $u_1$ and $u_2$ are the martingale solutions of \eqref{eqn-main-prob-Ito} satisfying the energy estimate \eqref{eqn-ener-1} together with Lemma \ref{Lem-energy}, it follows immediately that $\rho$ is well defined for all $t\in[0,T]$. Furthermore, we have 
				\begin{equation*}
					\lim_{R\to\infty}\tau_R=T\ \  \wtilde{\Pr} - a.s. \ \text{ and }\ \wtilde{\E}[\rho(t)]<\infty.
				\end{equation*}
				Therefore, from \eqref{eqn-unique-7}, we deduce that $w(t) =0$ for every $t\in [0, T]$, $\wtilde{\Pr} - $a.s., which concludes the proof of pathwise uniqueness.
			\end{proof}
			
			\subsection{The strong solution}
			In this subsection, we first introduce the notion of uniqueness in law, followed by an application of the Yamada-Watanabe result, see Theorem \ref{Thm-YW}, which shows that any pathwise unique martingale solution is, in fact, a strong solution to the problem \eqref{eqn-main-prob-Ito}, in the sense of Definition \ref{Def-Strong-Soln}.
			
			\begin{definition}[{\cite[Definition 6.2]{ZB+GD-21}}]
				Let $(\Omega^i,\Fn^i,\Fb^i,\Pr^i,W^i,u^i),$ for $ i=1, 2$ be the martingale solutions of the problem \eqref{eqn-main-prob-Ito} with $u^i(0) =u_0$, $i=1, 2$. Then, we say that the solutions are \emph{unique in law} if
				\begin{equation*}
					\bL_{\Pr_1}(u^1)=\bL_{\Pr_2}(u^2) \ \text{ on }\  C([0,T];\Vp)\cap L^2(0,T;D(A))\cap L^{2p-2}(0,T;L^{2p-2}(\bO)),
				\end{equation*}
				where $\bL_{\Pr_i}(u^i),$ $ i=1, 2$ represents the law of the random variable $u_i$ which takes values in $C([0,T];\Vp)\cap L^2(0,T;D(A))\cap L^{2p-2}(0,T;L^{2p-2}(\bO))$.
			\end{definition}

			Hence, an immediate consequence of the Yamada--Watanabe Theorem \ref{Thm-YW} is the following result.
			
			\begin{corollary}[{\cite[Corrolary 6.3]{ZB+GD-21}}]\label{cor-weak-strong}
				Let $T>0$ be fixed and $u_0\in \Vp\cap\bM$ be given.
				\begin{itemize}
					\item[(a)] Then, there exists a pathwise unique strong solution of the problem \eqref{eqn-main-prob-Ito}.
					\item[(b)] Moreover, if  $(\Omega,\Fn,\Fb,\Pr,W,u)$ is a strong solution of \eqref{eqn-main-prob-Ito}, then for $\Pr$-almost all $\omega\in\Omega$, the trajectory $u(\cdot,\omega)$ is equal almost everywhere to a continuous $\Vp-$valued function defined on $[0, T].$
					\item[(c)] The martingale solution of  the problem \eqref{eqn-main-prob-Ito} is unique in law.
				\end{itemize}
			\end{corollary}
			
			\begin{proof}[Proof of Corollary \ref{cor-weak-strong}]
				Assertions (a) and (b) are immediate consequences of Theorem \ref{Thm-main} and Lemma \ref{Lem-unique}, respectively; namely, the existence of a $\Vp-$valued martingale solution whose paths $u(\cdot,\omega)$ are a.e. equal to a continuous function defined on $[0, T]$, and pathwise uniqueness.
				By an application of Theorem \ref{Thm-YW}, assertion (c) follows immediately, which completes the proof.
			\end{proof}
			
			We end this article by presenting the final result, which follows as a consequence of the preceding results and ultimately establishes Theorem~\ref{Thm-uniqueness}.
			
			\begin{proof}[Proof of Theorem \ref{Thm-uniqueness}]
				Thanks to Theorem \ref{Thm-main},  Lemma \ref{Lem-unique} and Corollary \ref{cor-weak-strong}, by which the proof follows immediately.
			\end{proof}


			\appendix
			\renewcommand{\thesection}{\Alph{section}}
			\numberwithin{equation}{section}
			
			\section{Yamada-Watanabe Theorem}\label{Sec-Append}
			The objective of this section is to provide one of the important and well-known result on the strong solution in SPDEs, i.e., the Yamada-Watanabe result by Ond\v{r}ej\'at \cite{MO-04}.

			\begin{theorem}[{\cite[Theorem 2 and Theorem 11]{MO-04}}]\label{Thm-YW}
				Suppose $(\Omega,\Fn,\Fb,\Pr,W,u)$ be a pathwise unique solution of the problem \eqref{eqn-main-prob-Ito}. Then,
				\begin{itemize}
					\item[(i)] The equation \eqref{eqn-main-prob-Ito} has a strong solution in the sense of Definition \ref{Def-Strong-Soln}.
					\item[(ii)] The uniqueness in law hold for \eqref{eqn-main-prob-Ito}.
				\end{itemize}
			\end{theorem}

			\begin{remark}
				We would like to thank M. Veraar and E. Theewis for bringing to our attention the article \cite{ET-25+} on the Yamada--Watanabe result for analytically strong solutions in UMD spaces and for equations without a semigroup structure. Since our martingale solution is strongly continuous, see Definition \ref{Def-Martingale}, the result of Ond\v{r}ej\'at \cite{MO-04} is sufficient in our setting to establish the existence of a pathwise unique strong solution. Therefore, while the result in \cite{ET-25+} is very interesting, valuable, and we believe it will be useful in our future work, we can rely on \cite{MO-04} in the present scenario.
			\end{remark}

			\medskip\noindent
			\textbf{Acknowledgments:} AB is grateful to University Grants Commission (UGC), Government of India for financial assistance (File No.: 368/2022/211610061684).  MTM would  like to thank the Department of Science and Technology (DST) Science $\&$ Engineering Research Board (SERB), India for a MATRICS grant (MTR/2021/000066). ZB and MTM acknowledge the support of the London Mathematical Society (Scheme 5), Grant Ref. 52426, for supporting MTM’s visit to the University of York, where part of this work was discussed.
			
			\medskip\noindent	\textbf{Declarations:} 
			
			\noindent 	\textbf{Ethical Approval:}   Not applicable 
			
			\noindent  \textbf{Competing interests: } The authors declare no competing interests. 
			
			\noindent  \textbf{Conflict of interest: }On behalf of all authors, the corresponding author states that there is no conflict of interest.
			
			\noindent 	\textbf{Authors' contributions:} All authors have contributed equally. 
			
			\noindent 	\textbf{Availability of data and materials:} Not applicable.

			\bibliographystyle{plain}

\begin{thebibliography}{}

\end{thebibliography}


\begin{thebibliography}{10}

\bibitem{DA-78}
D.~Aldous.
\newblock Stopping times and tightness.
\newblock {\em Ann. Probability}, 6(2):335--340, 1978.

\bibitem{PA+PC+BS-24}
P.~Antonelli, P.~Cannarsa, and B.~Shakarov.
\newblock Existence and asymptotic behavior for {$L^2$}-norm preserving
  nonlinear heat equations.
\newblock {\em Calc. Var. Partial Differential Equations}, 63(4):Paper No. 108,
  31, 2024.

\bibitem{AB-06}
A.~Badrikian.
\newblock {\em S{\'e}minaire sur les fonctions al{\'e}atoires lin{\'e}aires et
  les mesures cylindriques}, volume 139.
\newblock Springer, 2006.

\bibitem{AB+ZB+MTM-25+}
A.~Bawalia, Z.~Brze{\'z}niak, and M.T. Mohan.
\newblock Global well-posedness and asymptotic analysis of a nonlinear heat
  equation with constraints of finite codimension.
\newblock {\em \href{https://arxiv.org/pdf/2507.00160}{arXiv:2507.00160}},
  2025.

\bibitem{AB+ZB+MTM+PR-25+}
A.~Bawalia, Z.~Brze{\'z}niak, M.T. Mohan, and P.~Rybka.
\newblock Well-posedness and the Łojasiewicz-{S}imon inequality in the
  asymptotic analysis of a nonlinear heat equation with constraints of finite
  codimension.
\newblock {\em \href{https://arxiv.org/pdf/2512.21158}{arXiv:2512.21158}},
  2025.

\bibitem{ZB+SC-25}
Z.~Brze{\'z}niak and S.~Cerrai.
\newblock Stochastic wave equations with constraints: Well-posedness and
  {S}moluchowski–{K}ramers diffusion approximation.
\newblock {\em Commun. Math. Phys.}, 406(9):223, 2025.

\bibitem{ZB+GD-21}
Z.~Brze\'zniak and G.~Dhariwal.
\newblock Stochastic constrained {N}avier-{S}tokes equations on {$\Bbb{T}^2$}.
\newblock {\em J. Differential Equations}, 285:128--174, 2021.

\bibitem{ZB+GD+MM-18}
Z.~Brze\'zniak, G.~Dhariwal, and M.~Mariani.
\newblock 2{D} constrained {N}avier-{S}tokes equations.
\newblock {\em J. Differential Equations}, 264(4):2833--2864, 2018.

\bibitem{ZB+BF+MZ-24}
Z.~Brze\'zniak, B.~Ferrario, and M.~Zanella.
\newblock Invariant measures for a stochastic nonlinear and damped 2{D}
  {S}chr\"odinger equation.
\newblock {\em Nonlinearity}, 37(1):Paper No. 015001, 66, 2024.

\bibitem{ZB+FH+UM-20}
Z.~Brze\'zniak, F.~Hornung, and U.~Manna.
\newblock Weak martingale solutions for the stochastic nonlinear
  {S}chr\"odinger equation driven by pure jump noise.
\newblock {\em Stoch. Partial Differ. Equ. Anal. Comput.}, 8(1):1--53, 2020.

\bibitem{ZB+FH+LW-19}
Z.~Brze\'zniak, F.~Hornung, and L.~Weis.
\newblock Martingale solutions for the stochastic nonlinear {S}chr\"odinger
  equation in the energy space.
\newblock {\em Probab. Theory Related Fields}, 174(3-4):1273--1338, 2019.

\bibitem{ZB+JH-20}
Z.~Brze\'zniak and J.~Hussain.
\newblock Global solution of nonlinear stochastic heat equation with solutions
  in a {H}ilbert manifold.
\newblock {\em Stoch. Dyn.}, 20(6):2040012, 29, 2020.

\bibitem{ZB+JH-24}
Z.~Brze\'zniak and J.~Hussain.
\newblock Global solution of nonlinear heat equation with solutions in a
  {H}ilbert manifold.
\newblock {\em Nonlinear Anal.}, 242:Paper No. 113505, 17, 2024.

\bibitem{ZB+JH-26}
Z.~Brze\'{z}niak and J.~Hussain.
\newblock Large deviations for the stochastic nonlinear heat equation on a
  {H}ilbert manifold.
\newblock {\em Potential Anal.}, 64(2):Paper No. 36, 30, 2026.

\bibitem{ZB+MO-07}
Z.~Brze\'zniak and M.~Ondrej\'at.
\newblock Strong solutions to stochastic wave equations with values in
  {R}iemannian manifolds.
\newblock {\em J. Funct. Anal.}, 253(2):449--481, 2007.

\bibitem{ZB+SP-01}
Z.~Brze\'zniak and S.~Peszat.
\newblock Stochastic two dimensional {E}uler equations.
\newblock {\em Ann. Probab.}, 29(4):1796--1832, 2001.

\bibitem{ZB+TZ-99}
Z.~Brze\'zniak and T.~Zastawniak.
\newblock {\em Basic stochastic processes}.
\newblock Springer-Verlag London, Ltd., London, 1999.

\bibitem{LC+FL-09}
L.~Caffarelli and F.~Lin.
\newblock Nonlocal heat flows preserving the {$L^2$} energy.
\newblock {\em Discrete Contin. Dyn. Syst.}, 23(1-2):49--64, 2009.

\bibitem{SC+MX-25}
S.~Cerrai and M.~Xie.
\newblock The small-mass limit for some constrained wave equations with
  nonlinear conservative noise.
\newblock {\em Electron. J. Probab.}, 30:Paper No. 25, 27, 2025.

\bibitem{RF+HK+HS-05}
R.~Farwig, H.~Kozono, and H.~Sohr.
\newblock An {$L^q$}-approach to {S}tokes and {N}avier-{S}tokes equations in
  general domains.
\newblock {\em Acta Math.}, 195:21--53, 2005.

\bibitem{DG+DC-24}
D.~Goodair and D.~Crisan.
\newblock {\em Stochastic calculus in infinite dimensions and {SPDE}s}.
\newblock Springer, Cham, 2024.

\bibitem{IG+NVK-82}
I.~Gy\"ongy and N.~V. Krylov.
\newblock On stochastics equations with respect to semimartingales. {II}.
  {I}t\^o{} formula in {B}anach spaces.
\newblock {\em Stochastics}, 6(3-4):153--173, 1981/82.

\bibitem{IG+DS_2017}
I.~Gy\"ongy and D.~Šiška.
\newblock It\^o{} formula for processes taking values in intersection of
  finitely many {B}anach spaces.
\newblock {\em Stoch. Partial Differ. Equ. Anal. Comput.}, 5(3):428--455, 2017.

\bibitem{FH-18}
F.~Hornung.
\newblock The nonlinear stochastic {S}chr\"odinger equation via stochastic
  {S}trichartz estimates.
\newblock {\em J. Evol. Equ.}, 18(3):1085--1114, 2018.

\bibitem{LH-18}
L.~Hornung.
\newblock Strong solutions to a nonlinear stochastic {M}axwell equation with a
  retarded material law.
\newblock {\em J. Evol. Equ.}, 18(3):1427--1469, 2018.

\bibitem{JH-23}
J.~Hussain.
\newblock Faedo-{G}alerkin approximations for nonlinear heat equation on
  {H}ilbert manifold.
\newblock {\em Carpathian J. Math.}, 39(3):667--682, 2023.

\bibitem{JH+FA+AS-24+}
J.~Hussain, A.~Fatah, and S.~Ahmed.
\newblock Existence of martingale solutions to stochastic constrained heat
  equation.
\newblock {\em \href{https://arxiv.org/pdf/2411.04631}{arXiv:2411.04631}},
  2024.

\bibitem{AI-86}
A.~Ichikawa.
\newblock Some inequalities for martingales and stochastic convolutions.
\newblock {\em Stochastic Analysis and Applications}, 4(3):329--339, 1986.

\bibitem{AJ-97}
A.~Jakubowski.
\newblock The almost sure {S}korokhod representation for subsequences in
  nonmetric spaces.
\newblock {\em Teor. Veroyatnost. i Primenen.}, 42(1):209--216, 1997.

\bibitem{NVK-10}
N.~V. Krylov.
\newblock It\^{o}'s formula for the {$L_p$}-norm of stochastic {$W^1_p$}-valued
  processes.
\newblock {\em Probab. Theory Related Fields}, 147(3-4):583--605, 2010.

\bibitem{JLL-69}
J.-L. Lions.
\newblock {\em Quelques m\'ethodes de r\'esolution des probl\`emes aux limites
  non lin\'eaires}.
\newblock Dunod, Paris; Gauthier-Villars, Paris, 1969.

\bibitem{LM+LC-09}
L.~Ma and L.~Cheng.
\newblock Non-local heat flows and gradient estimates on closed manifolds.
\newblock {\em J. Evol. Equ.}, 9(4):787--807, 2009.

\bibitem{LM+LC-13}
L.~Ma and L.~Cheng.
\newblock Global solutions to norm-preserving non-local flows of porous media
  type.
\newblock {\em Proc. Roy. Soc. Edinburgh Sect. A}, 143(4):871--880, 2013.

\bibitem{MM-88}
M.~M\'etivier.
\newblock {\em Stochastic partial differential equations in
  infinite-dimensional spaces}.
\newblock Scuola Normale Superiore, Pisa, 1988.

\bibitem{MO-04}
M.~Ondrej\'at.
\newblock Uniqueness for stochastic evolution equations in {B}anach spaces.
\newblock {\em Dissertationes Math. (Rozprawy Mat.)}, 426:63, 2004.

\bibitem{EP-75}
E.~Pardoux.
\newblock \'equations aux d\'eriv\'ees partielles stochastiques de type
  monotone.
\newblock In {\em S\'eminaire sur les \'Equations aux {D}\'eriv\'ees
  {P}artielles (1974--1975), {III}}, pages Exp. No. 2, 10. Coll\`ege de France,
  Paris, 1975.

\bibitem{GDP+JZ-14}
D.~G. Prato and J.~Zabczyk.
\newblock {\em Stochastic equations in infinite dimensions}, volume 152.
\newblock Cambridge University Press, Cambridge, second edition, 2014.

\bibitem{PR-06}
R.~Rybka.
\newblock Convergence of a heat flow on a {H}ilbert manifold.
\newblock {\em Proc. Roy. Soc. Edinburgh Sect. A}, 136(4):851--862, 2006.

\bibitem{BS-97}
B.~Schmalfuss.
\newblock Qualitative properties for the stochastic {N}avier-{S}tokes equation.
\newblock {\em Nonlinear Anal.}, 28(9):1545--1563, 1997.

\bibitem{BS-25}
B.~Shakarov.
\newblock Global solutions and asymptotic behavior to a norm-preserving
  non-local parabolic flow.
\newblock {\em Boll. Unione Mat. Ital.}, pages 1--20, 2025.

\bibitem{ET-25+}
E.~Theewis.
\newblock The {Y}amada-{W}atanabe-{E}ngelbert theorem for {SPDE}s in {B}anach
  spaces.
\newblock {\em
  \href{https://arxiv.org/pdf/2502.00189v3}{{arXiv:2502.00189v3}}}, 2025.

\bibitem{MJV+AVF-88}
M.~J. Vishik and A.~V. Fursikov.
\newblock {\em Mathematical problems of statistical hydromechanics}, volume~9.
\newblock Kluwer Academic Publishers Group, Dordrecht, 1988.

\end{thebibliography}

		\end{document}